\documentclass[11pt]{article}
\usepackage{bbm}
\usepackage{amssymb}
\usepackage{amssymb, amsthm, amsmath, amscd}
\usepackage[usenames,dvipsnames]{color}
\usepackage{framed}

\newtheorem{thm}{Theorem}[section]
\newtheorem{lem}[thm]{Lemma}
\newtheorem{prop}[thm]{Proposition}
\newtheorem{cor}[thm]{Corollary}
\theoremstyle{definition}
\newtheorem{NN}[thm]{}
\theoremstyle{definition}
\newtheorem{df}[thm]{Definition}
\theoremstyle{definition}
\newtheorem{rem}[thm]{Remark}
\newtheorem{nota}[thm]{Notation}
\theoremstyle{definition}
\newtheorem{exm}[thm]{Example}

\renewcommand{\phi}{\varphi}

\definecolor{purple}{RGB}{150,10,200} 

\newcommand{\N}{\mathbb{N}}
\newcommand{\Z}{\mathbb{Z}}
\newcommand{\Q}{\mathbb{Q}}
\newcommand{\R}{\mathbb{R}}
\newcommand{\C}{\mathbb{C}}

\numberwithin{equation}{section}

\newcommand{\Aff}{\operatorname{Aff}}

\newcommand{\Tw}{\overline{T(A)}^w}

\newcommand{\hm}{homomorphism}
\newcommand{\dt}{\delta}
\newcommand{\ep}{\varepsilon}

\newcommand{\td}{\tilde}

\newcommand{\andeqn}{\,\,\,{\rm and}\,\,\,}
\newcommand{\rforal}{\,\,\,{\rm for\,\,\,all}\,\,\,}
\newcommand{\CA}{$C^*$-algebra}
\newcommand{\SCA}{$C^*$-subalgebra}

\newcommand{\af}{{\alpha}}
\newcommand{\bt}{{\beta}}

\newcommand{\diag}{{\rm diag}}

\newcommand{\wtd}{\widetilde}
\newcommand{\wh}{\widehat}

\newcommand{\wilog}{without loss of generality}
\newcommand{\Wlog}{Without loss of generality}

\newcommand{\beq}{\begin{eqnarray}}
\newcommand{\eneq}{\end{eqnarray}}
\newcommand{\tforal}{\,\,\,\text{for\,\,\,all}\,\,\,}
\newcommand{\tand}{\,\,\,\text{and}\,\,\,}

\usepackage{amsfonts}
\usepackage{mathrsfs}
\usepackage{textcomp}
\usepackage[all]{xy}

\newcommand{\Her}{\mathrm{Her}}
\newcommand{\Cu}{\mathrm{Cu}}

\newcommand{\cto}{\stackrel{c.}{\to}}
\newcommand{\cdto}{\stackrel{c.}{\searrow}}

\newcommand{\Qw}{\overline{QT(A)}^w}
\newcommand{\Qwy}{\overline{QT(A_1)}^w}

\title{Tracial  oscillation zero and stable rank one}
\author{Xuanlong Fu and Huaxin Lin }
\date{ }

\begin{document}

\maketitle

\begin{abstract}
Let $A$ be a  separable (not necessarily unital) simple \CA\, with strict comparison.
We show that if $A$ has tracial approximate oscillation zero then $A$ has stable rank one
and the canonical map $\Gamma$ from the Cuntz semigroup of $A$ to the corresponding 
lower-semicontinuous affine function space is surjective. 
The converse also holds. 
As a by-product,   we find that a separable simple \CA\, which has almost stable rank one
must have stable rank one, provided  it has strict comparison and the canonical map 
$\Gamma$ is surjective.

\end{abstract}

\section{Introduction}

Let $X$ be a compact metric space and $T$ be a set  of probability Borel measures on $X.$
For each open subset $O$ of $X$, we consider its measure $\mu(O).$
This gives a function $\widehat{O}(\mu)=\mu(O)$ ($\mu\in T$) on $T.$
This function {{is}}  lower-semicontinuous on $T$ if we endow $T$ with  the weak*-topology.
Let  $\af: X\to X$ be  a  
homeomorphism on $X$ and $T$ be the set of $\af$-invariant probability  Borel measures.
One considers the case that there are sufficiently many open sets $O$ 
{{for which $\widehat{O}$ is continuous on $T.$}}
This is certainly the case when the action is uniquely  ergodic. 
The small boundary condition, or the condition of mean dimension zero,  requires that
in any neighborhood $N(x)$ of each point $x\in X,$  there is {a} neighborhood $O(x)\subset N(x)$ 
such that $\widehat{O(x)}$ is continuous.
Let $\omega(\widehat{O})$ be the oscillation of the function $\widehat{O}.$ 
If $\widehat{O}$ is continuous, then $\omega(\widehat{O})=0.$ 

Let $A$ be a \CA\, with tracial state space $T(A).$ 
For each $a\in A_+,$ one defines the rank  function of $a$ 
by $\widehat{[a]}(\tau)=\lim_{n\to\infty}\tau(a^{1/n})$ for $\tau\in T(A).$ 
When $A=M_n,$ i.e., $A$ is the $n\times n$ matrix algebra, 
$\widehat{[a]}$ is just the normalized rank of $a.$  We study the oscillation of the function 
$\widehat{[a]}.$  It is called tracial oscillation of  the element $a.$
This notion of tracial oscillation has been studied in \cite{eglnkk0} and \cite{LincuA}
in connection with the augmented Cuntz semigroups. 
  We introduce the notion of tracial approximate oscillation  zero
for \CA s.  Roughly speaking  a  unital \CA\, $A$ has tracial approximate oscillation zero, 
if each positive element $a$ is approximated (tracially) by elements 
in the hereditary \SCA\, generated by $a$ with 
small tracial oscillation (see Definition \ref{DTos-2}).   If $\af$ is a minimal 
homeomorphism on $X$ and $(X, \af)$  has mean dimension  zero,  it is shown in  \cite{EN}
that the crossed product \CA\, $C(X)\rtimes_\af \Z$ is ${\cal Z}$-stable. As a consequence, 
$C(X)\rtimes_\af \Z$ has tracial approximate  oscillation zero (see Theorem \ref{Teqiv} below).

 The notion of  stable rank was introduced to \CA\, theory by Marc  Rieffel  in \cite{Rff}.
 A unital \CA\, has stable rank one if 
 its invertible elements are dense in $A.$ 
 The notion plays an important role in the study of simple \CA s
 (see some earlier work, for example, \cite{Pstr} and \cite{DNNP}). 
 It is proved by M. R\o rdam (\cite{Ror04JS}) that if $A$ is a unital {{finite}} separable simple 
 ${\cal Z}$-stable \CA, then $A$ has stable rank one. 
 L. Robert in \cite{Rlz} introduced the notion of almost stable rank one,  which is also a very useful notion, 
 and showed that  every stably projectionless ${\cal Z}$-stable simple \CA\, has almost stable rank one.
   A question remains open, however, whether a separable simple \CA\, with almost stable rank one
 actually has stable rank one. 
 
 There is a canonical map $\Gamma$ from the Cuntz semigroup of $A,$ denoted by
 $\Cu(A),$ to ${\rm LAff}_+(\wtd{QT}(A)),$ the set of strictly positive 
 lower semi-continuous affine functions (vanishing at zero) on the cone of 2-quasitraces on $A,$ defined 
 by $\Gamma([a])(\tau)=d_\tau(a)$ (for $\tau\in {\wtd{QT}}(A)$). 
 A question {{posed}} by N. Brown (see the remark after Question 1.1 of \cite{Th})  asked whether 
 this map is surjective, i.e., whether every strictly positive lower semi-continuous affine function 
 on ${\wtd{QT}}(A)$ is a rank function for some positive element in 
 $A\otimes {\cal K}.$   It is of course an important question.  In fact,
 the strict comparison and surjectivity of $\Gamma$ are perhaps 
 equally important when one studies Cuntz semigroups. 
 If we denote by $\Cu(A)_+$ the set of {{purely}} non-compact elements in the Cuntz semigroup 
 of   a separable stably finite simple \CA\, $A,$ then strict comparison is the condition 
 that $\Gamma$  restricted on $\Cu(A)_+$ is injective. If $\Gamma$ is also surjective, then 
 the map $\Gamma$ gives an isomorphism from $\Cu(A)_+$ 
 onto ${\rm LAff}_+(\wtd{QT}(A)).$
 In \cite{ERS}, it is shown that if $A$ is ${\cal Z}$-stable, 
 then the map $\Gamma$ is indeed surjective, which extends  an earlier result
 of \cite{BPT}.  
 More recently, it is proved in \cite{Th}  and \cite{APRT} that   when $A$ has stable rank one,
 $\Gamma$ is  surjective. 
 %
 We show that if $A$ is  a $\sigma$-unital simple \CA\, {{which}} 
 has strict comparison and 
 tracial approximate oscillation zero,
  then the map $\Gamma$ is surjective. 
  {{On the other hand,}} if $A$ is a {{$\sigma$-unital}}  stably finite simple \CA\, with strict comparison which has almost 
 stable rank one and $\Gamma$ is surjective,  then $A$ has tracial approximate oscillation zero. 
 
 {{Let $A$ be a $\sigma$-unital simple \CA.}}
 We also found 
 that if $A$  has tracial approximate oscillation zero,  then 
 $A$ has a nice matricial structure, a property that we call (TM) (see Definition \ref{DTM}). 
 We prove that if $A$ has strict comparison  and has property (TM), 
 then $A$ has stable rank one. 
 As a by-product, we show that, if $A$ has strict comparison and $\Gamma$ is surjective,
 then the condition that  $A$ has almost stable rank one  implies that $A$ actually has stable rank one.

Our main  result may be stated as follows:

\begin{thm}\label{Teqiv}
Let $A$ be a separable simple \CA\, 
which admits at least one 
densely defined
non-trivial  2-quasitrace  and 
has strict comparison.

Then the following are equivalent:

(1) $A$ has tracial approximate oscillation zero;

(2) $\Gamma$ is surjective (see \ref{DGamma}) and $A$ has stable rank one;

(3)  $A$ has stable rank one;

(4) $\Gamma$ is surjective and $A$ has almost stable rank one; 

(5) 
$A$ has property (TM).

\end{thm}
The technical terms in the statement above  will be discussed in detail in the process and 
some examples of simple \CA s which satisfy (1) will be  given (e.g. 
Proposition \ref{Prr0} and  Theorem \ref{Tcounableos}). The condition that $A$ has a non-trivial densely defined 2-quasitrace could be replaced by
that $A$ is stably finite (see 
 Remark \ref{Rlast}).
Note that Theorem \ref{Teqiv} is stated without assuming that $A$ is nuclear or exact. 
Related to the Toms-Winter conjecture, 
H. Thiel in \cite{Th} shows that under the same assumption as that of Theorem \ref{Teqiv}, 
if $A$  is unital and has stable rank one, then $\Gamma$ is surjective, and, if, 
in addition, $A$ has local finite nuclear dimension, then $A$ is ${\cal Z}$-stable.
With the same spirit, Corollary \ref{Clocalnuc}  below states that, under the same assumption 
as in Theorem \ref{Teqiv}, if (1)  in the theorem above also holds and $A$ has local finite nuclear dimension, 
then $A$ is ${\cal Z}$-stable (see also Remark \ref{Rweakloc} for an even weaker hypothesis).  In fact, the idea of tracial oscillation zero can also be directly used in 
the study of Toms-Winter conjecture (see \cite{Linzstable}).

 The paper is organized as follows.
 Section 2 is a  preliminary that lists  a number of notations and definitions that are used 
 in the paper. It also includes some known facts which may not be stated explicitly 
 in the literature.  Section 3 discusses some  preliminary cancellation properties that will be used later. 
  In section 4, we recall the notion of tracial oscillation  and  introduce 
 the notion of  
 tracial approximate
 oscillation for positive elements.   In section 5, we introduce the {{notion}} of 
tracial approximate oscillation  zero for \CA s and give some examples of separable \CA s 
which have positive tracial approximate oscillation and examples which have tracial approximate oscillation zero.  In particular, we show that, 
if  the cone of 2-quasitraces of $A$ has a basis $S$ which has 
countably many  
extremal points,
then $A$ 
has tracial approximate  oscillation  zero.
In section 6, we study sequence algebras and its quotients for compact \CA s $A$.
We find that $l^\infty(A)/I_{_{\Qw}},$ where $I_{_{\Qw}}$ is the quasitrace kernel ideal,
is a SAW*-algebra and has real rank zero and stable rank one, provided  $A$ has tracial approximate oscillation zero.
Section 7 contains one of the main results: if $A$ has
strict comparison and 
tracial approximate oscillation zero, 
then $\Gamma$ is surjective.  In section 8, we introduce the property (TM), a property 
of tracial matricial structure.  We show 
that, under the assumption  {{of strict comparison,
the  property (TM) is 
 equivalent  to  the property of tracial approximate oscillation zero.}}
The last  section  {{is devoted}}  to the proof of Theorem \ref{Teqiv} mentioned above,
in particular, (1) $\Rightarrow$ (2) without assuming that $A$ is separable (but $\sigma$-unital).

 \vspace{0.2in}
 
 {\bf Acknowledgments}
 
The first named author was 
partially supported by 
Natural Sciences and Engineering
Research Council of Canada Discovery Grant.
The second named author was partially supported by a NSF grant (DMS-1954600). Both authors would like to acknowledge the support during their visits
to the Research Center of Operator Algebras at East China Normal University
which is partially supported by Shanghai Key Laboratory of PMMP, Science and Technology Commission of Shanghai Municipality (STCSM), grant 
\#22DZ2229014
and a NNSF grant (11531003).

\section{Preliminary}

In this section we will give a list of basic notations and a number of definitions 
which will be used  throughout  this paper. Most of them are familiar to experts.
It also includes some basic facts about Cuntz semigroups and 2-quasitraces, as well as 
some ad hoc but more or less known facts. Readers are encouraged to skip them until 
they are needed.

\bigskip
 
{\bf Some basic notations and definitions}

\begin{nota}
In this paper,
the set of all positive integers is denoted by $\N.$ 
The set of all compact operators on a separable 
infinite-dimensional Hilbert 
space is denoted by ${\cal K}.$ 

Let  $A$ 
be a normed space and ${\cal F}\subset A$ a subset. Let  $\epsilon>0$.
For any pair  $a,b\in A,$
we  write $a\approx_{\epsilon}b$ if
$\|a-b\|< \epsilon$.
We write $a\in_\ep{\cal F}$ if there is $x\in{\cal F}$ such that
$a\approx_\ep x.$

Let $A$ be a \CA\ and $x\in A.$ Let  $|x|:=(x^*x)^{1/2}.$  
{{If $a, b\in A$ and $ab=ba=a^*b=ba^*=0,$ we often write $a\perp b.$}}

\end{nota}

\begin{nota}
Let $A$ be a $C^*$-algebra and
 $S\subset A$ a subset of $A.$
Denote by
${\rm Her}_A(S)$ (or just $\Her(S),$ when $A$ is clear)
the hereditary $C^*$-subalgebra of $A$ generated by $S.$
Denote by $A^{\bf 1}$ the unit ball of $A,$ and 
by $A_+$ the set of all positive elements in $A.$
Put $A_+^{\bf 1}:=A_+\cap A^{\bf 1}.$
Denote by $\wtd A$ the minimal unitization of $A.$
When $A$ is unital, denote by $GL(A)$ the group of invertible elements of $A,$
 and by   $U(A)$ the unitary group of $A.$
 Let  ${\rm Ped}(A)$ denote
the Pedersen ideal of $A,$ ${\rm Ped}(A)_+= {\rm Ped}(A)\cap A_+,$
${\rm Ped}(A)^{\bf 1}=A^{\bf 1}\cap {\rm Ped}(A)$ and ${\rm Ped}(A)_+^{\bf 1}={\rm Ped}(A)_+\cap {\rm Ped}(A)^{\bf 1}.$
  Denote by $T(A)$ the tracial state space of $A.$
 Except the Pedersen ideal, all other ideals mentioned in this paper are {\bf closed two-sided} ideals.
\end{nota}
\begin{df}
Let $A$ and $B$ be \CA s and 
$\phi: A\rightarrow B$ a  linear map.
The map $\phi$ is said to be positive if $\phi(A_+)\subset B_+.$  
The map $\phi$ is said to be completely positive contractive, abbreviated to  c.p.c.,
if $\|\phi\|\leq 1$ and  
$\phi\otimes \mathrm{id}: A\otimes M_n\rightarrow B\otimes M_n$
are positive for all $n\in\mathbb{N}.$ 
A c.p.c.~map $\phi: A\to B$ is called order zero, if for any $x,y\in A_+,$
$xy=0$ implies $\phi(x)\phi(y)=0$ (see  Definition 2.3  of \cite{WZ09OZ}).

In what follows, $\{e_{i,j}\}_{i,j=1}^n$ (or just $\{e_{i,j}\},$ if there is no confusion) stands for  a system of matrix {{units}} for $M_n,$
$1_n$ for  the identity of $M_n,$   $\iota\in C_0((0,1])$ 
for  the identity function on $(0,1],$  i.e., $\iota(t)=t$ for all $t\in (0,1].$
 {{We also write  $\{e_{i,j}\}$ for 
a system of matrix units for ${\cal K}.$}}

\end{df}
\begin{df}
A   \CA\ $A$ is said to have stable rank one
(\cite{Rff}) 
if  $\widetilde A=\overline{GL(\widetilde A)},$
i.e.,  $GL(\widetilde A)$ is dense in $\widetilde A.$
A  \CA\ $A$ is said to have 
almost stable rank one (\cite{Rlz})
if, for any hereditary $C^*$-subalgebra $B\subset A,$
$B\subset  \overline{GL(\widetilde B)}.$
\end{df}

\begin{nota}\label{Nfg}
Let $\epsilon, \dt >0.$ Define  continuous functions
$f_{\epsilon}, g_\dt
: [0,+\infty) 
\rightarrow [0,1]$ by
{\small{\beq\nonumber
f_{\epsilon}(t)=
\begin{cases}
0  &t\in [0,\epsilon/2],\\
1 &t\in [\epsilon,\infty),\\
\mathrm{linear } &{t\in[\epsilon/2, \epsilon],}
\end{cases}\andeqn
g_\dt(t)=\begin{cases}
0  &t\in \{0\}\cup [\dt, 
\infty),\\
1 &{t\in[\dt/8, \dt/2]},\\
\mathrm{linear}  &t\in [0,\dt/8]\cup [\dt/2, \dt].
\end{cases}
\eneq}}
(Note that $(t-\dt/2)_+$ and $f_{\dt}$ have the same support.)
\end{nota}


{\bf Cuntz semigroup and quasitraces}

\begin{df}\label{Dcuntz}
Let $A$ be a \CA\
and let  $a,\, b\in (A\otimes {\cal K})_+.$ 
{{We}} write $a \lesssim b$ if there are
$x_k\in A\otimes {\cal K}$
such that
$\lim_{k\rightarrow\infty}\|a-x_k^*bx_k\|=0$.
We write $a \sim b$ if $a \lesssim b$ and $b \lesssim a$  both hold (\cite{Cuntzdim}).
The Cuntz relation $\sim$ is an equivalence relation.
Set $\Cu(A)=(A\otimes {\cal K})_+/\sim.$  
Denote by  $V(A)$  the subset of those elements in $
\Cu(A)$  {{which are}} represented by projections.

\end{df}

\begin{df}\label{Dqtr}
{{Let $A$ be a \CA. 
A densely  defined {{2-quasitrace}}  is a 2-quasitrace defined on ${\rm Ped}(A\otimes {\cal K})$ (see  Definition II.1.1 of \cite{BH}). 
Denote by ${\widetilde{QT}}(A)$ the set of densely defined 2-quasitraces 
on 
$A\otimes {\cal K}.$  
 In what follows we will identify 
$A$ with $A\otimes e_{1,1}$ whenever it is convenient. 
Note that we require that a 2-quasitrace has finite value on ${\rm Ped}(A\otimes {\cal K}).$ 
In particular, we exclude the function on ${\rm Ped}(
{{A\otimes {\cal K}}})$ with only $\infty$ value
from the {{consideration.}}


We endow ${\widetilde{QT}}(A)$ 
{{with}} the topology  in which a net 
${{\{}}\tau_i{{\}}}$ 
 converges to $\tau$ if 
${{\{}}\tau_i(a){{\}}}$ 
 converges to $\tau(a)$ for all $a\in 
 {\rm Ped}(A\otimes {\cal K})$ 
 (see also (4.1) on page 985 of \cite{ERS}).

Note {{that,}} for each $a\in ({{A}}
\otimes {\cal K})_+$ and $\ep>0,$ $f_\ep(a)\in {\rm Ped}(A\otimes {\cal K})_+.$ 
Define, for each $\tau\in {\wtd{QT}}(A),$
\beq
{{\widehat{a}(\tau):=\tau(a):=\lim_{\ep \to 0}\tau(af_\ep(a))\andeqn}}
\widehat{[a]}(\tau):=d_\tau(a):=\lim_{\ep\to 0}\tau(f_\ep(a)).
\eneq
We will use properties of 2-quasitraces discussed in \cite{BH} and \cite{ERS} (see, in particular, section 4.1 and  Theorem 
4.4 of \cite{ERS}).  Denote by  ${\wtd T}(A)$ {{the}} subset of $\wtd{QT}(A)$ consisting of traces.}}

\end{df}

\begin{df}\label{DQw}
Recall (Theorem 4.7 of \cite{eglnp}) that a $\sigma$-unital \CA\, $A$ is compact if and only 
if $A={\rm Ped}(A).$
Every unital \CA\, is compact.
Let $A$ be a  compact \CA. {{Since $A={\rm Ped}(A),$ every (densely defined) 2-quasitrace 
is actually defined on $A.$ By II 2.3 of \cite{BH}, every 2-quasitrace on $A$ is bounded.}}
{{Put}} $QT_{[0,1]}(A)=\{\tau\in  {\wtd{QT}}(A): \|
\tau{|_A}\|\le 1\}.$ 
Then $QT_{[0,1]}(A)$ is a 
compact convex subset of ${\wtd{QT}}(A)$
(see \cite[Theorem 4.4]{ERS}).
Denote by $QT(A)$ the set of {{2-quasitraces}}  $\tau$ with $\|\tau{|_A}\|=1.$ 
It is {{a}} convex subset of ${\wtd{QT}}(A).$ 
%
Denote by $\overline{QT(A)}^w$
the (weak*) closure of $QT(A).$   Then, {{in the case that $A$ is compact and 
$\wtd{QT}(A)\setminus \{0\}\not=\emptyset,$}}   
$\R_+ \cdot \Qw={\wtd{QT}}(A)$ 
(if $QT(A)=\emptyset,$ {{then}} $\overline {QT(A)}^w=\emptyset$).

Let $I\subset A$ be an ideal and $\{e_\lambda\}$ be a quasi-central 
approximate identity for $I.$ Suppose that $\tau\in \wtd{QT}(I).$ Then 
$\tau(a)=\lim_\lambda \tau(ae_\lambda)$ (for $a\in A$) defines a (densely defined) 2-qausitrace of $A.$
Note that $\|\tau|_A\|=\|\tau|_I\|.$ If $\tau\in \wtd{QT}(A),$ then 
$\tau_I(a)=\lim_{\lambda} \tau(ae_\lambda)$
also {{densely defines}} a 2-quasitrace of $A$ with $\|\tau_I|_{A}\|\le \|\tau\|$
(see  Definition 2.5 of \cite{Lincrell}).
Let $a\in {\rm Ped}(A\otimes {\cal K})_+$ and $I_a$ be the ideal generated by $a.$ 
By \cite[II.4.2.]{BH}, 
every $\tau$ in 
$\wtd{QT}(\Her(a))$ 
can be uniquely extended to  a 2-quasitrace $\tau$ in 
$\wtd{QT}(I_a).$ 
In what follows we will identify $\wtd{QT}(\Her(a))$ with $\{\tau_{I_a}: \tau\in \wtd{QT}(\Her(a))\}.$
\end{df}


The following is a quasitrace version of Lemma 4.5 of \cite{eglnp}.

\begin{prop}[Lemma 4.5 of \cite{eglnp}]\label{Pcompact}
Let $A$ be a $\sigma$-unital compact \CA.  Then $0\not\in \Qw$ and $\Qw$ is compact.

\end{prop}

\begin{proof}
{{We may assume that $QT(A)\not=\emptyset.$}}
By  Lemma 4.4 of \cite{eglnp},  
there is $e_1\in M_n(A)$ with $0\le e_1\le 1$ and $x\in M_n(A)$ 
such that $e_1x^*x=x^*xe_1=x^*x$  and $a_0=xx^*$ is a strictly positive element 
of $A$ (for some $n\in \N$).  Note that $\tau(e_1)\ge d_\tau(a_0)=1$ {{for all $\tau\in QT(A).$}}
 Note   also that
\beq
QT(A)=\{\tau\in {\wtd{QT}}(A): d_\tau(a_0)=1\}.
\eneq
Put $S=\{\tau\in QT_{[0,1]}(A): \tau(e_1)\ge 1\}.$ Then 
$S$ is {{compact}} and $0\not\in S.$  Since 
$\tau(e_1)\ge d_\tau(a_0)=1,$
$QT(A)\subset S.$ So $\Qw\subset S$ and $0\not\in \Qw.$
This also implies that $\Qw$ is compact.
\end{proof}

\begin{prop}\label{Pboundedness}
Let $A$ be a $\sigma$-unital 
 \CA\, and 
$S_1, S_2\subset {\wtd{QT}}(A)$  nonempty compact subsets.
 Then  {{one has  the following (with $\|\tau\|=\|\tau|_A\|$):}}
 
 (1) If $\R_+\cdot S_1=
 \wtd{QT}(A)$ and $0\not\in S_1,$
then there exists  $L_1\in \R_+$ such that
$$
S_2\subset \{r \cdot s: s\in S_1\andeqn r\in [0,L_1]\}.
$$
%
 %
 
(2) {{If $a\in {\rm Ped}(A\otimes {\cal K})_+^{\bf 1},$}}
  then 
$
{{d=\sup\{\|\tau|_{\Her(a)}\|: \tau\in S_1\}}}<\infty.
$

(3) If $A$ is compact,  then $M_1=\sup\{\|\tau\|: \tau \in S_1\}<\infty.$

(4) If $a$ is as in (2),  and $S_1$ is as in (1), 
then $\overline{QT(\Her(a))}^w \subset \{r\cdot \tau: \tau\in S_1, r\in [0, L]\}$ for some $L\in \R_+$ 
{{(see the last paragraph of Definition \ref{DQw}).}}
\end{prop}

\begin{proof}
To see (1) holds, let us assume  otherwise.
Then there exist sequences $r_n\in \R_+,$  $s_n\in S_1$  and $t_n\in S_2$ 
such that  $r_ns_n=t_n,$ $n\in \N$ and $\lim_{n\to\infty} r_n=\infty.$
Since both $S_1, S_2$ are compact, \wilog, we may assume that 
$s_n\to s\in S_1$ and $t_n\to t\in S_2.$ 

Since $s\not=0,$ choose $
c\in {\rm Ped}(A)_+^{\bf 1}$ such that $s(c)>0.$ 
It follows that there exists $n_0\in \N$ such that 
{{$s_n(c)>s(c)/2>0$}} 
for all $n\ge n_0.$ 
Consequently,
\beq
t_n(c)=r_ns_n(c)\to\infty.
\eneq
Hence   {{$t(c)=\infty.$}}
 However 
 $c\in {\rm Ped}(A)_+.$ A contradiction. 

For (2),  since $a\in {\rm Ped}(A\otimes {\cal K})_+,$
there are $b_i\in (A\otimes {\cal K})_+$ and $f_i\in C_{c}((0, \infty))_+$
{($1\le i\le m$),}
the set of continuous functions  with compact supports, 
such that
$
a\le \sum_{i=1}^m f_i(b_i)
$
(see \cite[5.6.1]{Pedbk}).
It follows that 
$a\lesssim \diag(f_1(b_1), f_2(b_2),...,f_m(b_m)).$ 
One can choose $f\in C_{c}((0, \infty))_+$ with $0\le f\le 1$ such that $ff_i=f_i,$
$1\le i\le m.$  Put $b=\diag(f(b_1), f(b_2),...,f(b_m))).$ Then
\beq
\tau(b)\ge d_\tau(\diag(f_1(b_1), f_2(b_2),...,f_m(b_m)))\ge d_\tau(a)\rforal \tau\in S.
\eneq
But  $\widehat{b}$ 
is bounded on the compact subset $S_1.$ Put $M=\sup\{\tau(b): \tau\in {{S_1}}\}.$ 
Then $M<\infty$ and 
\beq
\sup\{\|\tau|_{\Her(a)}\|:\tau\in S_1\}=\sup\{d_\tau(a): \tau\in S_1\}\le M.
\eneq

To see (3), let $a\in A$ be a strictly positive element.  Since $A={\rm Ped}(A),$ 
$a\in {\rm Ped}(A)_+$ and $\Her(a)=A.$  Thus (3) follows from (2).

For (4), let $I_a$ be the (closed) ideal of $A\otimes {\cal K}$ generated by 
$a.$ 
Then $\{\tau_I: \tau\in {\overline{QT(\Her(a))}}^w\}$ is a compact subset of $\wtd{QT}_{[0, 1]}(A)$
(see the last paragraph of \ref{DQw}).
Hence part (4) of the lemma  then follows from (1).
\end{proof}

{\bf Comparison and canonical map $\Gamma$}

\begin{df}

A simple \CA\ 
$A$ 
is said to have (Blackadar's) strict comparison, if, for any $a, b\in (A\otimes {\cal K})_+,$ 
one has 
$a\lesssim b,$ {{provided}}
\beq
d_\tau(a)<d_\tau(b)\rforal \tau\in {{{\widetilde{QT}}(A)\setminus \{0\}.}}
\eneq

\end{df}

\begin{NN}\label{Ralgebraicsimpl}
Let $A$ be a $\sigma$-unital \CA\, and $e\in {\rm Ped}(A\otimes {\cal K})_+\setminus \{0\}.$ 
If $e$ is a full element, 
put $T_e=\{\tau\in \wtd{QT}(A): \tau(e)=1\}.$ Then $T_e$ is a compact convex subset and 
is a basis for the cone $\wtd{QT}(A)$ (see Proposition 3.4 of \cite{T-0-Z}). 
If, in addition, $A$ is simple, then $e$ is always full.
Put $A_1=\Her(e).$ By Brown's stable isomorphism theorem {{(see \cite{Br})}}, $A\otimes {\cal K}\cong A_1\otimes {\cal K}.$ So $e\in {\rm Ped}(A_1\otimes {\cal K})_+.$   
Then $A_1={\rm Ped}(A_1)$ (see, for example, (iii) of Theorem 2.1 of \cite{T-0-Z}),  in other words, $A_1$ is algebraically simple.
Therefore, instead of studying $A\otimes {\cal K},$ we will 
study $A_1\otimes {\cal K}.$     Throughout the paper, we 
{{often}} consider 
$\sigma$-unital simple \CA s\, $A$ with ${\rm Ped}(A)=A$ ({{in other words,}} algebraically simple \CA s).
 %
\end{NN}

\begin{df}\label{DGamma}
Let $A$ be a \CA\, with ${\wtd{QT}}(A)\setminus\{0\}\not=\emptyset.$
{{Denote by  $L(\wtd{QT}(A))$ the family of continuous real valued 
functions $f$  on $\wtd{QT}(A)$ 
such that $f(\af \tau)=\af f(\tau)$
for all $\af\in \R_+$ and $\tau\in \wtd{QT}(A)$ and 
$f(\tau+t)=f(\tau)+f(t)$ for all $\tau, t\in \wtd{QT}(A).$}}
Let $S\subset {\wtd{QT}}(A)$ be a convex subset. 
{{Set}}
\beq
\Aff_+(S)&=&\{f|_S: f\in L(\wtd{QT}(A)), f(\tau)>0\,\, {\rm if}\,\, \tau\in S\setminus \{0\}\}\cup \{0\},\\
{\rm LAff}_+(S)&=&
\{f:S\to [0,\infty]: \exists \{f_n\}, f_n\nearrow f,\,\,
 f_n\in \Aff_+(S)\}
 \eneq
{{Note that if $0\in S,$ then $f(0)=0$ for all $f\in {\rm LAff}_+(S).$}}
For  a {{simple}}
{{\CA\,}} $A$ and $a\in (A\otimes {\cal K})_+,$ the function $\hat{a}(\tau)=\tau(a)$ ($\tau\in S$) 
is in general in ${\rm LAff}_+(S).$   If $a\in {\rm Ped}(A\otimes {\cal K})_+,$
then $\wh{a}\in \Aff_+(S).$
{{Recall that}} $\widehat{[a]}(\tau)=d_\tau(a)$ for $\tau\in {\wtd{QT}}(A).$ {{So}}
$\widehat{[a]}\in {\rm LAff}_+(\wtd{QT}(A)).$
 {{Caution: $\wh{a}$ and $\wh{[a]}$ are not the same in general.}}

We will write $\Gamma: \Cu(A)\to {\rm LAff}_+({\wtd{QT}}(A))$ for 
the canonical map defined by $\Gamma([a])(\tau)=\widehat{[a]}(\tau)=d_\tau(a)$ 
for all $\tau\in {\wtd{QT}}(A).$
\end{df}

{{(1)}} In the case that $A$ is simple and $A={\rm Ped}(A),$ 
$\Gamma$ also induces a canonical map 
$\Gamma_1: \Cu(A)\to {\rm LAff}_+(\Qw).$ Since, in this case, 
$\R_+\Qw={\wtd{QT}}(A),$ the map $\Gamma$ is surjective if and only if $\Gamma_1$
is surjective.

{{(2)}} In the case that $A$ is stably finite and simple,
denote by $\Cu(A)_+$ the set of purely non-compact elements (see Proposition 6.4 of \cite{ERS}).
Suppose that $\Gamma$ is surjective. 
Let $p\in (A\otimes {\cal K})_+$ be a projection (so $p\in {\rm Ped}(A\otimes {\cal K})$).
There are $a_n\in (A\otimes {\cal K})_+$ with $0\le a_n\le 1$ such 
that $\widehat{[a_n]}=(1/2^n)\widehat{[p]},$ $n\in \N.$ 
Define $b=\diag(a_1/2, a_2/2^2,...,a_n/2^n,...)\in  A\otimes {\cal K}.$ 
Then $0$ is a limit point of ${\rm sp}(b).$ Therefore $[b]$ cannot be represented by 
a projection. In other words, $[b]\in \Cu(A)_+.$  
We compute that ${\widehat{[b]}}={{\widehat{[p]}}}.$ 
It {{then}} follows that  $\Gamma|_{\Cu(A)_+}$ is surjective.

Suppose that $A$ is simple and $a$ is a purely non-compact element and 
\beq
d_\tau(a)\le d_\tau(b)\rforal \tau\in {\wtd{QT}}(A).
\eneq
Then,  {{for}} any $\ep>0$ {{(recall that $f_\ep(a)\in {\rm Ped}(A)$)}},
\beq
d_\tau(f_\ep(a))<d_\tau(b)\rforal \tau\in \wtd{QT}(A).
\eneq
If $A$ has strict comparison, then $f_\ep(a)\lesssim b.$ Since this holds for all $\ep>0,$
we conclude that $a\lesssim b.$ 

The reader should be {{reminded}} that when $A$ is exact, every {{2-quasitrace}} is a trace (see \cite{Haagtrace}).
These facts will be used without further explanation.

\vspace{0.1in}

{\bf Cuntz null sequences
and the ideal generated by Cuntz null sequences}

\begin{df}\label{Dcnull-1}
Let $A$ be a separable non-elementary simple \CA. Then 
$A$ contains a sequence of nonzero elements $e_n\in {\rm Ped}(A)$ with 
$0\le e_n\le 1$ such that 
$e_{n+1}\lesssim e_n$ for all $n\in \N,$ and for any finite subset ${\cal F}\subset A_+\setminus \{0\},$
there exists $n_0\in \N$ such that, for all $n\ge n_0,$ 
\beq
n[e_n]\le [d]\rforal d\in {\cal F}
\eneq
(see Lemma 4.3 of \cite{FL21}).

{{For general \CA\, $A,$ a
}} sequence $\{ a_n\}\subset 
A_+$ is said to be {\it truly} Cuntz-null and {{written}} 
$a_n\cto 0$  if, for any finite subset ${\cal F}\subset A_+{{\backslash\{0\}}},$ 
there exists $n_0\in \N$ such that, for all $n\ge n_0,$
\beq
a_n\lesssim d\rforal d\in {\cal F}.
\eneq
This is equivalent to {{saying that,}} for any $d\in A_+\setminus \{0\},$ there exists $n_0\in \N$ such that, 
for all $n\ge n_0,$
$
a_n\lesssim d.
$
We also write  $a_n\cdto 0$ if $a_{n+1}\lesssim a_n$ for all $n\in \N$ and 
$a_n\cto 0.$

A sequence $\{x_n\}\subset A\otimes {\cal K}$ is said to be Cuntz-null if, 
for any $\ep>0,$ $f_\ep(x_n^*x_n)\cto 0.$

\end{df}

\begin{df}\label{DNcu}
 Let $X$ be a normed space and let 
$l^\infty(X)$ denote the space of bounded sequences of $X.$ 
When $A$ is a \CA, 
$l^\infty(A)$ is also a \CA, and 
$c_0(A)=\{\{a_n\}\in l^\infty(A): \lim_{n\to\infty}\|a_n\|=0\}$ 
is an
ideal of $l^\infty(A).$ 
Let $A_\infty=l^{\infty}(A)/c_0(A)$  {{and}}
$\pi_\infty: l^\infty(A)\to 
A_\infty$ be the quotient map. 
We view $A$ as a 
\SCA\, of $l^\infty(A)$ 
via the canonical map $\iota: a\mapsto\{a,a,...\}$ for all $a\in A.$
In what follows, we may identify $a$ with the constant sequence $\{a,a,...\}$ in 
$l^\infty(A)$ without further warning.
Let $\{C_n\}$ be a sequence of \SCA s of $A.$
We may also use notation  $l^\infty(\{C_n\})=\{\{c_n\}\in l^\infty(A): c_n\in  C_n\}$
for  the infinite product of  $\{C_n\}.$

Denote by     
$
N_{cu}(A)
$
(or just $N_{cu}$)
 the set of all Cuntz-null sequences  in $l^\infty(A).$
 
 It follows from Proposition 3.5 of \cite{FLL21}  that,  if $A$ has no one-dimensional hereditary \SCA s, then 
$N_{cu}(A)$ is an ideal  of $l^\infty(A).$
 Moreover, if $A$ is non-elementary and simple, $c_0(A)\subsetneqq N_{cu}(A).$
 Denote by $\Pi_{cu}: l^\infty(A)\to l^\infty(A)/N_{cu}(A)$ the quotient map {{and}}
 $\Pi_{cu}(A)^\perp=\{b\in l^\infty(A)/N_{cu}: b\Pi_{cu}(a)=\Pi_{cu}(a)b=0
 \mbox{ for all } a\in A\}.$
\end{df}

\begin{df}\label{D2norm}
Let $A$ be a 
\CA\,  
with ${{\wtd{QT}(A)\not=\emptyset.}}$
Fix a compact subset $T\subset {{\wtd{QT}(A).}}$
For each $x\in A,$   define
\beq
\|x\|_{_{2,T}}=\sup\{\tau(x^*x)^{1/2}: \tau\in T\}.
\eneq
Then $\|x^*\|_{_{2, T}}=\|x\|_{_{2, T}}.$

By Lemma 3.5 of \cite{Haagtrace} (one does not need to assume that $A$ is unital),
\beq
&&\tau(a+b)^{1/2}\le \tau(a)^{1/2}+\tau(b)^{1/2}\rforal a, b\in {{{\rm Ped}(A\otimes {\cal K})_+}}\andeqn \tau{{\in T,}}\\
&&\|x+y\|_{_{2,\tau}}^{2/3}\le \|x\|_{_{2,\tau}}^{2/3}+\|y\|_{_{2,\tau}}^{2/3}\,\,\rforal  {{x, y\in {\rm Ped}(A\otimes {\cal K})\andeqn}} \tau{{\in T.}}
\eneq
Then
\beq
\sup\{\|x+y\|_{_{2,\tau}}^{2/3}:\tau\in T\}\le \sup\{ \|x\|_{_{2,\tau}}^{2/3}:\tau\in T\}+
\sup\{\|y\|_{_{2,\tau}}^{2/3}:\tau\in T\}.
\eneq
In other words, 
\beq
\|x+y\|_{_{2,T}}^{2/3}\le \|x\|_{_{2,T}}^{2/3}+\|y\|_{_{2,T}}^{2/3}.
\label{F-0114-1}
\eneq
We also have
\beq
&&\|xy\|_{_{2,T}}\le \|x\| \|y\|_{_{2,T}}\andeqn \|xy\|_{_{2, T}}\le \|x\|_{_{2,T}}\|y\|.
\eneq
%
It follows that $\{a\in A: \|a\|_{_{2,T}}=0\}$ is a (closed two-sided) ideal of $A.$

{{We also have the following inequality for $a\in {\rm Ped}(A\otimes {\cal K})_+:$
\beq\label{Qnorm}
\|a\|_{_{2, T}}\le \|a\|(\sup\{d_\tau(a): \tau\in T\})^{1/2}.
\eneq
In fact, for all  $n\in \N,$ we have 
$\tau(a^2)=\tau(a^{1/2n}a^{2-(1/n)}a^{1/2n})\le \|a^{2-(1/n)}\|\tau(a^{1/n}).$
Let $n\to\infty.$ We obtain $\tau(a^2)\le \|a\|^2d_\tau(a).$ So \eqref{Qnorm} holds.}}
\end{df}

\begin{df}\label{DIQw}
Suppose that $A$ is a $\sigma$-unital  \CA\, with ${\wtd{QT}}(A)\setminus \{0\}\not=\emptyset,$ 
{{and}}
$T\subset {\wtd{QT}}(A)$ {{is}} a compact subset with $T\not=\{0\}.$
Define 
\beq
\hspace{-0.2in}I_{_{T}}=\{\{x_n\}\in l^\infty(A): \lim_{n\to\infty}\sup\{\tau(x_n^*x_n):\tau\in T\}=0\}.
\eneq
Then $I_{T}$ is an ideal of $l^\infty(A).$

Suppose that $A$ is a
simple non-elementary
\CA. Then 
it is clear that
\beq
N_{cu}(A)\subset I_{{_{\Qw}}}.
\eneq
It follows from  the proof of  Proposition 3.8 of \cite{FLL21} that $I_{_{\Qw}}=N_{cu}(A)$ 
if $A={\rm Ped}(A)$  and  {{$A$}} has strict comparison.
Denote by $\Pi: l^\infty(A)\to l^\infty(A)/I_{{{_{\Qw}}}}$ the quotient map.

\end{df}

\begin{prop}\label{Pideals=}
Let $A$ be a $\sigma$-unital algebraically simple \CA\, 
such 
that $QT(A)\not=\emptyset.$ Let $S
\subset {\wtd{QT}}(A)\setminus \{0\}$ be a compact subset
such that $QT(A)\subset \R_+ \cdot S.$
Then 
\beq
I_{_S}=I_{_{\Qw}}.
\eneq
Moreover, if $A$ has strict comparison, then 
$I_{_{\Qw}}=N_{cu}.$
\end{prop}

\begin{proof}
By Proposition \ref{Pboundedness},
$0<s_1=\sup\{\|\tau{{|_A}}\|: \tau\in S\}<\infty.$
Since $S\subset {\wtd{QT}}(A)\setminus \{0\}$ and is compact,
\beq
s_2:=
{{\inf}}
\{\|\tau{{|_A}}\|: \tau\in S\}>0.
\eneq
Suppose that $\{a_n\}
\in 
(I_{_{\Qw}})_+^{\bf 1}.$ 
Then, for any $\ep>0,$ there exists $n_0\in \N$ such that, if $n\ge n_0,$
\beq
\tau(a_n^2)< (\ep/(s_1+1))^2\rforal \tau\in \Qw.
\eneq
Thus, if $n\ge n_0,$  for any $t\in S,$ 
\beq 
t(a_n^2)={{\|t|_A\|}} (t/\|t|_A\|)(a_n^2)<{{\|t|_A\|}} (\ep/(s_1+1))^2\le \ep^2.
\eneq
This implies that $\{a_n\}\in I_{_S}.$ It follows that $I_{_{\Qw}} \subset I_{_S}.$ 

Conversely, 
if $\{a_n\}\in I_S,$ there exists $n_1\in\N$ such that, if $n\ge n_1,$
\beq
t(a_n^2)< s_2 \ep^2
\rforal t\in S.
\eneq
Note that, for all {{$\tau\in QT(A),$}}  there {{are}} $r_\tau\in \R_+$ and $t_\tau\in  S$ such that
$\tau=r_\tau t_\tau.$  {{Since $r_\tau\|t_\tau\|\le 1,$
we have}}
$
r_\tau\le 1/\|t_\tau\|\le 1/s_2.
$
{{Suppose $\tau\in \Qw.$ Then there are $t_n\in S$ and $r_n>0$ such that 
$r_n t_n\in QT(A)$ and  $r_nt_n\to \tau.$ As mentioned above, we have $r_n\le 1/s_2$
for all $n\in \N.$ Since $S$ is compact (and $\{r_n\}$ is bounded), by choosing a subsequence,
we may assume that 
$t_n\to t_\tau\in S$ and   $r_n\to r_\tau.$ In other words, $\tau=r_\tau t_\tau.$}}
Note that $\|\tau\|\le 1.$ So we also have $r_\tau\le 1/s_2.$  
%
Therefore, for any $n\ge n_1,$ if $\tau\in \Qw,$
\beq
\tau(a_n^2)=r_\tau 
{{t_\tau}}
(a_n^2)\le  (1/s_2)t_\tau(a_n^2)
<\ep^2.
\eneq
Thus $\{a_n\}\in I_{_{\Qw}}.$ 

To see the last part of the statement,  choose $b\in {\rm Ped}(A)_+^{\bf 1}\setminus \{0\}.$ 
{{Let}} $S=\{\tau\in \wtd{QT}(A): \tau(b)=1\}.$ Then $S\subset \wtd{QT}(A)\setminus \{0\}$ is a compact subset
and, $\wtd{QT}(A)=\R \cdot S.$ By  (the ``Moreover" part of) Proposition 3.8 of \cite{FLL21}, 
$N_{cu}=I_S=I_{_{\Qw}}.$
\end{proof}

{{Let $A$ be a $\sigma$-unital simple \CA\, and $\{e_n\}$ be an approximate identity 
with $e_{n+1}e_n=e_ne_{n+1}$ ($n\in \N$). 
Recall that $A$ is said to have continuous scale, if, for any $a\in A_+\setminus \{0\},$
there is $n_0\in \N$ such that
\beq
e_m-e_n\lesssim a\,\,\, {{\rforal}} m>n\ge n_0.
\eneq
This definition does not depend on the choice of $\{e_n\}$ (see, 2.1 of \cite{Lin04cs}  and 2.5 of \cite{Lin91cs}).
With terminology of \ref{Dcnull-1}, $A$ has continuous scale if and only if, for any $m(n)>n,$
$e_{m(n)}-e_n\cto 0$ for any $\{e_n\}$ for which $e_{n+1}e_n=e_ne_{n+1}=e_n$ ($n\in \N$).}}

{{The following is known. The proof of it is exactly the same as {{that of}}  the case 
 $T(A)=QT(A)$ (see 5.1, 5.2, 5.3 and 5.4 of \cite{eglnp} for details, and also see {{the}} remark after Definition 6.3 of \cite{FLL21}).}}

\begin{thm}[cf. Theorem 5.3 and Proposition 5.4 of \cite{eglnp}, also \cite{Lin91cs}]\label{Pcontscale-1}
Let $A$ be a $\sigma$-unital simple \CA\, with {{a}} strict positive element $e_A,$ 
continuous scale and  $QT(A)\not=\emptyset.$
Then $QT(A)$ is compact and $\widehat{[e_A]}$ is continuous on ${\wtd{QT}}(A).$ {{Assuming}} $A$ has strict comparison, 
then $A$ has continuous scale if only if $\widehat{[ e_A]}$ is continuous on ${\wtd{QT}}(A).$
\end{thm}

\begin{prop}\label{Porthogoinal}
Let $A$ be a separable {{non-elementary}} simple \CA. Then $A$ has continuous scale 
if and only if  $\Pi_{cu}(A)^\perp=\{0\}.$
\end{prop}
  
\begin{proof}

Suppose that $A$ has continuous scale. Let $\{b_n\}\in l^\infty(A)_+^{\bf 1}$  {{be}}
such that $b=\Pi_{cu}(\{b_n\})\in \Pi_{cu}(A)^\perp.$ 
Fix a {{truly}} Cuntz-null sequence of $\{a_n\}$ 
in the unit ball of $A_+$ such that $a_n\not=0$ for all $n\in \N$ {{(see \ref{Dcnull-1}).}}
Let $e\in A_+^{\bf 1}$ be a strictly positive element. 
Note that, for each $k\in \N,$ $\{f_{1/2k}(e)b_n\}_{n\in \N}$ is a Cuntz-null sequence.
For each $k\in \N,$ there {{are}} $l(k), n(k)\in \N$ such that
\beq\label{Porthogoinal-1}
\hspace{-0.1in}
&&[f_{1/2k}(b_n^*f_{1/2k}(e)^2b_n)]\le [a_i]\,\,(1\le i\le k)
\rforal n\ge l(k)\\\label{120222-1}\andeqn 
&&\|(1-f_{{1/n}}(e))b_k\|<1/k\rforal n\ge n(k).
\eneq
We may assume that $l(k+1)>l(k)\ge k$ and $n(k)>{{2}}k$ for all $k\in \N.$
Define $c_{0,i}=c_{1,i}=0$ if $1\le i<l(1),$   and, if $l(k)\le i<l(k+1),$ define
\beq
{{c_{0,i}=f_{1/2k}(e)b_i\andeqn c_{1, i}=(1-f_{1/n(i)}(e))b_i.}}
\eneq
Put ${\bar b}_i=b_i-c_{0,i}-c_{1,i},$ $i\in \N.$  
{{Note that, if $l(k)\le i {{<}} l(k+1),$ ${\bar b_i}=(f_{1/n(i)}(e)-f_{1/2k}(e))b_i$ ($k>1$).}}
By \eqref{Porthogoinal-1}, 
one verifies that $\{c_{0,n}\}\in N_{cu}(A).$ 
In fact, for a fixed  $1/2>\ep>0,$  choose $k_0$ such that $1/k_0<\ep.$ 
For any finite subset ${\cal F}\subset A_+\setminus \{0\},$ choose $J\in \N$ such that
$a_i\lesssim b$ for all $b\in {\cal F}$ and for all $i\ge J.$ 
It follows that, if $k_1\ge J,$ by 
{{\eqref{Porthogoinal-1},}}
for all $l(k)\le i<l(k+1)$ and $k\ge \max\{k_0, k_1\},$   
$$
f_{\ep}(c_{0,i}^*c_{0,i})\le f_{1/2k}(c_{0,i}^*c_{0,i})\lesssim {{a_J}}\lesssim b
$$
for all $b\in {\cal F}.$   Hence $\{c_{0,n}\}\in N_{cu}(A).$  Also, by \eqref{120222-1}, $\{c_{1,i}\}\in c_0(A).$ 
It follows that
\beq
\Pi_{cu}(\{b_n\})=\Pi_{cu}(\{{\bar b}_n\}).
\eneq
It suffices to  show that $\{{\bar b}_n\}\in  N_{cu}.$
In fact, for all $l(k)\le i <l(k+1),$ 
\beq
{{{\bar b}_i^*{\bar b}_i}}\lesssim f_{1/2n(l(k+1))}(e)-f_{1/2k}(e),\,\,k\in \N.
\eneq
Since $A$ has continuous scale, {{then}} $f_{1/2n(l(k+1))}(e)-f_{1/2k}(e)\cto 0$ (see, 2.1 of \cite{Lin04cs}  and 2.5 of \cite{Lin91cs}, for example).
It follows that $\{{\bar b}_n\}\in N_{cu}(A).$ This implies that $\{b_n\}\in N_{cu}(A).$
Consequently $\Pi_{cu}(A)^\perp=\{0\}.$

Conversely,  suppose that 
$\Pi_{cu}(A)^\perp=\{0\}.$
{{Let $e_n=f_{1/2n}(e),$ $n\in \N.$}} 
Choose any $m(n)>n.$
Define $d_n=e_{4m(n)}-f_{1/n}(e).$ Then, for any $a\in A,$ 
$\lim_{n\to\infty}ad_n=0.$ It follows that $\Pi_{cu}(\{d_n\})\in \Pi_{cu}(A)^\perp=\{0\}.$
In other words, $\{d_n\}\in N_{cu}(A).$ 
Therefore, for any $0<\dt<1/4,$
\beq
f_\dt(d_n)\cto 0.
\eneq
Note {{that,}} for all $n\in \N,$ 
{{$$
f_{1/4m(n)}-f_{1/2n}\le f_\dt(f_{1/8m(n)}-f_{1/n})\,\,\,{\rm in}\,\,\, C_0((0,1]).
$$}}
Thus
\beq
e_{2m(n)}-e_{n}\lesssim f_\dt(d_n),\,\, n\in \N.
\eneq
It follows that $(e_{2m(n)}-e_{n})\cto 0.$ 
Hence $A$ has continuous scale.
\end{proof}

\section{Comparison and cancellation of 
projections}

\begin{lem}
\label{P1025-2}
Let $A$ be a \CA\ and $\tau\in {{\wtd{QT}}}(A)\setminus \{0\}.$ 
Let $e\in A_+$ and $a\in A$ such that 
$ea=a=ae.$ Suppose that $\tau(e)<\infty.$ Then 
for any 
$f\in C_0({\R}),$ {{it holds that}}
$\tau(f(e-a^*a))=\tau(f(e-aa^*)).$
In particular, $\|e-a^*a\|_{2,\tau}=\|e-aa^*\|_{2,\tau}.$
Moreover, 
$d_\tau(g(e-a^*a))=d_\tau(g(e-aa^*))$
{{for any $g\in C_0(\R),$}}
assuming $g(e-a^*a)$ and $g(e-aa^*)$ are positive.

\begin{proof}
Note that since $C^*(e,a^*a)$ and $C^*(e,aa^*)$ 
are commutative,
the restrictions of $\tau$ on them are linear. 
Let $n\in\N.$ Then 
\beq\nonumber
\tau((e-a^*a)^n)
&=&\tau(e^n+\sum_{m=1}^n\frac{n!}{m!(n-m)!}({{-}}a^*a)^m)
=
\tau(e^n)+\sum_{m=1}^n\frac{n!}{m!(n-m)!}\tau(
({{-}}a^*a)^m)
\\\nonumber
&=&
\tau(e^n)+\sum_{m=1}^n\frac{n!}{m!(n-m)!}\tau(
({-}aa^*)^m)
\\
&=&\tau(e^n+\sum_{m=1}^n\frac{n!}{m!(n-m)!}
({{-}}aa^*)^m)
=\tau((e-aa^*)^n).
\label{e1025-1}
\eneq
Thus,  for any polynomial $P,$
$\tau(P(e-a^*a))=\tau(P(e-aa^*)).$
In particular, $\|e-a^*a\|_{2,\tau}=\|e-aa^*\|_{2,\tau}.$
Therefore, by the continuity of 2-quasitraces  (see 
\cite[Corollary II.2.5]{BH}), and {{the}} Stone-Weierstrass theorem,  
$\tau(f(e-a^*a))=\tau(f(e-aa^*))$ for all  $f\in C_0({\R}).$
Moreover {for any $g\in C_0(\R),$} assuming $g(e-a^*a)$ and ${{g}}(e-aa^*)$ are positive,
\beq
d_\tau(g(e-a^*a))=\sup_{\ep>0}
\tau(f_\ep(g(e-a^*a)))=
\sup_{\ep>0}
\tau(f_\ep(g(e-aa^*)))=d_\tau(g(e-aa^*)).
\eneq
\end{proof}

\end{lem}

\begin{lem}
\label{Lem1026-1}
Let $A,B$ be  \CA s  and 
 $\pi:A\to B$ be a surjective homomorphism. 
Assume $p,q\in B$ are projections, and $x\in B$
satisfies $px=x=xq.$
Then there are $\tilde p,\tilde q\in A_+^{\bf 1}$ 
and $\tilde x\in A,$
such that $\pi(\tilde  p)=p,$ $\pi(\tilde q)=q,$
$\pi(\tilde x)=x,$
and $\tilde p\tilde x=\tilde x=\tilde x\tilde q.$
Moreover, if $p=q,$ we can take $\tilde p=\tilde q.$
\end{lem}
\begin{proof}
We may assume that $\|x\|\leq 1.$
Let $p_1,q_1\in A_+^{\bf 1}$ such that 
$\pi(p_1)=p,$ $\pi(q_1)=q.$
Since $p,q$ are projections,  we also have 
$\pi(f_{1/2}(p_1))
=f_{1/2}(\pi(p_1))
=p,$ 
and $\pi({{f_{1/2}}}(q_1))
=f_{1/2}(\pi(q_1))=q.$
Note $x^*x\leq q.$
By \cite[Proposition 1.5.10]{Pedbk},
there exists $y\in A^{\bf 1}$ such that $\pi(y)=x$ 
and $y^*y\le f_{1/2}(q_1).$ 
Put $\td x=f_{1/2}(p_1)y.$ Then 
$$
\pi(\td x)=px=x, \td x \td x^*=f_{1/2}(p_1)yy^*f_{1/2}(p_1)\le f_{1/4}(p_1)
$$
and $\td x^*\td x\le y^*y\le f_{1/2}(q_1).$
Set $\tilde p=f_{1/8}(p_1)$ 
and $\tilde q=f_{1/8}(q_1).$ 
Then $\pi(\tilde  p)=p,$ $\pi(\tilde q)=q.$
The facts that 
$f_{1/8}(p_1)f_{1/4}(p_1)=f_{1/4}(p_1)$
and $f_{1/8}(q_1)f_{1/4}(q_1)=f_{1/4}(q_1)$
imply that 
$\tilde p\tilde x=\tilde x=\tilde x\tilde q.$
Moreover, if $p=q,$ we can take $p_1=q_1,$
hence 
$\tilde p=\tilde q.$
\end{proof}

\begin{prop}
\label{P1025-3}
Let $A$ be a  \CA\ with 
${\wtd{QT}}(A)\setminus \{0\}\not=\emptyset.$
Suppose that $T\subset {\wtd{QT}}(A)\setminus \{0\}$ is a compact subset.
Then every  projection in $l^\infty(A)/I_{T}(A)$
is  finite (see \ref{DIQw}).
\end{prop}
\begin{proof}
Let $B=l^\infty(A)/I_T(A)$
and  $\pi:l^\infty(A)\to B$ be the quotient map.  
Assume $p\in B$ is a projection, $u\in B$ satisfies 
$u^*u=p$ and $uu^*\leq p.$ 
we need to show 
$uu^*=p.$

By Lemma \ref{Lem1026-1},
there are $a=\{a_1,a_2,...\}\in l^\infty(A)_+^{\bf 1}$
and $v=\{v_1,v_2,...\}\in l^\infty(A)$
such that 
$\pi(a)=p,$ $\pi(v)=u,$
and $av=v=va.$
Since $\pi(a)
=
p=\pi(v^*v),$
we have 
$\lim_{n\to\infty}\|a_n-v_n^*v_n\|_{_{2,T}}=0.$
By Lemma \ref{P1025-2},
$\|a_n-v_nv_n^*\|_{_{2,T}}=\|a_n-v_n^*v_n\|_{_{2,T}}\to 0$
$(n\to\infty).$
Hence 
$$
{{p-uu^*=\pi(\{a_1-v_1v_1^*,
a_2-v_2v_2^*,...\})=0,}}
$$
which shows $p$ is a finite projection. 
\end{proof}

\begin{prop}
\label{P1025-1}
Let $A$ be 
a non-elementary
simple 
\CA\ with ${\wtd{QT}}(A)\setminus \{0\}\not=\emptyset.$
Let $T\subset {\wtd{QT}}_{[0,1]}(A)\setminus \{0\}.$
Then for any 
 $a\in {\rm Ped}(A)_+^{\bf 1}\backslash\{0\},$
any $\ep>0,$
there is $b\in \Her(a)_+$ 
such that  {{$b\le a,$}}
$\|a-b\|_{_{2, T}}<\ep,$
and $d_\tau(b)<d_\tau(a)$ for all $\tau\in T.$
\end{prop}

\begin{proof}
It follows from the first paragraph of \ref{Dcnull-1} that
there exists  
$c\in \Her(a)_+$  with $\|c\|=1$ such that
$d_\tau(c)<\ep^2$ for all $\tau\in T.$ 
Define $b=a^{1/2}(1-f_{1/4}(c))a^{1/2}.$  
{{Then $0\le b\le a.$}}
It follows from \eqref{F-0114-1}  that 
$$
\|a-b\|_{_{2,T}}=\|a^{1/2}f_{1/4}(c)a^{1/2}\|_{_{2,T}}
\leq  \|f_{1/4}(c)\|_{_{2,T}}\leq 
(d_\tau(c))^{1/2}\leq \ep.
$$
For all $\tau\in T,$
\beq\nonumber
d_\tau(b)
&=&d_{\tau}(a^{1/2}(1-f_{1/4}(c))a^{1/2})
=
d_{\tau}((1-f_{1/4}(c))^{1/2} a (1-f_{1/4}(c))^{1/2})
\\\nonumber
&<&
d_{\tau}((1-f_{1/4}(c))^{1/2} a (1-f_{1/4}(c))^{1/2})+
d_\tau(f_{1/2}(c))
\\\nonumber
&\overset{{\rm (orthogonality)}}{=}&
d_{\tau}((1-f_{1/4}(c))^{1/2} a (1-f_{1/4}(c))^{1/2}
+f_{1/2}(c))
\le  d_\tau(a).
\eneq
\end{proof}

\begin{thm}
\label{P1025-0}
Let $A$ be a 
non-elementary
{{algebraically simple \CA\ with}}
$QT(A)\not=\emptyset.$
Assume that $A$ has strict comparison.
Then $l^\infty(A)/I_{\Qw}$
has cancellation of projections,
i.e., for any projections $p,q,r\in l^\infty(A)/I_{\Qw},$
if $p,q\leq r$ and $p\sim q,$ then  
$r-p\sim r-q.$
\end{thm}

\begin{proof}
Set $B:=l^\infty(A)/I_{_{\Qw}}$ and
{{let}}
$\Pi:l^\infty(A)\to B$ denote the quotient map.

Let $p,q,r\in B$
be projections with $p,q\leq r,$
and assume that
there is a partial isometry 
$v\in B$
such that $v^*v=p,$ $vv^*=q.$ 
By Lemma \ref{Lem1026-1},
there are $e=\{e_1,e_2,...\}\in l^\infty(A)_+^{\bf 1}$
and $w=\{w_1,w_2,...\}\in l^\infty(A)$
such that $\pi(e)=r,$ $\pi(w)=v,$ and 
$ew=w=we.$
Then, by Lemma \ref{P1025-2}, we have 
$
d_\tau(f_{1/4}(e_n-w^*_nw_n))=
d_\tau(f_{1/4}(e_n-w_nw_n^*))
$
for all $\tau\in \Qw$
and 
$n\in \N.$
By Proposition \ref{P1025-1}, 
for each $n\in \N,$
there is $b_n\in A_+^{\bf 1}$
such that 
\beq
&&\|f_{1/4}(e_n-w_n^*w_n)-b_n\|_{_{2,\Qw}}<1/n,
\text{ and }
\label{e1025-3}
\\
&&
d_\tau(b_n)<d_\tau(f_{1/4}(e_n-w_n^*w_n))=
d_\tau(f_{1/4}(e_n-w_nw_n^*))
\text{ for all } \tau\in \Qw.
\eneq 
Since $A$ has strict comparison, 
we have $b_n\lesssim f_{1/4}(e_n-w_nw_n^*).$ 
By \cite[Proposition 2.4 (iv)] {Ror92UHF2},
for each $n\in\N,$
there is {{$x_n'\in A$ 
such that 
\beq\label{313}
(x_n')^*(x_n')=f_{1/n}(b_n)\andeqn x_n'(x_n')^*\in \Her(f_{1/4}(e_n-w_nw_n^*))
\eneq
Note that $f_{1/n}(b)(b-1/n)_+=(b-1/n)_+.$ Choose $x_n=x_n'(b-1/n)_+^{1/2}.$ Then}} 
\beq
\label{e1025-2}
x_n^*x_n=(b_n-1/n)_+
\quad
\text{ and } 
\quad
x_nx_n^*\in \Her(f_{1/4}(e_n-w_nw_n^*)).
\eneq
Note $\|x_n\|^2=\|x_n^*x_n\|=\|(b_n-1/n)_+\|\leq 1.$
The second part of \eqref{e1025-2} implies 
\beq
\label{e1025-4}
x_nx_n^*\leq f_{1/8}(e_n-w_nw_n^*).
\eneq

Let $c_n=e_n-w^*_nw_n$ and $d_n=e_n-w_nw_n^*$ ($n\in\N$).
Let $x=\{x_1,x_2,...\},$ $b{{=}} 
\{b_1,b_2,...\}
,c=\{c_1,c_2,...\},$  and $d=\{d_1,d_2,...\}\in l^\infty(A).$ 
Then
\beq
\Pi(x)^*\Pi(x)\overset{\eqref{e1025-2}}{=}\Pi(b)
\overset{\eqref{e1025-3}}{=}
\Pi(f_{1/4}(c))
=f_{1/4}(\Pi(c))
=f_{1/4}(r-p)
=r-p,
\label{e1025-5}
\eneq
and
\beq
\Pi(x)\Pi(x)^*\overset{\eqref{e1025-4}}{\leq}
\Pi(f_{1/8}(d))
=f_{1/8}(\Pi(d))
=f_{1/8}(r-q)
=r-q.
\label{e1025-6}
\eneq
Let $y=v+\Pi(x).$
Note that $v\Pi(x)^*=vp(r-p)\Pi(x)^*=0$
and $\Pi(x)v^*=(v\Pi(x)^*)^*=0.$
{{Also note that $v^*\Pi(x)=v^*q(r-q)\Pi(x)=0$
and $\Pi(x)^*v=(v^*\Pi(x))^*=0.$}}
Then  we compute (using also 
\eqref{e1025-5} and \eqref{e1025-6})
that
$${{r=y^*y\sim yy^*= \Pi(x)\Pi(x)^*+q \leq r.}}$$
By Proposition \ref{P1025-3},
$r$ is a finite projection.
Hence $\Pi(x)\Pi(x)^*+q=r.$  Consequently,
$\Pi(x)\Pi(x)^*=r-q.$  Together with
\eqref{e1025-5}, we obtain  $r-p\sim r-q.$ 
The theorem  then follows.
\end{proof}

\section{Tracial oscillations}
%
%
%
{In this section we will introduce the notion of tracial approximate oscillation for 
positive elements in a \CA\,  and 
present some basics around the notion.}

\begin{df}\label{DefOS1}
Let $A$ be a \CA\, with ${\wtd{QT}}(A)\setminus \{0\}\not=\emptyset.$ 
Let $S\subset {\wtd{QT}}(A)$
be a compact subset. 
Define, for each $a\in  (A\otimes {\cal K})_+,$
\beq
\omega(a)|_S&=&\inf\{\sup\{d_\tau(a)-\tau(c): \tau\in S\}: c\in  {\rm Ped}(\overline{a(A\otimes {\cal K})a}), \,0\le c\le 1\}
\eneq
(see A1 of \cite{eglnkk0}).
The number $\omega(a)|_S$ is called the (tracial) oscillation of $a$ on $S.$

If $a\in {\rm Ped}(A\otimes {\cal K})_+,$ then 
$\omega(a)|_S<\infty$ (see (2) of Proposition \ref{Pboundedness}).
Since $\tau(f_{1/n}(a))\nearrow d_\tau(a)$ (point-wisely) and $\wh{c}$ is continuous on compact 
set $S$ for each $c\in {\rm Ped}(\overline{a(A\otimes {\cal K})a})_+,$
one has
\beq
\omega(a)|_S=\lim_{n\to\infty}\sup\{d_\tau(a)-\tau(f_{1/n}(a)): \tau\in S\}.
\eneq

Note that, exactly as in  A1 of \cite{eglnkk0}, if $a, b\in (A\otimes {\cal K})_+^{\bf 1}$ and 
$a\sim b,$ then $\omega(a)|_S=\omega(b)|_S$ (cf. Proposition 
\ref{Pomega-1} 
below). 
%
%
%
For each $h\in {\rm LAff}_+(\wtd{QT}(A)),$ define 
\beq
\omega(h)|_S=\inf\{\sup\{h(\tau)-f(\tau): \tau\in S\}:  0<f<h, \, f\in \Aff_+(\wtd{QT}(A))\}.
\eneq
{{Recall that, in general, for any real function $f$ defined on $S,$ 
the oscillation of $f$ at $s\in S$ is defined as 
\beq
\omega(f)(s)=\inf\{ \sup\{|f(s')-f(s'')|: s',s''\in O(s)\}: O(s)\},
\eneq
where $O(s)$ is an open neighborhood of $s$ and the infimum  above is taken among 
all such $O(s).$
Denote by $\omega(f)|_S=\sup\{\omega(f)(s): s\in S\}.$
}}
Then
(recall that $S$ is  compact)
\beq\label{Eosc-1226}
\omega(a)|_S={{\omega(\wh{[a]})|_S}}
\eneq

Let $a\in (A\otimes {\cal K})_+.$ 
For each $\tau\in S,$  and its neighborhood {{$O(\tau),$}}  define
\beq
\omega_{O(\tau)}(a)|_S&=&\lim_{n\to\infty}\sup\{d_t(a)-t(f_{1/n}(a)): t\in O(\tau)\cap S\}.
\eneq
{{One may note that, if $O_1(\tau)\subset O_2(\tau),$
then $\omega_{O_1(\tau)}(a)|_S\le \omega_{O_2(\tau)}(a)|_S.$}}
Define
\beq
\omega(a)(\tau)|_S&=&\inf\{\omega_{O(\tau)}(a): \tau\in O(\tau)\cap S\}
\eneq
(the infimum is taken among all neighborhood $O(\tau)$ of $\tau$
in $S$). 
In other words,  when $S$ is fixed, $\omega(a)(\tau)|_S$ is the oscillation 
of the lower-semicontinuous function $\widehat{[a]}$ at $\tau.$ {{In particular, 
$\widehat{[a]}$ is continuous on $S$ if and only if $\omega(a)(\tau)|_S=0$ for all $\tau\in S.$}}

\end{df}

(1) If $c_n\in {{{\rm Her}(a)_+^{\bf 1}}}$ 
and $\tau(c_n)\nearrow d_\tau(a)$ for all $\tau\in S,$ then
\beq
\omega_{O(\tau)}(a)|_S&=&\lim_{n\to\infty}\sup\{d_t(a)-t(c_n): t\in O(\tau)\cap S\}.
\eneq

In general, 
one checks that
\beq
\label{F109-2}
\sup\{\omega(a)(\tau)|_S: \tau\in S\}=\omega(a)|_S.
\eneq

(2) For most of the time, we will assume that $A$ is simple
and $S$ is a compact subset  {{of}} ${\wtd{QT}}(A)\setminus \{0\}$ such that
${{\R_+\cdot S}}={\wtd{QT}}(A),$ {{for example,}} $S=T_b$
for some $b\in {\rm Ped}(A)_+\setminus \{0\}.$
{{Or,}} in the case that $A={\rm Ped}(A),$ $S=\Qw.$
When $S$ is understood, we may omit $S$ in the notation.
{{In fact, when $A$ is compact, we may write $\omega(a)$ instead of $\omega(a)|_{\Qw}.$}}

 (3) Let  $S_1, S_2\subset \wtd{QT}(A)\setminus \{0\}$ be compact {{subsets}}
such that $\R_+\cdot S_i=\wtd{QT}(A),$ {$i=1,2.$}
If $\omega(a)|_{S_1}=0,$ then $\omega(a)|_{S_2}=0$ (see also \ref{Pboundedness}).
In what follows we write $\omega(a)=0$ if $\omega(a)|_{S}=0$ 
for one compact subset of $\wtd{QT}(A)$ such that $\R_+\cdot S=\wtd{QT}(A).$

\begin{prop}{{{\rm{[\cite[A1]{eglnkk0}}]}}}\label{Pomega-1}
Let $a,b\in (A\otimes {\cal K})_+.$
Let $S\subset \wtd{QT}(A)$ 
be a compact subset.
If $a\sim b,$ then $\omega(a)(\tau)|_S=\omega(b)(\tau)|_S$ 
for all $\tau\in S,$
{{and $\omega(a)|_S=\omega(b)|_S.$}}
\end{prop}

{{\begin{proof}
Let $\tau\in S.$
Let $O(\tau)$ be any open neighborhood 
of $\tau.$
For any $\ep>0,$ 
there is $\dt>0$ 
such that 
\beq
\sup\{d_t(a)-t(f_{\dt}(a)):t\in O(\tau)\cap S\}<\omega_{O(\tau)}(a){{|_S}} +\ep.
\eneq
Since $a\sim b,$  there exists a sequence 
$x_n\in A\otimes {\cal K}$  such that $x_nx_n^*\to a$ and 
$x_n^*x_n\in {\rm Her}(b)=\overline{b(A\otimes {\cal K})b}.$ 
{{Since $a^{1/m}aa^{1/m}\to a$ as $m\to\infty,$
replacing $x_n$ by $a^{1/m(n)}x_n$ for some subsequence $\{m(n)\},$
we may assume that $x_nx_n^*\in \Her(a).$}}
{{Note that, for any $\dt>0,$
$$ 
\lim_{n\to\infty}\|f_\dt(x_nx_n^*)-f_\dt(a)\|=0.
$$}} 
{{Since $S$ is compact,  by (2) of Proposition \ref{Pboundedness}, 
$\sup\{\|t|_{\Her(a)}\|:t\in S\}<\infty.$ It follows that}}
there is $m\in\N$ such that 
\beq
\sup\{|t(f_\dt(a))-t(f_\dt(x_mx_m^*))|:t\in S\}<\ep.
\eneq
Note that $d_t(a)=d_t(b)$ for all $t\in S$
because of $a\sim b.$ 
Also note that 
$t(f_\dt(x_mx_m^*))=t(f_\dt(x_m^*x_m))$
for all $t\in S.$ Then 
\beq
\omega_{O(\tau)}(b)|_S
&\leq & \sup\{d_t(b)-t(f_\dt(x_m^*x_m)):t\in O(\tau)\cap S\}\\
&=& \sup\{d_t(a)-t(f_\dt(x_mx_m^*)):t\in O(\tau)\cap S\}\\
&\leq &
\sup\{d_t(a)-t(f_\dt(a)):t\in O(\tau)\cap S\}
\\
&&\quad + \sup\{|t(f_\dt(a))-t(f_\dt(x_mx_m^*))|:
t\in O(\tau)\cap S\}\\
&\leq & \omega_{O(\tau)}(a){{|_S}}  +\ep +\ep. 
\eneq
Since $\ep$ is arbitrary, we have 
$\omega_{O(\tau)}(b){{|_S}}\leq \omega_{O(\tau)}(a){{|_S}}.$
Exactly the same argument {{shows}} that 
$\omega_{O(\tau)}(a)|_S\leq \omega_{O(\tau)}(b){{|_S}}.$
Hence $\omega_{O(\tau)}(a)|_S=\omega_{O(\tau)}(b){{|_S}}.$
Since $O(\tau)$ is an arbitrary open neighborhood of $\tau,$ we have 
$$
\omega(b)(\tau)|_S
=\inf \{\omega_{O(\tau)}(b)|_S:\tau\in O(\tau){{\cap S}}\}
=
\inf \{\omega_{O(\tau)}(a){{|_S}}:\tau\in O(\tau){{\cap S}}\}
=\omega(a)(\tau)|_S.
$$
{{For the last identity in the proposition, we note that,}} by  \eqref{F109-2},
$$\omega(a)|_S
=\sup\{\omega(a)(\tau)|_S: \tau\in S\}
=\sup\{\omega(b)(\tau)|_S: \tau\in S\}=\omega(b)|_S.$$
\end{proof}

\begin{df}\label{DefOs2}
In the case that $A$ does not have strict comparison, we may still want to consider 
elements with  zero tracial oscillation. 
We write  $\omega^c(a)=0$ if 
$g_{1/n}(a)\cto 0$ (recall  Definition \ref{Nfg} for $g_\dt,$ and also see \ref{Dcnull-1}).
Let $\{a_n\}\in l^\infty(A)_+.$ We write 
$\lim_{n\to\infty} \omega^c(a_n)=0,$ if there exists $\dt_n\in (0, 1/2)$
such that $g_{\dt_n}(a_n)\cto 0.$
\end{df}

Note that, by Proposition \ref{Pboundedness}, if $A$ is compact, then the number 
$s$ in part (1) of the next proposition  is always finite. 
{{Let $\tau\in S.$  In the next lemma, we write  $O(\tau)$  for an open neighborhood of $\tau$ in $S.$}} 
%
\begin{prop}\label{Pbfos}
Let $A$ be a $\sigma$-unital 
{{\CA. 
Let}} $S\subset {\wtd{QT}}(A)\setminus \{0\}\not=\emptyset$ be a 
compact subset.

(1)  Suppose that $s:=\{\|\tau|_A\|: \tau\in S\}<\infty.$
 If $a, b\in (A\otimes {\cal K})_+^{\bf 1},$ 
then 
\beq
\omega(a)(\tau)|_S-\overline{d_\tau}(b)|_S\le \omega(a+b)(\tau)|_S \le \omega(a)(\tau)|_S+\overline{d_\tau}(b)|_S\tforal  \tau\in  S,
\eneq
where 
$\overline{d_\tau}(b)|_S:=
\inf\{\sup\{d_t(b):t\in O(\tau)\}: 
O(\tau)\,\, \mbox{open neighborhoods of}\,\, \tau\,\, {\mbox{in}} \,\,S\}.$

(2) If $a\perp b,$ 
then 
\beq\label{Pbfos-orth}
\max\{\omega(a)|_S, \omega(b)|_S\}\le \omega(a+b)|_S\le \omega(a)|_S+\omega(b)|_S.
\eneq
{Moreover,}
\beq\label{Pbfos-2}
\max\{\omega(a)(\tau)|_S, \omega(b)(\tau)|_S\}\le\omega(a+b)(\tau)|_S\le \omega(a)(\tau)|_S+\omega(b)(\tau)|_S\tforal \tau\in S.
\eneq

(3) For $\sigma$-unital simple \CA\, $A,$ $\omega^c(a)=0$ if and only if ${\rm Her}(a)$ has continuous scale.

\end{prop}

\begin{proof}

(1) For the  inequality on the left, 
let $\ep>0.$ {{Fix $\tau\in S.$}}
{{Choose an open  neighborhood}}  $O(\tau)$ 
of $\tau$ {{in $S$}} such that 
\beq
\label{F107-3}
&&\omega_{O(\tau)}(a+b)|_S \le
\omega(a+b)(\tau)|_S+\ep, 
\mbox{ and }\\ 
&&\sup\{d_t(b):t\in O(\tau)\}
\le 
\overline{d_\tau}(b)|_S+\ep.
\label{F107-4}
\eneq
Note that  there is $\dt>0$ such that 
\beq\label{F107-1}
\sup\{d_t(a+b)
-t(f_\dt(a+b)):t\in O(\tau)\}
\le \omega_{O(\tau)}(a+b){{|_S}}+\ep.
\eneq
Note that $a+b\approx_{\dt/2} (a-\dt/2)_++b.$
By \cite[Proposition 2.2]{Ror92UHF2}, 
we have $f_{\dt}(a+b)\lesssim (a-\dt/2)_++b.$
Then, for any $t\in {{O(\tau)}},$ we have 
\beq
\hspace{-0.2in}
t(f_{\dt}(a+b))&\leq &
d_t(f_{\dt}(a+b)) \leq  d_t((a-\dt/2)_++b)
\leq 
d_t((a-\dt/2)_+)+d_t(b)
\\&\overset{\eqref{F107-4}}{\leq}  &
t(f_{\dt/2}(a))+\overline{d_\tau}(b)|_S+\ep.
\eneq
Then, {{for $t\in O(\tau),$}} 
$d_t(a)+t(f_{\dt}(a+b)) \leq
d_t(a+b)+t(f_{\dt/2}(a))+\overline{d_\tau}(b)|_S+\ep.$
It follows that 
\beq
d_t(a)-t(f_{\dt/2}(a))
&\leq& 
d_t(a+b)-t(f_{\dt}(a+b))
+\overline{d_\tau}(b)|_S+\ep
\\
&\overset{\eqref{F107-1}}{\leq} & 
\omega_{O(\tau)}(a+b)(\tau)|_S
+\overline{d_\tau}(b)|_S+2\ep
\\&\overset{\eqref{F107-3}}{\leq} &
\omega(a+b)(\tau)|_S
+\overline{d_\tau}(b)|_S+3\ep\,\,\,{{\rforal t\in O(\tau).}}
\eneq
Hence 
\beq
\omega(a)(\tau)|_S
&\leq&
\omega_{O(\tau)}(a)(\tau)
\leq 
\sup\{d_t(a)-t(f_{\dt/2}(a)):t\in O(\tau)\}
\\
&\leq&
\omega(a+b)(\tau)|_S
+\overline{d_\tau}(b)|_S+3\ep. 
\eneq
Let $\ep\to 0,$ then we have the desired inequality. 

Now we turn to the inequality on the right. 
By definition, for any $\ep>0,$
there are 
{{open neighborhood $O(\tau)$  of $\tau$ in $S,$}} and  $\dt>0$ 
such that 
\beq
\label{F107-5}
&&
\sup\{d_t(b):t\in O(\tau)\}
\le \overline{d_\tau}(b)+\ep, \mbox{ and }
\\
\label{F107-6}&&
\sup\{d_t(a)-t(f_\dt(a)):t\in O(\tau)\}{{\le}}
\omega_{O(\tau)}(a){{|_S}}+\ep/2
\le \omega(a)(\tau)|_S+\ep.
\eneq
Note that $a\in \Her(a+b),$
then  
there is $\eta>0$ such that 
$f_\dt(a)\approx_{\ep/(s+1)} f_\eta(a+b)f_\dt(a) f_\eta(a+b).$ 
Hence,  
for any $t\in O(\tau),$ by \cite[(iii) Corollary II.2.5]{BH},
\beq\nonumber
d_t(a+b)-t(f_\eta(a+b))
&\leq &
d_t(b)+d_t(a)-t(f_\eta(a+b)f_\dt(a) f_\eta(a+b))
\\
&\leq &
d_t(b)+d_t(a)-t(f_\dt(a))+\ep
\\&\overset{\eqref{F107-5},
\eqref{F107-6}}{\leq}&
\overline{d_\tau}(b)+\omega(a)(\tau)|_S
+3\ep.
\eneq
Hence 
\beq\nonumber
\omega(a+b)(\tau)|_S &\leq&
\sup\{d_t(a+b)-t(f_\eta(a+b)):t\in O(\tau)\}\\
&\leq& \overline{d_\tau}(b)+\omega(a)(\tau)|_S+3\ep.
\eneq
Let $\ep\to 0,$ (1) then follows.

For (2),  we have, for any $1/2>\ep>0,$ since $a\perp b,$ 
\beq
d_\tau(a+b)-\tau(f_\ep(a+b))=(d_\tau(a)-\tau(f_\ep(a))+(d_\tau(b)-\tau(f_\ep(b))
\eneq
 for all $\tau\in S.$  Thus  
 \beq
 &&\omega(a+b)|_S\le \omega(a)|_S+\omega(b)|_S\andeqn\\
 &&\omega(a+b)(\tau)|_S\le \omega(a)(\tau)|_S+\omega{{(b)}}(\tau)|_S\rforal \tau\in S.
 \eneq
 {{Hence 
  the inequality on the right {{in}} \eqref{Pbfos-orth} holds.}}

 {{ 
Now we turn to  the inequality on the left of \eqref{Pbfos-orth}. 
Since $a\bot b,$  for all $\tau\in S$ and all $\eta>0,$ 
 \beq
 \hspace{-0.2in}
 d_\tau(a)-\tau(f_\eta(a))\leq
 ( d_\tau(a)-\tau(f_\eta(a)))+
 ( d_\tau(b)-\tau(f_\eta(b)))
 =
 d_\tau(a+b)-\tau(f_\eta(a+b)).
 \eneq
Thus  $
 \omega(a)|_S
 \leq 
  \sup\{d_\tau(a)-f_{\eta}(a): \tau\in S\}
  \leq 
   \sup\{d_\tau(a+b)-f_{\eta}(a+b): \tau\in S\}.
$
Since $\eta$ can be arbitrary small, we have 
\beq
\omega(a)|_S\le 
 \inf_{\eta>0}\sup\{d_\tau(a+b)-f_{\eta}(a+b): \tau\in S\}=\omega(a+b)|_S.
 \eneq
Similarly, $\omega(b)|_S\le \omega(a+b)|_S.$ 
Thus the inequality on the left of \eqref{Pbfos-orth} holds. 
 The estimates \eqref{Pbfos-2} can be checked similarly. 
}}

 For (3), recall that $\Her(a)$ has continuous scale if 
 only if $e_{m(n)}-e_n\cto 0$ for any $m(n)>n,$  where $e_n=f_{1/2^n}(a),$ $n\in \N.$
 {{Suppose that $\omega^c(a)=0.$ Then,}} 
 for each $n\in \N$ and any  $m(n)\ge n,$
 \beq
 e_{m(n)}-e_n\lesssim g_{1/n}(a)\cto 0.
 \eneq
 It follows that $\Her(a)$ has continuous scale.
 
 Conversely, suppose that $\Her(a)$ has continuous scale. 
 For any $d\in A_+\setminus \{0\},$ 
 choose $n_0\in \N$ such that, for any $m(n)>n\ge n_0,$ 
 \beq
 e_{m(n)}-e_n\lesssim d.
 \eneq
 Suppose that $k_0>2^{n_0}.$ Fix $k\ge k_0.$  {{For}} any $\ep\in (0,1/4),$
 there is $m(n)>n_0$ such that
 \beq
 f_\ep(g_{1/k}(a))\lesssim e_{m(n)}-e_{n_0}.
 \eneq
In other words, for any $\ep\in (0, 1/4),$
$
f_\ep(g_{1/k}(a))\lesssim d.
$
 It follows that $g_{1/k}(a)\lesssim d$ (for any $k\ge k_0$). This proves (3).
 %
\end{proof}

\begin{lem}\label{Los-dt}
Let $A$ be a 
\CA\, with a nonempty compact 
subset $S\subset \wtd{QT}(A)$ 
and let
$a\in {{(A\otimes {\cal K})_+.}}$
Then, for  any $\ep>0,$ there exists $\dt_0>0$ such that
\beq
\omega(f_\dt(a))|_S<\omega(a)|_S+\ep {{\tforal}} 0<\dt<\dt_0.
\eneq

\end{lem}

\begin{proof}
{{We may assume that $\omega(a)|_S<\infty.$}}
There exists $\dt_0>0$ such that, for all $0<\eta\le 2\dt_0,$  
\beq
d_\tau(a)-\tau(f_{\eta}(a))<\omega(a)+\ep/2\rforal \tau\in S.
\eneq
Then, there exists $\sigma_0>0$ such that, if $0<\dt<\dt_0,$
\beq
d_\tau(f_\dt(a))-\tau(f_{\sigma_0}(f_\dt(a)))\le d_\tau(a)-\tau(f_{2\dt_0}(a))<\omega(a)+\ep/2
\eneq
for all $\tau\in S.$
Note that $d_\tau(f_\dt(a))\le \tau(f_{\dt/2}(a))$ for all $\tau\in \wtd{QT}(A).$
It follows that (see (2) of Proposition \ref{Pboundedness})
\beq\nonumber
\omega(f_\dt(a))|_S<\omega(a)|_S+\ep.
\eneq
\end{proof}

\begin{prop}\label{Ppartomega}
Let $A$ be a \CA\, with a nonempty compact  subset $S\subset \wtd{QT}(A)$ and 
$a\in {{(A\otimes {\cal K})_+^{\bf 1}}}$
with $\omega(a)<\infty,$ and $0<\dt<1/2.$
Then, for any $\ep>0,$ there is $0<\eta<\dt/2$ and $n_0\in\N$ such that
\beq
&&\sup\{\tau(f_\eta(a))-d_\tau(f_\dt(a)):\tau\in   S
\} \ge \omega(a)|_S-\ep\tand\\
&&\sup\{\tau(a^{1/n_0})-d_\tau(f_\dt(a)):\tau\in S\}\ge \omega(a)|_S-\ep.
\eneq
\end{prop}

\begin{proof}
 Fix $0<\dt<1/2.$ For any $\ep>0,$ 
  there exists  $\tau_0\in S$
  such 
  that
  \beq
  d_{\tau_0}(a)-\tau_0(f_{\dt/2}(a))>\omega(a)|_S-\ep/4.
  \eneq
{{For this $\tau_0,$}} choose $0<\eta<\dt/2$ such  that
 $
 d_{\tau_0}(a)-\tau_0(f_\eta(a))<\ep/4.
 $
Then
\beq
\tau_0(f_\eta(a))-d_{\tau_0}(f_\dt(a))
&>& d_{\tau_0}(a)-{\tau_0}(f_{\dt/2}(a))-( d_{\tau_0}(a)-\tau_0(f_\eta(a)))\\
&>&\omega(a)|_S-\ep/4-\ep/4.
\eneq
Hence
\beq\label{56-add-1}
\sup\{{{\tau(f_\eta(a))}}-d_\tau(f_\dt(a)):\tau\in   S
\} \ge \omega(a)|_S-\ep/2.
\eneq
To see the second inequality, choose 
 $n_0\in \N$ such that 
\beq
\|a^{1/n_0}f_\eta(a)-f_\eta(a)\|<\ep/4.
\eneq
It follows that, for all $\tau\in S,$
\beq
\tau(a^{1/n_0})-d_\tau(f_\dt(a))&\ge& \tau(a^{1/n_0}f_\eta(a))-d_\tau(f_\dt(a))\\
&\ge& \tau(f_\eta(a))-d_\tau(f_\dt(a))-\ep/4.
\eneq
Therefore, by \eqref{56-add-1},
$
\sup\{\tau(a^{1/n_0})-d_\tau(f_\dt(a)): \tau\in S
\}\ge \omega(a)|_S-\ep.
$
%
\end{proof}

\begin{df}\label{defOs2}
Let $A$ be a \CA, $S\subset {\wtd{QT}}(A)
$ be a compact 
subset  and 
let $a\in  {{(A\otimes {\cal K})_+.}}$
Put $B=\Her(a)$ and $I_{_{S,B}}=\{\{b_n\}\in l^\infty(B): \lim_{n\to\infty}\|b_n\|_{_{2, S}}
=0\}.$
Denote by $\Pi_S: l^\infty(B)\to l^\infty(B)/I_{_{S,B}}$ and 
$\Pi_{cu}: l^\infty(B)\to l^\infty(B)/N_{cu}(B)$
(in the case that $B$ has no one-dimensional hereditary \SCA)
the quotient maps, respectively.

{{Let}} $A$ {{be}} a $\sigma$-unital \CA\, and  {{$a\in (A\otimes {\cal K})_+$
with $\|a\|_{_{2, S}}<\infty.$}}
 Define 
(here we assume that $b_n\in {\rm Ped}(A\otimes {\cal K})_+$ {{and $\Her(a)=\overline{a(A\otimes {\cal K})a}$}}\,)
\beq\nonumber
\Omega^T(a)|_S&=& \inf\{\|\Pi_S(\iota(a)-\{b_n\})\|: \{b_n\}\in l^\infty(\Her(a))_+, \|b_n\|\le \|a\|, \,\, \lim_{n\to\infty}\omega(b_n)|_S=0\},\\\nonumber
\Omega^T_{T}(a)|_S&=& \inf\{\limsup_{n}\|a-b_n\|_{_{2, S}}: b_n\in {\rm Her}(a)_+, \|b_n\|\le \|a\|, \lim_{n\to\infty}\omega(b_n)|_S=0\},\\\nonumber
\Omega_C^T(a)|_S&=&\inf\{\|\Pi_{cu}(\iota(a)-\{b_n\}{{)}}\|: \{b_n\}\in l^\infty(\Her(a))_+, \|b_n\|\le \|a\|, \lim_{n\to\infty} \omega(b_n)|_S=0\},
\\\nonumber
\Omega^{C}(a)&=&\inf\{\|\Pi_{cu}(\iota(a)-\{b_n\}{{)}}\|: \{b_n\}\in l^\infty({\rm Her}(a))_+, \,\, \|b_n\|\le \|a\|,\,\,\lim_{n\to\infty} \omega^c(b_n)=0\}\andeqn\\
\Omega^N(a)|_S&=&\inf \{\|\pi_\infty(\iota(a)-\{b_n\})\|: b_n\in l^\infty({\rm Her}(a))_+, \lim_{n\to\infty} \omega(b_n)|_S=0\}.
\eneq

{{We will focus on $\Omega^T(a).$}}

{{ (1)  
Note, for the convenience,   in the definition above, 
we always assume that $\|a\|_{_{2, S}}<\infty.$}}

{{(2) Note also that $\lim_{n\to\infty}\|af_{1/n}(a)-a\|=0$ and 
$${{af_{1/n}(a)=a^{1/2}f_{1/n}(a)a^{1/2}\sim f_{1/n}(a),\,\,\, n\in\N.}}$$
Hence,
if $\omega(a)|_S=0,$  by Lemma \ref{Los-dt},
then $\Omega^N(a)|_S=\Omega^T(a)|_S={{\Omega^T_T(a)|_S}}=0.$}}

One may call $\Omega^T(a)|_S$ the tracial approximate oscillation of $a$ (on $S$).
If $\Omega^T(a)|_S=0,$ we say $a$ has  approximately tracial oscillation zero (on $S$).
Often, when $S$ is understood, we may omit $S$ in notation above.
In particular, when $A$ is algebraically simple, we write $\Omega^T(a):=\Omega^T(a)|_{_{\Qw}}.$

{{(3) It is, perhaps, convenient to use (1) and (2) of Proposition \ref{Omega-1} for the definition 
of $\Omega^T(a)|_S=0.$}} 
We would like to mention that, for the definition of  $\Omega^T_C(a)$ and $\Omega^C(a),$ 
we  also require that \CA\, $A$ has no one dimensional hereditary \SCA s (see \ref{DNcu}).

(4) Moreover, since $\omega(0){{|_S}}=0,$  we have 
{{\beq
\Omega^T(a)|_S\le \|\Pi_S(\iota(a))\|\le {{\|a\|,}} \andeqn
 \Omega_C^T(a)|_S,\,\Omega^C(a),\,\Omega^N(a)\le \|a\|.
 \eneq
 }}
When $A$ is unital, $\Omega^T(a){{|_S}}=0$ for any $a\in GL(A)\cap A_+,$ since 
$\omega(1_A)=0.$ 

{{(5) In the case that $A$ is a $\sigma$-unital algebraically  simple \CA\, with strict comparison, 
if  $S=\Qw$ and  
$\Omega^T_T(a)|_S=0,$ then $\Omega^T(a)|_S=\Omega^N(a)|_S=\Omega^T_C(a)|_S=\Omega^C(a)=0$
(see  (2) of Proposition \ref{Omega-1}, (2) after Definition \ref{DTos-2} and Proposition \ref{PN=T}).}}
\end{df}

\begin{prop}\label{Omega-1}
Let $A$  be a 
\CA, ${{a\in( A\otimes {\cal K})_+^{\bf 1}}}$
and $S\subset \wtd{QT}(A)$ a compact subset
{{such that $\|a\|_{_{2, S}}<\infty.$}}

(1) If $\Omega^T(a)|_S=0,$ then 
there exists 
a sequence
$\{b_n\}\subset\ 
{{{\rm Ped}(\Her(a))_+^{\bf 1}}}$ 
such that
\beq
\lim_{n\to\infty}\omega(b_n)|_S=0  {{\tand}}
\|\Pi_S(\iota(a)-\{b_n\})\|=0,
\eneq
 and, if $\Omega^T_T(a)|_S=0,$ there exists $b_n\in {\rm Ped}(\Her(a))_+^{\bf 1},$ 
 $n\in \N,$ such that 
 \beq\label{PT=TT}
\lim_{n\to\infty}\omega(b_n)|_S=0  {{\tand}}
\lim_{n\to\infty}\|a-b_n\|_{_{2, S}}=0.
\eneq

(2) $\Omega^T(a)|_S=0$ 
if and only if  $\Omega_T^T(a)|_S=0.$

(3) If there exists $M\ge 1$ such that
$$\inf\{ \|\Pi_S(\iota(a)-\{b_n\})\|: 
b_n\in {{{\rm Ped}({\rm Her}(a))_+}},\, \|b_n\|\le M, \,\, \lim_{n\to\infty}\omega(b_n)|_S=0\}=0,$$
then  
$\Omega^T(a)|_S=0.$

{{(4) If $\{a_n\}\in (I_{_{S}})_+^{\bf 1},$ then there is $\{b_n\}\in (I_{_{S}})^{\bf 1}_+, $
$n\in \N,$ such that\\
{{ \beq
 \lim_{n\to\infty}\sup\{d_\tau(b_n): \tau\in S\}=0\andeqn
\lim_{n\to\infty} \|a_n-b_n\|=0.
\eneq}}
}} 

\end{prop}

\begin{proof}
Recall that  $B={{\Her(a)}}$ and  $\Pi_S: l^\infty(B)\to l^\infty(B)/I_{_{S,B}}$ is the quotient map.

For (1), there is, for each $k\in \N,$ a sequence 
$\{c_n^{(k)}\}\in l^\infty({{{\rm Ped}(\Her(a))}})$ with $0\le c_n^{(k)}\le 1$ and $\lim_{n\to\infty}\omega(c_n^{(k)})|_S=0$
such that
\beq
\|\Pi_S(\iota(a)-\{c_n^{(k)}\})\|<1/k.
\eneq
Therefore, for each $k\in \N,$ there is $n(k)\in \N$ such that
\beq
\|a-c_{n(k)}^{(k)}\|_{_{2,S}}
<2/k\andeqn \omega(c_{n(k)}^{(k)})<1/k.
\eneq
Define $b_k=c_{n(k)}^{(k)},\,\, k\in \N.$ Then $0\le b_k\le 1,$ $b_k\in {{{\rm Ped}(\Her(a))}}$ and 
$\lim_{n\to\infty}\omega(b_n)|_S=0.$ Moreover
\beq
\|a-b_n\|_{_{2, S}}<2/n\rforal n\in \N.
\eneq
It follows that ${{\|}}\Pi_S(\iota(a)-\{b_n\})\|=0,$ and \eqref{PT=TT} holds.
A similar proof above shows that, if $\Omega^T_T(a)=0$ implies that there is $b_n\in 
{\rm Ped}(\Her(a))_+^{\bf 1}$ such that \eqref{PT=TT} holds.

For (2),  we note
that \eqref{PT=TT} implies that $\iota(a)-\{b_n\}\in I_{_{S}}.$
So, if $\Omega^T_T(a)|_S=0,$  {{then}}
$\|\Pi_S(\iota(a)-\{b_n\})\|=0.$ Hence 
$\Omega^T(a)|_S=0.$  The converse also holds.

To show (3) holds, 
suppose that  there are $c_n^{(k)}\in {{{\rm Ped}}}(\Her(a))$ such that  $0\le c_n^{(k)}\le M,$ 
$\lim_{n\to\infty}\omega(c_n^{(k)}){{|_S}}=0$ and 
\beq
\|\Pi_S(\iota(a)-\{c_n^{(k)}\}{{)}}\|<1/k\rforal k\in \N.
\eneq
By the proof of (1),  one obtains $b_n\in {{{\rm Ped}}}(\Her(a))_+$ with $0\le b_n\le M$ such that
\beq
\lim_{n\to\infty}\omega(b_n)|_S=0\andeqn
\Pi_S(\iota(a))=\Pi_S(\{b_n\}).
\eneq
Define $g\in C_0((0, \|a\|+M])$ by $g(t)=\|a\|$ if $t\in [\|a\|, \|a\|+M]$ and $g(t)=t$ if $t\in [0,\|a\|].$ 
Since $M$ is fixed and $g(a)=a,$ we have
\beq\label{56-n2}
\Pi_S(a)=\Pi_S(\{g(b_n)\}). 
\eneq
Put $c_n=g(b_n).$ Then $c_n\in {{{\rm Ped}}}(\Her(a))_+$  and $\|c_n\|\le \|a\|.$ Since $c_n=g(b_n)\sim b_n,$
{{then}} $\lim_{n\to\infty}\omega(c_n){|_{S}}=0.$ Therefore, by \eqref{56-n2}, $\Omega^T(a){{|_S}}=0.$

(4) Since $\{a_n\}\in (I_{_{S}})_+^{\bf 1},$ $\lim_{n\to\infty}\|a_n\|_{_{2, S}}=0.$
Choose $\dt_n:=\sqrt{\|a_n\|_{2,S}+1/n}$ and $b_n=(a_n-\dt_n)_+,$ $n\in \N.$
Then $\|b_n\|\le 1,$ $n\in \N,$ 
$\|b_n\|_{_{2, S}}\le \|a_n\|_{_{2, S}}\to 0$ and $\lim_{n\to\infty}\|a_n-b_n\|=0.$

Note that 
{{\beq
f_\eta((x-\dt_n)_+^2)\le \chi_{[\dt_n,+\infty)}(x)\le (1/\dt_n^2)x^2\rforal x \in\mathbb{R}_+,
\eneq}} 
where $\eta\in(0,1)$ and $\chi_{[\dt_n,+\infty)}$ is 
the characteristic function of $[\dt_n,+\infty).$
 Then 
for all $\tau\in S,$ 
\beq
d_\tau( b_n)=d_\tau(b_n^2)
=\sup_{\eta>0}\tau(f_\eta({{(a_n-\dt_n)_+^2}}))
\le 
 (1/\dt_n^2)\tau(a_n^2)
\le  (1/\dt_n^2) \|a_n\|^2_{_{2,S}}
=\|a_n\|_{_{2, S}}.
\eneq
It follows that $\lim_{n\to\infty}\sup\{d_\tau(b_n):\tau\in S\}=0.$
\end{proof}

{{The next proposition also justifies that we often write 
$\Omega^T(a)=0$ instead of $\Omega^T(a)|_S$ for some compact set $S\subset \wtd{QT}(A)\setminus \{0\}$
such that $\R_+\cdot S=\wtd{QT}(A).$}}

\begin{prop}\label{PMfinite}
Let $A$ be a \CA\, and $S_1, S_2\subset \wtd{QT}(A)\setminus \{0\}$ be 
nonempty compact subsets such that $\R_+\cdot S_1=\R_+ \cdot S_2=\wtd{QT}(A).$
Suppose that $a \in {{(A\otimes {\cal K})_+^{\bf 1}}}.$
Then 
$\Omega^T(a)|_{S_1}=0$ ($\Omega^T_C(a)|_{S_1}=0,$ or 
$\Omega^N(a)|_{S_1}=0$) if and only if $\Omega^T(a)|_{S_2}=0$
($\Omega^T_C(a)|_{S_2}=0,$ or 
$\Omega^N(a)|_{S_2}=0$).
Moreover, if ${{a\in {\rm Ped}(A\otimes {\cal K})_+}}$ and $\Omega^T(a)|_{S_a}=0,$ 
where $S_a=\overline{\{\tau\in \wtd{QT}(A): \|\tau|_{\Her(a)}\|=1\}}^w,$
then $\Omega^T(a)|_{S_1}=0.$
\end{prop}
\begin{proof}
It follows from  (1) of Proposition \ref{Pboundedness} that 
there is $L\in \R_+$ such {{that}}  
$$S_2\subset \{r s: s\in S_1 \andeqn r\in [0, L]\}.$$
If  $\Omega^T(a)|_{S_1}=0,$  then,  by Proposition \ref{Omega-1},
there exists a sequence $b_n\in {{{\rm Ped}(\Her(a))}}$ 
with $\|b_n\|\le \|a\|$
such that  
 $\lim_{n\to\infty}\omega(b_n)|_{S_1}=0$ {{and}} 
$\lim_{n\to\infty}\|a-b_n\|_{_{2, S_1}}=0.$
It follows that 
$$
\lim_{n\to\infty}\omega(b_n)|_{S_2}\le  \lim_{n\to\infty}L\cdot \omega(b_n)|_{S_1}=0\andeqn \lim_{n\to\infty}\|a-b_n\|_{_{2, S_2}}\le \lim_{n\to\infty}L\|a-b_n\|_{_{2, S_1}}=0.$$
 Then $\Omega^T(a)|_{S_2}=0.$ 
 Exactly the same argument shows that 
 if $\Omega_C^T(a)|_{S_1}=0$ (or $\Omega^N(a)|_{S_1}=0$), then
 $\Omega_C^T(a)|_{S_2}=0$  (or $\Omega^N(a)|_{S_2}=0$).

To see the last statement, we note that $(S_1)|_{\Her(a)}$ is bounded (see (2) of Proposition \ref{Pboundedness}). {{In other words, there is $L>0$ such that
$(S_1)|_{\Her(a)}\subset \{r\cdot \tau: \tau\in S_a,\,\, r\in [0, L]\}.$}}
\end{proof}

\begin{lem}\label{Leae}
Let $A$ be a \CA,  $a\in (A\otimes {\cal K})_+$
{{and}}
$S\subset {{{\wtd{QT}}(A)}}$
 a compact subset.
Suppose 
${{e\in {\rm Her}(a)_+}}.$ 
Then 
$e\sim a^{1/2}ea^{1/2}\sim e^{1/2}ae^{1/2},$
and
\beq\label{F109-1}
&&\omega(e)(\tau)|_S={{\omega(a^{1/2}ea^{1/2})(\tau)|_S=\ }}
\omega(e^{1/2}ae^{1/2})(\tau){{|_S}}\tforal \tau\in S,{\mbox{ and }}
\\\label{F109-3}
&&\omega(e)|_S= \omega(a^{1/2}ea^{1/2})|_S= 
\omega(e^{1/2}ae^{1/2})|_S.
\eneq
\end{lem}

\begin{proof}
Since $e\in {\rm Her}(a)_+^{\bf 1},$ we compute that 
\beq
\lim_{n\to\infty}\|(a+1/n)^{-1/2} a^{1/2} e a^{1/2}(a+1/n)^{-1/2}-e\|=0.
\eneq
It follows that $e\sim a^{1/2}ea^{1/2}\sim e^{1/2}ae^{1/2}\le e.$
Therefore, by 
{{Proposition}} \ref{Pomega-1}, \eqref{F109-1} and \eqref{F109-3} hold.
\end{proof}

Let us  end this section with the following  fact.
The proof could be simplified when $QT(A)=T(A).$
Recall that when $A={\rm Ped}(A),$ we write $\omega(a)=\omega(a)|_{_{\Qw}}$ and 
{{$\Omega^T_T(a)=\Omega^T_T(a)|_{_{\Qw}}.$}}

 \begin{prop}\label{Pfirstosc}
 Let $A$  be  an algebraically simple \CA\, 
 with $QT(A)\not=\emptyset.$
 Suppose that $A$ has strict comparison and $\Gamma$ is surjective (see Definition \ref{DGamma}).
 {{Then}}, for any $a\in {\rm Ped}(A\otimes {\cal K})_+,$  
 \beq\label{Pfirstosc-1}
 \Omega^T_T(a)\le \|a\|\sqrt{\omega(a)}. 
 \eneq
 \end{prop}
 
 \begin{proof}
Fix $a\in  {\rm Ped}(A\otimes {\cal K})_+\setminus \{0\}.$ 
{Let $\ep>0.$}
If there is a subsequence $\{n_k\}\subset \N$ such that
 $f_{1/4n_k}(a)-f_{{{1/}}n_k}(a)=0,$
then,  {{$f_{1/4n_k}(a)$ is a projection. Consequently,}} 
$\omega(f_{1/4n_k}(a))=0$ for all $k\in \N.$ Note that $af_{1/4n_k}(a)\sim f_{1/4n_k}(a).$
It follows that $\omega(af_{1/4n_k}(a))=0,$ $k\in \N.$ Since 
$\lim_{k\to\infty}\|a-af_{1/4n_k}(a)\|=0,$ $\Omega^T(a)=0.$ Hence \eqref{Pfirstosc-1} holds. 

Next we  assume that for some $n_0\in \N,$ 
$f_{1/4n}(a)-f_{{1/n}}(a)\not=0$ for all $n\ge n_0.$
Moreover,
\beq\label{Pfirstosc-n1}
\sup\{d_\tau(a)-\tau(f_{1/n}(a)): \tau\in \Qw\}<\omega(a)+\ep/2{{\rforal n\ge n_0.}}
\eneq

{{For the rest of this proof,}}
we will 
assume 
$n\geq n_0.$
Since the map $\Gamma$ is surjective,   
there is $c_n\in (A\otimes {\cal K})_+^{\bf 1}$ 
such that 
$d_\tau(c_n)=\tau(f_{1/n}(a))$ for all $\tau\in \Qw.$ 
Since $\widehat{f_{1/n}(a)}$ is continuous on  $\Qw,$
$\omega(c_n)=0.$
{{Choose}} $0<\dt_n<1$ such that {{(by also Lemma \ref{Los-dt}),}}
\beq\label{571-n1}
d_\tau(c_n)-\tau(f_{\dt_n}(c_n))<1/2^n\rforal {{\tau\in}}\Qw{{\andeqn \omega(f_{\dt_n}(c_n))<1/2^n.}}
\eneq
Note that we have 
\beq
{{d_\tau(c_n)=}}\tau(f_{1/n}(a){{)}} <d_\tau(f_{1/4n}(a{{))}}\rforal \tau\in \Qw.
\eneq
Since $A$ has strict comparison, 
{{$c_n\lesssim f_{1/4n}(a).$ By}}
 Proposition 2.4 of \cite{Ror92UHF2}, there is $x_n\in A\otimes {\cal K}$ such that
\beq
x_n^*x_n=f_{{{\dt_n}}
}(c_n)\andeqn x_nx_n^*\in {\rm Her}(f_{1/4n}(a)).
\eneq
Put $b_n=a^{1/2}x_nx_n^*a^{1/2}.$  
Then   $\|b_n\| \le \|a\|$ {{and}} $b_n\in  {\rm Her}(f_{1/4n}(a)).$ 
{{By Lemma \ref{Leae} and \eqref{571-n1},
$\omega(b_n)=\omega(x_nx_n^*)=\omega(x_n^*x_n)=\omega(f_{\dt_n}(c_n))\leq 1/2^n\to 0.$}}
Note that 
\beq\label{Pfirstosc-3}
{{f_{1/8n}}}(a)(x_nx_n^*)=x_nx_n^*=x_nx_n^*
{{f_{1/8n}}}(a).
\eneq
 {{Put $a_n=a^{1/2}{{f_{1/8n}(a)}}
 a^{1/2}$ and $d_n=
 {{f_{1/8n}(a)}}-x_nx_n^*.$}} 
 Then 
 $0\le d_n\le 1$ and 
 $0\le a_n-b_n,$ $n\in \N.$
 For all $\tau\in \Qw,$ we
 {{
 compute that 
 }}
 \beq
 \hspace{-0.3in}
{{\tau((a_n-b_n)^2)}}
&=&{{\tau(a^{1/2}d_nad_na^{1/2})\le \|a\|\tau(a^{1/2}d_n^2a^{1/2})}}\\ 
    &\le &\|a\|\tau(a^{1/2}d_na^{1/2})\le \|a\|^2 \tau(
 {{d_n}}  
    )\\
    &{\stackrel{\eqref{Pfirstosc-3}}{=}}& \|a\|^2(\tau(
  {{f_{1/8n}(a)}}  
    )-\tau(x_nx_n^*))\le \|a\|^2(d_\tau(a)-\tau(x_n^*x_n))\\
    &=&\|a\|^2{{(}}d_\tau(a)-\tau(f_{1/n}(a))+\tau(f_{1/n}(a))-\tau(f_{\dt_n}(c_n)){{)}}\\
    &=&\|a\|^2{{(}}d_\tau(a)-\tau(f_{1/n}(a))+d_\tau(c_n)-\tau(f_{\dt_n}(c_n)){{)}}\\\label{Pfirstosc-10}
    &\overset{{{\eqref{Pfirstosc-n1}, \eqref{571-n1}}}}{\le} &\|a\|^2( (\omega(a)+\ep/2)+1/2^n).
 \eneq
 {{
{{Note}} that $\|a-a_n\|<1/n.$ 
  It follows that, 
  by 
{{\eqref{F-0114-1},}}  
  \eqref{Qnorm}, and \eqref{Pfirstosc-10},
 \beq
\hspace{-0.4in}
 \|a-b_n\|_{_{2, \Qw}}^{2/3}
 &\le& \|a-a_n\|^{2/3}_{_{2,\Qw}}+\|a_n-b_n\|^{2/3}_{_{2, \Qw}}\\
 &<&(1/n)^{2/3}\sup\{d_\tau(a)^{1/3}:\tau\in \Qw\}+(\omega(a)+\ep/2+1/2^n)^{1/3}.
 \eneq
 Hence (by also (2) of Proposition \ref{Pboundedness}),
 \beq\label{F110-1}
 \limsup_{n\to\infty}\|a-b_n\|_{_{2, \Qw}}\le \|a\|\sqrt{\omega(a)+\ep/2}.
 \eneq
 Let $\ep\to 0.$ We then obtain \eqref{Pfirstosc-1} (recall that $\lim_{n\to\infty}\omega(b_n)=0$).}}

 \end{proof}

\section{\CA s with tracial approximate oscillation zero}

In this section, we will introduce the notion of T-tracial approximate oscillation zero 
for \CA s with non-trivial 2-quasitraces. We also present some examples of simple 
\CA s with tracial approximate oscillation zero (see  Proposition \ref{Prr0}, Theorem \ref{Tcounableos}
and Theorem \ref{Talsr1}).

\begin{df}\label{DTos-2}
{{Let $A$ be  a}} \CA\, with ${\wtd{QT}}(A)\setminus \{0\}\not=\emptyset$
and $S\subset \wtd{QT}(A)\setminus \{0\}$ a compact convex subset 
of $A$ such that $\R_+ \cdot S={\wtd{QT}}(A).$ 
\CA\, 
$A$
 is said to have  norm approximate oscillation zero
(relative to $S$)
if for any $a\in {\rm Ped}(A\otimes {\cal K})_+,$ $\Omega^N(a)|_S=0.$ It is said to have 
 tracial approximate oscillation zero
 (relative to $S$),
if for any $a\in  {\rm Ped}(A\otimes {\cal K})_+,$ $\Omega^T_C(a)|_S=0.$ 
We say that $A$ has T-tracial approximate oscillation zero 
(relative to $S$) if $\Omega^T(a)|_S=0$ 
for all $a\in {\rm Ped}(A\otimes {\cal K})_+.$
We say  that $A$ has C-tracial approximate oscillation zero
if $\Omega^C(a)=0$
for all $a\in {\rm Ped}(A\otimes {\cal K})_+.$ 

Note that, by  Proposition \ref{PMfinite},
these definitions do not depend on 
the choices of $S.$  {{Therefore, we often omit $S$ in the notation.}}

(1) If $A$ is a $\sigma$-unital simple \CA\, and $a\in {\rm Ped}(A\otimes {\cal K})_+^{\bf 1},$ 
then $QT(\Her(a))$ may be viewed as a convex subset of $\wtd{QT}(A).$  Put $S_1=\overline{QT(\Her(a))}^w.$
Therefore (see also Proposition \ref{Pcompact}), $\Omega^T(a)|_S=0$ if and only if $\Omega^T(a)|_{S_1}=0.$

{{(2)   By Proposition \ref{Omega-1}, $\Omega^T_T(a)|_S=0$ if and only if $\Omega^T(a)|_S=0$
for all $a\in (A\otimes {\cal K})_+$ with $\|a\|_{_{2, S}}<\infty.$
By Proposition \ref{PN=T} below,  {{that}} $A$ has norm approximate oscillation zero is the same as  {{that $A$ has}} $T$-tracial 
approximate oscillation zero.}}

If $A$  is $\sigma$-unital, {{non-elementary and}} simple and has strict comparison, then,  by Proposition \ref{Pideals=}, 
for any  nonzero $a\in {\rm Ped}(A\otimes {\cal K})_+,$  one has 
$I_{_{\overline{QT(A_a)}^w}}=N_{cu}(A_a),$ where $A_a=\Her(a).$ It follows that, {{in this case,}} 
{{the notion of}}
tracial approximate oscillation zero,
{{the notion of 
 T-tracial 
approximate oscillation zero, 
the notion of 
$C$-tracial approximate oscillation zero,  and  that of 
norm approximate oscillation zero all coincide (see also Proposition \ref{PN=T}).}}

(3)  Note also that, if $A$ has (T-) tracial approximate 
oscillation zero, then $M_n(A)$ also has (T-) tracial approximate 
oscillation zero. 

If we view $\|\cdot \|_{_{2, \Qw}}$ as an  $L^2$-norm, 
 then  that $A$ has T-tracial approximate oscillation zero has an analogue  to that ``almost'' continuous functions are 
 dense in the $L^2$-norm.
 It is worth mentioning that a  $\sigma$-unital simple \CA\, has (T-) tracial approximate oscillation zero,
 if, for some $e\in {\rm Ped}(A)_+\setminus \{0\},$  $\Her(e)$ has (T-) tracial approximate oscillation zero.
\end{df}

\begin{df}\label{DOO}
Let $A$ be a $\sigma$-unital 
\CA\, 
with  
$\wtd{QT}\setminus \{0\}\not=\emptyset.$
{{Define}}
\beq
\mathbb{O}(A)=\sup\{\Omega^T(a)|_{S_a}:
a\in {\rm Ped}(A\otimes {\cal K})_+^{{\bf 1}}\},
\eneq
{{where 
$S_a=\overline{QT(\Her(a))}^w$
(see the last paragraph of \ref{DQw}).}} 
Note that, since  $\Omega^T(a)|_{_{S_a}}\le \|a\|$
for any $a\in {\rm Ped}(A\otimes {\cal K})_+^{\bf 1}$   and
for any \CA\, $A,$ one has that $0\le \mathbb{O}(A)\le 1.$ 
The number $\mathbb{O}(A)$ is called T-tracial approximate oscillation of $A.$

{{Suppose that $S\subset \wtd{QT}(A)\setminus \{0\}$ is a compact convex set such that
$\R_+\cdot S=\wtd{QT}(A).$
By (the ``Moreover" part of) Proposition \ref{PMfinite}, if $\mathbb{O}(A)=0,$ 
then $A$ has T-tracial approximate oscillation zero.  
Conversely,  
if $A$ has $T$-tracial approximate oscillation zero,
then $\mathbb{O}(A)=0$ (see (4) of Proposition \ref{Pboundedness}).}}
\end{df}

The next example shows that there are  (commutative) \CA s $A$  of stable rank one such that $\mathbb{O}(A)>0.$
By Theorem \ref{Teqiv}, if $A$  is a separable simple \CA\, which has strict comparison but does not have stable rank one, 
then $\mathbb{O}(A)>0.$ 
However, we
do not have any such examples.

\begin{exm}\label{Exoa>0}
Let $A=C([0,1]).$  {{Then $T(A)$ is compact and ${\wtd{QT}}(A)=\wtd{T}(A).$}} Let 
${{a\in}} A_+\setminus \{0\}$ be  such that 
$0\le a\le 1$ that is not invertible.  Let $G=\{t\in [0,1]: a(t)>0\}.$  Then $G$ is an open subset.
Note $0\in {\rm sp}(a).$ 
Since $A$ has no non-trivial projection, there are $t_n\in  {\rm sp}(a)$ 
with $\lim_{n\to\infty}t_n=0.$ For any $0<\dt<1/2,$ let $b=f_\dt(a).$ 
Let $s_n\in G$ such that $a(s_n)=t_n,$ $n\in \N.$
Then, for some $0<\eta<\dt/2,$ there is $s_n$ in the support of $c=f_\eta(a)-f_{\dt/2}(a).$
There is a Borel probability measure $\mu_n$ on $[0,1]$    such 
that $\mu({{\{s_n\}}})=1.$ Let $\tau_{\mu_n}$ be the tracial state induced by $\mu_n,$ 
then $\tau_{\mu_n}(f_{\eta/2}(a)-b)=1.$ 
This implies that $\omega(a)|_{S_a}=1.$ 
This also holds for any {{non-zero}} $g\in \Her(a)_+\subset A_+.$
In other words, for any $g\in \Her(a)_+,$ $\omega(g)|_{S_a}=1.$ Therefore 
$\Omega^T(a)=1.$ It follows $\mathbb{O}(A)=1.$
Recall that $A$ has stable rank one.
\end{exm}

\begin{prop}\label{Pher}
If $A$ is a simple \CA\, {{which}} has  (T-) tracial approximate 
oscillation zero, then every hereditary \SCA\, of $A$  also has (T-) tracial approximate 
oscillation zero. 
\end{prop}

\begin{proof}
This follows from the definition immediately.
\end{proof}

\begin{df}\label{DappeT}
Let $S$ be a compact subset of ${{{\wtd{QT}}}}(A)\setminus \{0\}$ 
such that $\wtd{QT}(A)={{\R_+}}\cdot S$ and 
$B\subset A$ be a hereditary \SCA. 
A sequence of elements $\{e_n\}\subset  {\rm Ped}(B)_+^{\bf 1}$ 
  is said to be tracial approximate identity  for 
  {{$B,$}} 
if, for any $b\in B,$ 
\beq
\|\Pi_{cu}(\iota(b)-\iota(b)\{e_n\})\|={{0,}}
\eneq
and $\{e_n\}$  is said to be $T$-tracial approximate identity  for $B$ (relative to $S$), if, for any $b\in B$
{{with $\|b\|_{_{2, S}}<\infty,$}}
\beq
\lim_{n\to\infty} \|b-be_n\|_{_{2, S}}
=\lim_{n\to\infty}\|b-e_nb\|_{_{2, S}}=0.
\eneq
We do not require that $\{e_n\}$ is increasing.
\end{df}

\begin{prop}\label{PapproxeT}
Let $A$ be a {{\CA,}} {{$a\in (A\otimes {\cal K})_+$ with $\|a\|_{_{2, S}}<\infty$}} and $S\subset {\wtd{QT}}(A)\setminus \{0\}$ be a compact subset.

(1) Then $\Omega^N(a)|_S=0$ if  and only 
${\rm Her}(a)$ has a (not necessarily increasing)  approximate identity  $\{e_n\}$ such that 
$\lim_{n\to\infty}\omega_n(e_n)|{{_S}}=0.$ 

(2) 
Moreover, $\Omega^T(a)|_S=0$ (or $\Omega^T_C(a)|_S=0$) if and only if
${\rm Her}(a)$  admits a $T$-tracial (or tracial)  approximate identity (relative to $S$) $\{e_n\}$ with 
$\lim_{n\to\infty}\omega(e_n)|_S=0.$

\end{prop}

\begin{proof}
For (1),  let us assume 
that $\{e_n\}$ is a (not necessarily increasing) approximate identity 
for $\Her(a)$ such that $\lim_{n\to\infty}\omega(e_n)=0.$
Then
\beq
\lim_{n\to\infty}\|a- e_n^{1/2}ae_n^{1/2}\|=0.
\eneq
By Lemma \ref{Leae},
$
\omega(e_n^{1/2}ae_n^{1/2})|_S\le \omega(e_n)|_S\to 0.
$
Thus $\Omega^N(a)|_S=0.$

Conversely,
suppose that $a\in (A\otimes {\cal K})_+^{\bf 1}$ 
and $\{b_n\}\in l^\infty({\rm Her}(a))_+$
such that
\beq
\lim_{n\to\infty}\|a-b_n\|=0\andeqn \lim_{n\to\infty}\omega(b_n)|_S=0.
\eneq
Note that $\lim_{n\to\infty}\|b_n\|=\|a\|\le 1.$
Let $g\in C([0,\infty))_+^{\bf 1}$ such that
$g(t)=t$ if $t\in [0, 1]$ and $g(t)=1$ if $t>1.$
Then, for any $n\in \N,$ $g(b_n)\sim b_n.$
 We also have $\lim_{n\to\infty}\|g(a)-g(b_n)\|=0.$ But $g(a)=a.$ 
Then, replacing $b_n$ by $g(b_n),$  we may assume that 
 that $\|b_n\|\le 1.$ 
For each $k\in \N,$ choose $
{{n_k}}$ such that 
\beq
\|b_{n_k}^{1/k}-a^{1/k}\|<1/k,\,\,k\in \N.
\eneq
Put $e_k=b_{n_k}^{1/k},\,k\in \N.$ 
Then, for any $x\in {\rm Her}(a),$
\beq
\|x-xe_k\|\le \|x-xa^{1/k}\|+\|x\|\|a^{1/k}-e_k\|\le \|x-xa^{1/k}\|+{\|x\|\over{k}}.
\eneq
This shows that $\{e_k\}$ is a (not necessarily increasing) approximate identity for ${\rm Her}(a).$
Since $e_k=b_{n_k}^{1/k},$  {{we have that $\omega(e_k)=\omega(b_{n_k})$ and}}
\beq\nonumber
\lim_{k\to\infty}\omega(e_k)|_S=\lim_{k\to\infty}\omega(b_{n_k})|_S=0.
\eneq

 For (2), let us prove one case.

Fix $a\in {\rm Ped}(A\otimes {\cal K})_+.$ \Wlog, we may assume 
that $0\le a\le 1.$
Suppose that $\{e_n\}$ is a {{T-tracial}} approximate identity of ${\rm Her}(a)$  relative to $S.$
Then
\beq\nonumber
\lim_{n\to\infty}\| a-e_na\|_{_{2, S}}&=&\lim_{n\to\infty}\| a-ae_n\|_{_{2, S}}=0\andeqn\\\nonumber
\lim_{n\to \infty}\| a-e_nae_n\|^{2/3}_{_{2, S}} &\le& \lim_{n\to\infty}(\|a-e_na\|^{2/3}_{_{2, S}}+\|e_na-e_nae_n\|^{2/3}_{_{2, S}})\\\nonumber
\hspace{-0.3in}&\le&\lim_{n\to\infty}\|e_n(a-ae_n)\|_{_{2,S}}^{2/3}\le \lim_{n\to\infty}(\|e_n^2\|^{{{1/3}}}\| a-ae_n\|_{_{2, S}}^{2/3})=0.
\eneq
By Lemma \ref{Leae}, $\omega(e_nae_n)|_S\le \omega(e_n^2)|_S=\omega(e_n)|_S\to 0.$
{{Thus $\Omega^T(a)|_S=0.$}}

Conversely, suppose that $\Omega^T(a)|_S=0.$  Then, by Proposition \ref{Omega-1},
there exists $\{b_n\}\subset {{{\rm Ped}({\rm Her}(a))}}_+^{\bf 1}$ such that
\beq
\lim_{n\to\infty}\|a-b_n\|_{_{2, S}}=0\andeqn \lim_{n\to\infty} \omega(b_n)|_S=0.
\eneq
Let $B={\rm Her}(a)$ 
and  $\Pi_S: l^\infty(B)\to l^\infty(B)/I_{_{S,B}}$
be the quotient map (see Definition \ref{defOs2}  for $I_{_{S,B}}$).
Then 
\beq
\Pi_S(\iota(a))=\Pi_S(\{b_n\}). 
\eneq
For each $k\in \N,$ we have
\beq
\Pi_S({{\iota(a)^{1/k}}})=\Pi_S(\{b_n\}^{1/k}).
\eneq
It follows that, for each $k\in \N,$ there exists $n(k)\in \N$ such that
\beq
\|a^{1/k}-b_{n(k)}^{1/k}\|_{_{2, S}}<1/2k.
\eneq
Choose $e_k=b_{n(k)}^{1/k},$ $k\in \N.$ 
Then, for any $c\in {\rm Her}(a),$ 
\beq
\|c-ce_k\|^{2/3}_{{_{2, S}}}\le \|c-ca^{1/k}\|^{2/3}_{_{2, S}}+\|c(a^{1/k}-b_{n(k)}^{1/k})\|^{2/3}_{_{2, S}}\to 0,\,\,\, {\rm as}\,\,k\to\infty.
\eneq
Since $b_{n(k)}\sim e_k,$
\beq
\lim_{k\to\infty}\omega(e_k)|_S=0.
\eneq
{{The proposition then follows.}}
\end{proof}

\begin{prop}\label{PN=T}
Let $A$ be a $\sigma$-unital \CA, 
$S\subset\wtd{QT}(A)$ be a compact subset, 
and let $a\in {{{\rm Ped}(A\otimes{\cal K})_+.}}$
Then $\Omega^T(a)|_S=0$
if and only if $\Omega^N(a)|_S=0.$
\end{prop}

\begin{proof}

For the ``if'' part, let us assume 
$ \Omega^N(a)|_S=0.$
By {{the}} definition there is  
$\{b_n\}\in l^\infty(\Her(a))_+$
such that $ \omega(b_n)|_S<1/n$ 
and $\|a-b_n\|< 1/n.$ 
Let $b_n'=\frac{\|a\|b_n}{\|b_n\|+1/n}.$
Then $\|b_n'\|\leq \|a\|,$
\beq
&&{{\lim_{n\to\infty}\omega(b_n')|_S=
\lim_{n\to\infty}\omega(b_n)|_S=0,
\lim_{n\to\infty}\|b_n-b_n'\|=0\andeqn}}\\
&&0\le \Omega^T(a)|_S\le 
\|\Pi_S(\iota(a)-\{b_n'\})\|
\le\limsup_{n\to\infty}\|a-b_n'\|
{{\le}}\limsup_{n\to\infty}\|a-b_n\|=0.
\eneq

For the ``only if'' part, let us 
assume that $\Omega^T(a)|_S=0.$ 
Then, by Proposition \ref{PapproxeT}, there are $e_n\in 
{{\Her(a)}}_+^{\bf 1}$ 
such that
\beq
\lim_{n\to\infty} \omega(e_n)|_S=0\andeqn \lim_{n\to \infty}\|a-a^{1/2}e_na^{1/2}\|_{_{2, S}}=0.
\eneq
{{
It follows that $\{b_n\}=\{a-a^{1/2}e_na^{1/2}\}\in {{(I_{_{S}})_+}}$  (see \ref{D2norm} for 
the definition of $I_{_{S}}$).
By (4) of 
{{Proposition \ref{Omega-1}}}, there exists $\{c_n\}\in (I_{_{S}})_{{+}}$ such 
that
\beq
\lim_{n\to\infty}\sup\{d_\tau(c_n):\tau\in S\}=0\andeqn 
\lim_{n\to\infty}\|b_n-c_n\|=0.
\eneq
Put $d_n=a^{1/2}e_na^{1/2}+c_n, $ $n\in \N.$ Then $d_n\ge 0.$ 
Put ${\bar{d_n}}={\|a\| d_n\over{\|d_n\|+1/n}},$ $n\in \N.$ Then $\|{\bar d}_n\|\le \|a\|$ for all $n\in \N.$
Since $\lim_{n\to\infty}\|a-d_n\|=0,$  we have $\lim_{n\to\infty}\|d_n\|=\|a\|.$
It follows that
\beq\label{F111-1}
\lim_{n\to\infty}\|a-{\bar d}_n\|=0.
\eneq }}

{{On the other hand,  by Proposition \ref{Pbfos} (1) {{(see also (2) of Proposition \ref{Pboundedness})}}
for all $\tau\in S$ 
\beq\label{F111-3}
\omega(d_n)(\tau)|_S
\leq
\omega(a^{1/2}e_na^{1/2})(\tau)|_S
+\overline{d_\tau}(c_n)|_S
\le \omega(a^{1/2}e_na^{1/2}
)|_S+\sup\{d_\tau(c_n): \tau\in S\}
\eneq
{{for all $n\in \N.$}}
By Lemma \ref{Leae},
$\lim_{n\to\infty}\omega(a^{1/2}e_na^{1/2})
=\lim_{n\to\infty}\omega(e_n)=0.$
By 
the fact that $\bar d_n\sim d_n$ and 
Proposition \ref{Pomega-1},
we have
\beq
\lim_{n\to\infty}\omega({\bar d}_n)
&=&
\lim_{n\to\infty}\omega(d_n)
\overset{\eqref{F109-2}}{=}
\lim_{n\to\infty}\sup_{\tau\in S}\{
\omega(d_n)(\tau)|_S\}\nonumber
\\&\overset{\eqref{F111-3}}{\leq}& 
\lim_{n\to\infty}\omega(a^{1/2}e_na^{1/2}
)|_S+
\lim_{n\to\infty}\sup\{d_\tau(c_n): \tau\in S\}=0.
\label{F111-2}
\eneq
Then \eqref{F111-1} and \eqref{F111-2}
show that $\Omega^N(a)|_S=0.$}}

\end{proof}

Let us present some  examples of \CA s which have norm {{approximate}} oscillation  zero. 
\begin{prop}\label{Prr0}
Let $A$ be a \CA\, of real rank zero.
Then $A$ has norm approximate oscillation zero.
\end{prop}

\begin{proof}
Let $a\in {\rm Ped}(A)$ with  $0\le a \le 1.$
Put $B={\rm Her}(a).$ Then $B$ has an approximate identity $\{e_n\}$ consisting of projections.
Note  that, since $e_n$ is a projection, $\widehat{[ e_n]}=\widehat{e_n}$ is continuous on $S.$
The proposition follows from 
{{Proposition}} {{\ref{PapproxeT}}}
\end{proof}

Let  $T_b=\{s\in {\wtd{QT}}(A): s(b)=1\}$ for some  nonzero 
$b\in ({\rm Ped}(A\otimes {\cal K}))_+.$ It is a compact convex subset and $T_b$ is a basis 
for the cone ${{\wtd {QT}}}(A)$ if $A$ is simple.  The proof of the following 
is taken from  Lemma 4.8 of \cite{Lin96pacific}  (see also Remark 4.7 of \cite{Lin96pacific}).

\begin{thm}\label{Tcounableos}
Let $A$ be a  \CA\, with  countable $\partial_e(T_b)$ (for some $b\in {\rm Ped}(A)_+\setminus \{0\}$),
where $\partial_e(T_b)$ is the  set of extremal points  of $T_b.$ Then 
\beq
\Omega^{N}(a)=0\tforal a\in {\rm Ped}(A\otimes {\cal K})_+.
\eneq
In particular, $A$ has norm approximate oscillation zero.
\end{thm}

\begin{proof}
We may assume that $a\in {\rm Ped}(A\otimes {\cal K})_+\setminus \{0\}$ and 
$0\le a\le 1.$ 
If $0\in \overline{{{\R_+}}\setminus {\rm sp}(a)},$  then $\Omega^N(a)=0.$ 
To see this, let $s_n\in \R_+\setminus {\rm sp}(a)$ such that 
$s_n\searrow 0.$ 
Then the characteristic  functions $\chi_{[s_n, 1]}$ is continuous on ${\rm sp}(a).$
Therefore $p_n=\chi_{[s_n, 1]}(a)\in C^*(a)\subset {\rm Her}(a)$ is a projection.   Note 
that $p_n\le p_{n+1}$ for all $n\in \N.$
Then 
\beq
\|a-p_na\|{{\le}}s_n\to 0.
\eneq
In other words  $\{p_n\}$  is an approximate identity for $\Her(a).$ 
Moreover  $\omega(p_n)=0.$  So this case follows from 
{{Proposition \ref{PapproxeT}.}}

Therefore, \wilog, we may assume that  $[0,r]\subset {\rm sp}(a)$ for some 
$0<r<1.$ 
Let $r/2>\eta>0.$ 

Note that, since $a\in {\rm Ped}(A\otimes {\cal K})_+,$ 
$\sup\{d_\tau(a): \tau\in T_b\}=M<\infty$ (see \ref{Pboundedness}).
For each $\tau\in \partial_e(T_b),$ 
$\tau$ induces a Borel measure $\mu_\tau$ on ${\rm sp}(a)$ which is  bounded by $M.$ 

We claim that there is $s\in ( r-\eta, r]$ such that
\beq
\sup\{\mu_\tau(\{s\}):\tau\in T_b\}=0.
\eneq
To see this,  write $\partial_e(T_b)=\{\tau_n\}_{n\in \N'},$ where 
$\N'$ is a subset of $\N.$  For each $k\in \N',$ and $n\in \N,$ define 
\beq
S_{k,n}=\{ s\in (r-\eta, r]:  \mu_{\tau_k}(\{s\})\ge 1/n\}.
\eneq
Since $\mu_{\tau_k}({\rm sp}(a))\le M,$ $S_{k,n}$ must be finite.
It follows that 
$S_k=\cup_{n=1}^\infty S_{k,n}$ is countable. 
Thus $S=\cup_{k\in \N'}S_k$ is countable. 
Therefore there must be $s\in (r-\eta, r]$ such that  $s\not\in S.$ In other words, there must be 
$s\in (r-\eta,r]$ such that 
\beq
\mu_\tau(\{s\})=0\rforal \tau\in \partial_e(T_b).
\eneq
Since $T_b$ is compact, by  the Krein-Milman Theorem, this implies that
\beq
\mu_\tau(\{s\})=0\rforal \tau\in T_b.
\eneq
This proves the claim.

By the claim, for each integer $n\in \N,$ 
there is $\dt>0$ such that $1/2^{n+1}<\dt<1/2^n$ 
and 
\beq\label{Tcountable-2}
\mu_\tau(\{\dt\})=0\rforal \tau\in T_b.
\eneq

Note that 
\beq\label{Tcountable-2+}
d_\tau((a-\dt)_+)=\mu_\tau((\dt, 1]\cap {\rm sp}(a))\rforal \tau\in T_b
\eneq
is a lower semi-continuous function. 
By Portmanteau Theorem, the function $h: T_b\to \R_+$  given by
\beq
h(\tau):=\mu_\tau([\dt, 1]\cap {\rm sp}(a))
\eneq
is an upper semi-continuous {{function.}}\footnote{
This can be directly obtained as follows:
Let $f^{(n)}\in C_0((0,\|a\|{{]}})_+^{\bf 1}$ be such that
$f^{(n)}(t)=1$ for $t\in [{{\dt}} 
,\|a\|],$ $f^{(n)}=0$
if $t\in [0, {{\dt}} 
-\dt/2^n]$ and linear in $[{{\dt}} 
-\dt/2^n, {{\dt}}
].$
Then $\tau(f^{(n)}(a))\searrow h(\tau).$}

By \eqref{Tcountable-2},
\beq
h(\tau)=d_\tau({{(a-\dt)_+}})\rforal \tau\in T_b.
\eneq
It follows that $d_\tau({{(a-\dt)_+}})$ is continuous.

To prove the lemma, let $\ep>0.$ Choose $n\in \N$ such that $1/2^n<\ep.$ 
Choose $1/2^{n+1}<\dt_n<1/2^n$  {{such that $\mu_\tau(\{\dt_n\})=0$
for all $\tau\in S$}} as above.
Put $d_n={{(a-\dt_n)_+.}}$
Then 
\beq
\|a-d_n\|<\ep.
\eneq
Since $d_\tau(d_n)=d_\tau({{(a-\dt_n)_+}}),$  we have that
$\omega({{d_n}})(\tau)|_{T_b}=0$ for all $\tau\in T_b.$
The lemma follows.
\end{proof}

The following is a restatement of Lemma 7.2 of \cite{eglnp} 
with the same proof (and with some necessary modification and correcting a typo). 

\begin{thm}\label{Talsr1}
Let $A$ be a $\sigma$-unital simple \CA\, 
which has strict comparison and  almost stable rank one.
Suppose that the canonical map $\Gamma: {\rm Cu}(A)\to {\rm LAff}_+({\wtd{QT}}(A))$ is surjective
(see \ref{DGamma}). 
Then $A$ has norm approximate oscillation zero. 
\end{thm}

\begin{proof}
Let  $e_A\in {\rm Ped}(A)_+\setminus \{0\}$ and $A_1=\Her(e_A).$
Then ${\rm Ped}(A_1)=A_1.$ Since $A$ is $\sigma$-unital, $A\otimes {\cal K}\cong A_1\otimes 
{\cal K}$ by Brown's stable isomorphism theorem (\cite{Br}).
Therefore it suffices to show that $A_1$ has norm approximate oscillation zero.
To simplify notation, we may assume that $A=A_1$ (and $A={\rm Ped}(A)$).

Let $a\in {\rm Ped}(A\otimes {\cal K})_+^{\bf 1}.$
It follows from the proof of Lemma 7.2 of \cite{eglnp} that, since $\Her(a)$ also has strict comparison,
almost stable rank one and $\Gamma$ is surjective,  
$\Her(a)$ has an approximate identity  consisting of elements  $\{e_n\}$ such 
that $\widehat{e_n}$ is continuous on $\Qw.$ Then the theorem  follows from   Proposition 
{{\ref{PapproxeT}.}}
Since we do not assume that $QT(A)=T(A),$ to be complete, 
let us repeat some of the argument in the proof of Lemma 7.2 of \cite{eglnp}
which will be, again, used in the proof of \ref{LL-N1}. 

By Proposition {{\ref{PapproxeT},}}
it  suffices to show that, for any $\ep\in (0, 1/2),$ 
there is $e\in \Her(a)_+^{\bf 1}$ 
such that $f_\ep(a)\le e$ and $\widehat{[e]}$ is continuous. 
By the first part of the proof of Theorem \ref{Tcounableos}, we may assume that 
 $[0,\ep_0)\subset {\rm sp}(a)$ for some $\ep_0\in (0, \ep).$ 

 Note that, \wilog, we may assume that 
\beq\label{Talsr1-1}
d_\tau(f_{{{\ep/2}} 
}(a))<\tau(f_{\dt_1}(a))<
d_\tau(f_{\eta_1}(a))<\tau(f_{\dt_2}(a))<d_\tau(f_{\eta}(a))\rforal \tau\in \Qw,
\eneq
where $\ep/4>\dt_1, \dt_1/2>\eta_1, \eta_1/2>\dt_2,$ and $\dt_2/2>\eta>0.$
Put $h_i(\tau)=\tau(f_{\dt_i}(a))\rforal \tau\in 
{{\Qw}}
,$ $i=1,2.$
Then $h_i\in \Aff_+(\Qw).$
Since $\Gamma$ is surjective, there is $c\in  (A\otimes  {\cal K})_+^{{{\bf 1}}}$ such 
that $d_\tau(c)=h_2(\tau)$ for all $\tau\in \Qw.$ 
Since $A$ has strict comparison,
{{\eqref{Talsr1-1} and the choice of $c$ 
implies $c\lesssim f_{\eta}(a).$}} 
{{Since $A$ has}} almost stable rank one, by Lemma 3.2 of \cite{eglnp}
and 
\eqref{Talsr1-1}, 
there is $x\in A\otimes {\cal K}$ such 
that
\beq
x^*x=c\andeqn xx^*\in {\rm Her}(f_{\eta}(a)).
\eneq
Put $c_0=xx^*.$ Then $0\le c_0\le 1.$ Note that 
$d_\tau(c_0)=d_\tau(c)$ for all $\tau\in \Qw.$
Since $h_1$ is 
continuous, 
{{$h_1(\tau)<h_2(\tau)=
d_\tau(c)=d_\tau(c_0)=\lim_{n\to\infty}\tau(f_{1/n}(c_0))$
for all $\tau\in\Qw,$
and $\Qw$ is compact,}}
there is an integer $m>2$ such that
\beq
h_1(\tau)<\tau(f_{1/m}(c_0))\rforal \tau\in \Qw.
\eneq
By \eqref{Talsr1-1} and Lemma 3.2 of \cite{eglnp}, there is a unitary 
$u$ in the unitization of $\Her(f_{\eta}(a))$ such that
\beq
u^*f_{\ep/8}(f_{{{\ep/2}
}}(a))u\in \Her(f_{1/m}(c_0)).
\eneq
Set $c_1=uc_0u^*.$ Then 
\beq\label{Talsr1-2}
f_{\ep/8}(f_{{\ep/2}}(a))\in \Her(f_{1/m}(c_1))\subset \Her(c_1).
\eneq
There is $g\in  C_0((0,1])$ such that $0\le g\le 1,$ $g(t)>0$ for all $t\in(0,1],$ and 
\beq
gf_{1/2m}=f_{1/2m}.
\eneq
Put $e=g(c_1).$ Then $[e]=[ c_1]=[ c_0] = [ c].$
In particular, $d_\tau(e)$ is continuous on $\Qw.$ 
But we also have, by \eqref{Talsr1-2},
\beq
f_\ep(a)\le f_{\ep/8}(f_{{\ep/2}}(a))\le e. 
\eneq
This completes the proof.

\end{proof}

\section{\CA\, $l^\infty(A)/I_{_{\Qw}}$}

\begin{df}\label{Ddppro}
{{Let $A$ be a compact \CA\, and 
let}} $p\in  l^\infty(A)/I_{_{\Qw}}$ (or in $l^\infty(A)/N_{cu}(A)$)  be a projection  and 
$\{e_n\}\in l^\infty(A)_+^{\bf 1}$ such that $\Pi(\{e_n\})=p$
(recall that $\Pi: l^\infty(A)\to l^\infty(A)/I_{_{\Qw}}$ is  the quotient map). 
The sequence $\{e_n\}$ is said to be a permanent projection lifting, 
if for any sequence of  positive integers $\{m(n)\},$ 
\beq
\Pi(\{e_n^{1/m(n)}\})=p\,\,\, (or \,\, \Pi_{cu}(\{e_n^{1/m(n)}\})=p).
\eneq
\end{df}

\begin{prop}\label{Dpproj}
Let $A$ be a compact  \CA\, 
with ${\wtd{QT}}(A)\setminus \{0\}\not=\emptyset$ {{and $\{e_n\}\in l^\infty(A)_+^{\bf 1}.$}}

(1) Let $p=\Pi(\{e_n\})$ (or $p=\Pi_{cu}(\{e_n\})$)  be a projection.
Then $\{f_{\dt}(e_n)\}$ is a permanent projection lifting of $p$ for any $0<\dt<1/2$ (for both cases)
and  
\beq
\lim_{n\to\infty}\sup\{\tau(e_n-f_\dt(e_n)e_n): \tau\in \Qw\}=0\,\,\,
({\rm or}\,\, {{\{}}e_n- e_n^{1/2}f_\dt(e_n)e_n^{1/2}{{\}}}\in N_{cu}).
\eneq

(2)  If $\{e_n\}$ is a permanent projection lifting, then $\lim_{n\to\infty}\omega(e_n)=0.$ Moreover, 
an element $\{e_n\}$ is a permanent projection lifting (from $\l^\infty(A)/I_{_{\Qw}}$) if and only if 
$$\lim_{n\to\infty} \sup\{d_\tau(e_n)-\tau(e_n^2):\tau\in \Qw\}=0.$$

(3) If $\{e_n\}\in l^\infty(A)_+^{\bf 1}$ and $\lim_{n\to\infty}\omega(e_n)=0$ 
{{then for some $l(k)\in \N,$ 
$p=\Pi(\{e_k^{{{1/l(k)}}}\})$ 
is a projection,
and  $\{e_k^{{{1/{l(k)}}}}\}$ is a permanent projection lifting of $p.$}}

(4)  Suppose that ${{p=}}\Pi_{{{cu}}}(\{e_n\})$ is a projection 
for some $\{e_n\}\in l^\infty(A)_+^{\bf 1}.$ Then $\{e_n\}$ is a permanent projection lifting (from $\l^\infty(A)/N_{cu}$) if 
$g_\dt(e_n)\cto 0$ for some $\dt\in (0,1/4).$ 

(5)  If $\{e_n\}$ is a permanent projection lifting of  $p \in l^\infty(A)/I_{_{\Qw}},$ 
then $$l^\infty (\{\Her(e_n)\})/I_{_{\Qw}}=p(l^\infty(A)/I_{_{\Qw}})p$$ (see \ref{DNcu} for
$l^\infty(\{\Her(e_n)\})$).

(6) If $A$ is algebraically simple and $QT(A)\not=\emptyset$ and $e\in A_+^{\bf 1}$ is a strictly positive element such that $\widehat{{{[e]}}}$ is continuous 
on $\Qw,$   then $l^\infty(A)/I_{_{\Qw}}$ is unital.

(7)  A $\sigma$-unital simple {{\CA\,}} $A$ 
has  continuous scale if and only if 
$l^\infty(A)/N_{cu}$ is unital.

\end{prop}

\begin{proof}

(1) Note that $\Pi(f_\dt(\{e_n\}))=f_\dt(\Pi(\{e_n\}))=p$ for any $0<\dt<1/2.$
Therefore $\Pi(f_{\dt/2}(\{e_n\}))=p.$ Put $b_n=f_\dt(e_n),$ $n\in \N.$ 
For any integers $\{m(n)\},$   we have 
\beq
b_n^{1/m(n)}\le f_{\dt/2}(e_n),\,\,n\in \N.
\eneq
It follows that
\beq
p=\Pi(\{f_\dt(e_n)\})\le \Pi(\{b_n^{1/m(n)}\})\le \Pi(\{f_{\dt/2}(e_n)\})=p.
\eneq
This proves the first part of (1) 
{{(the proof for $p=\Pi_{cu}(\{f_\dt(e_n)\})$ is similar).}}

Since   
$\Pi(\{f_{\dt}(e_n)\})=p$ (or $\Pi_{cu}(\{f_\dt(e_n)\})=p$),  we have
 \beq
 e_n-f_\dt(e_n)e_n\in I_{_{\Qw}}, \,\,\,  {\rm hence}\,\,(e_n-f_\dt(e_n)e_n)^{1/2}  \in I_{_{\Qw}}\\
 {\rm (or }\,\,
e_n-f_\dt(e_n)e_n,\,\,(e_n-f_\dt(e_n)e_n)^{1/2}\in N_{cu} 
{\rm )}.
\eneq
 {{Part}} (1) follows.

(2)   If 
$$\lim_{n\to\infty} \sup\{d_\tau(e_n)-\tau(e_n^2):\tau\in \Qw\}=0,$$
then, for any
$\{m(n)\}{{\ \subset \N}},$ 
\beq
\sup\{\tau(e_n^{1/m(n)})-\tau(e_n^2): \tau\in \Qw\}\le \sup\{d_\tau(e_n)-\tau(e_n^2):\tau\in \Qw\}\to 0.
\eneq
It follows that 
$\{e_n^{1/m(n)}-e_n^2\}\in I_{_{\Qw}}.$
Since $e_n^2\le e_n$ for all $n\in \N,$ this also implies that $\{e_n-e_n^2\}\in I_{_{\Qw}}.$ 
Hence $\Pi(\{e_n\})$ is a projection and $\{e_n\}$ is a permanent projection lifting.

Now suppose that $\{e_n\}$ is a  permanent projection lifting of $p=\Pi(\{e_n\}).$ 
Let us show first 
that $\lim_{n\to\infty}\omega(e_n)=0.$
Otherwise, there exists a subsequence $\{l(k)\}$ such 
that $\omega(e_{l(k)})>\sigma$ for some $\sigma>0.$
Fix any $\dt\in (0,1/4).$ By Proposition \ref{Ppartomega},  for each of these $l(k),$ there are $m(l(k))$ such 
that
\beq
\sup\{\tau(e_{l(k)}^{1/m(l(k))})-\tau(f_{\dt}(e_{l(k)})): \tau\in \Qw\}>\omega(e_{l(k)})-\sigma/4>\sigma/2.
\eneq
Choose a sequence $m(n)$ of integers which extends $m(l(k)).$ 
Then
\beq
\limsup_n\|(e_n^{1/m(n)}-f_{\dt}(e_n))^{1/2}\|_{_{2, \Qw}}\ge \sigma/2.
\eneq
Therefore $\Pi(\{e_n^{1/m(n)}\})\not=\Pi(f_\dt(\{e_n\}))=p.$ A contradiction.
Hence
$\lim_{n\to\infty}\omega(e_n)=0.$

Therefore there exists a sequence $\{m(n)\}$ such that
\beq
\sup\{d_\tau(e_n)-\tau(e_n^{1/m(n)}): \tau\in \Qw\}<1/n,\,\,n\in \N.
\eneq
Then, since $\{(e_n^{1/m(n)}-e_n^2)^{1/2}\}\in I_{_{\Qw}}$
 (for any $\{m(n)\}$), we also have that
\beq\nonumber
&&\hspace{-0.5in}\sup\{d_\tau(e_n)-\tau(e_n^2): \tau\in \Qw\}\\\nonumber
&&\le 
\sup\{d_\tau(e_n)-\tau(e_n^{1/m(n)}): \tau\in \Qw\}+\sup\{\tau(e_n^{1/m(n)}-e_n^2):\tau\in \Qw\}\\
&&<1/n+\|(e_n^{1/m(n)}-e_n^2)^{1/2}\|_{_{2, \Qw}}\,\,\to 0 \,\,\,\, {\rm as\,\,{{n}}\to\infty}.
\eneq

(3) In this case, {{since $\lim_{n\to\infty}\omega(e_n)=0,$}} there are $l({{n}})\in \N$ such that
\beq
\lim_{n\to\infty}\sup\{d_\tau(e_n)-\tau(e_n^{1/l({{n}})}): \tau\in \Qw\}=0.
\eneq
It follows that  
\beq\label{Ppproj-n1}
\{e_{{n}}^{1/m({{n}})}-e_{{n}}^{1/l({{n}})}\}\in I_{_{\Qw}}
\eneq for any integers $m(n)\ge l({{n}}).$ 
Since 
\beq
\|e_{{n}}^{1/2{{l(n)}}}
-(e_{{n}}^{1/2{{l(n)}}})^2\|^2{{_{_{2, \Qw}}}}\le \sup\{d_\tau(e_{{n}})-\tau(e_{{n}}^{l({{n}})})
: \tau\in \Qw\}\to {{0,}}
\eneq
as ${{n}}\to\infty,$ 
$\Pi(\{e_n^{1/2{{l}}(n)}\})=p$ is a projection.
By \eqref{Ppproj-n1},
 {{for}} any  integers ${{m(n)}}\ge l({{n}}),$
\beq
\Pi(\{e_{{n}}^{1/{{m(n)}}}\})
=\Pi(\{ e_{{n}}^{1/{{l(n)}}}\})
{{=(\Pi(\{ e_n^{1/2l(n)}\}))^2}}
=p.
\eneq
It follows that $\{e_n^{{{1/l(n)}}}\}$ is a permanent projection lifting of $p.$

(4)
Suppose that $g_\dt(e_n)\cto 0$ for some $0<\dt<1/4.$ 
We have, {{for any $m(n)\in \N,$}}
\beq
e_n^{1/m(n)}-f_{\dt/2}(e_n)e_n^{1/m(n)}\lesssim g_\dt(e_n)\rforal n\in \N.
\eneq
It follows that $\{e_n^{1/m(n)}-f_{\dt/2}(e_n)e{{_n^{1/m(n)}}}\}\in N_{cu}.$ One then 
checks that 
\beq
{{\hspace{-0.1in}}}
p\le \Pi_{cu}(\{e_n^{1/{{m(n)}}}\})=\Pi_{cu}(\{f_{\dt/2}(e_n)e{{_n^{1/m(n)}}}\})\le \Pi_{cu}(
\{f_{\dt/2}(e_n)\})
{{=f_{\dt/2}(\Pi_{cu}(
\{e_n \}))}}=p.
\eneq
{{Thus (4) follows.}}

For (5),  {{let $B=l^\infty(A)/I_{_{\Qw}}.$}}
It is clear that $pBp\subset \Pi(l^\infty(
{{\{}}{\rm Her}(e_n){{\}}})).$
Suppose that 
${{g\in}}\Pi(l^\infty(
{{\{}}{\rm Her}(e_n){{\}}}){{)_+^{\bf 1}}}.$ 
Then we may write $g=\Pi(\{g_n\})$ such that $g_n\in ({\rm Her}(e_n))^{\bf 1}_+,$ {{$n\in \N.$}}
For any $\ep>0,$ there {{exists}} $m(n)\in \N$ such that
\beq
\|{{e_n^{1/m(n)}}}g_ne_n^{1/m(n)}-g_n\|<\ep\rforal n\in \N.
\eneq
Thus
\beq
\|\Pi(\{e_n^{1/m(n)}\})g\Pi(\{e_n^{1/m(n)}\})-g\|<\ep.
\eneq
However,  since $\{e_n\}$ is a permanent projection lifting of $p,$
$\Pi(\{e_n^{1/m(n)}\})=p.$ 
Thus
\beq
\|{{p}}gp-g\|<\ep.
\eneq
It follows  $g\in pBp.$ This shows  that $C=pBp=\Pi(l^\infty({{\{}}{\rm Her}(e_n){{\}}})).$ 

(6) In this case, $\omega(e)=0.$ 
Therefore, by (3), $\{e^{{{1/l(n)}}}\}$ is a permanent 
projection lifting for $p=\Pi(\{e^{{{1/l(n)}}}\})$  (for some $l({{n}})\in \N$).

For any $\{x_n\}\in l^\infty(A),$ 
there is a sequence $\{m(n)\}$ of integers such that
\beq
\|x_ne^{1/m(n)}-x_n\|<1/n\andeqn \|e^{1/m(n)}x_n-x_n\|<1/n,\,\,n\in \N.
\eneq
Hence $p\{x_n\}=\{x_n\}p=\{x_n\}.$ So $p$ is the unit of $l^\infty(A)/I_{_{\Qw}}.$

For (7), 
suppose that $A$ has continuous scale.  
{{Let $e\in A_+^{\bf 1}$ be a strictly positive element.}}
By (3) of Proposition \ref{Pbfos}, $\omega^c(e)=0.$
Then, for any $\ep>0$ and $n\in \N,$  there exists an integer $l(n)\in \N$
such that 
\beq
{{f_\ep(e^{1/(m(n)}-e^{1/l(n)})\lesssim g_{1/n}(e)\,\,\,{\rm   for\,\, any}\,\,\, m(n)>l(n).}}
\eneq
Since $g_{1/n}(e)\cto 0$ (see \ref{DefOs2}),  we conclude that 
$\{e^{1/(m(n)}-e^{1/l(n)}\},\, \{ e^{1/2l(n)}-e^{1/l(n)}\}\in N_{cu}.$
It follows that $p=\Pi(\{e^{1/l(n)}\})$ is a projection and $\{e^{1/l(n)}\}$ is a permanent projection lifting.

Let $\{b_n\}\in l^\infty(A).$ 
Then, for each $n\in \N,$  there is $m(n)\in \N$  with $m(n)\ge l(n)$ such that
$\|e^{1/m(n)}b_n-b_n\|<1/n.$ Recall that $p=\Pi(\{e^{1/m(n)}\}).$
Thus $p\Pi(\{b_n\})=\Pi(\{b_n\}).$ It follows  $l^\infty(A)/{{N_{cu}}}$
is unital.

Conversely,  let $p\in l^\infty(A)/{{N_{cu}}}$
be the unit. Let $\{e_n\}\in l^\infty(A)_+^{\bf 1}$
such that $\Pi(\{e_n\})=p.$  We claim that $\Pi(A)^\perp=\{0\}.$ 
Otherwise, for each $n,$ there exists $b_n\in A_+$ with $\|b_n\|=1$
such that $\|e_nb_n\|<1/n.$ Then 
$p\Pi(\{b_n\})=0.$ Impossible. Thus $\Pi(A)^\perp=\{0\}.$ By  Proposition \ref{Porthogoinal},
$A$ has continuous scale.
\end{proof}

\begin{lem}\label{LPPP}
Let $A$ be a    compact 
\CA\ with 
$QT(A)\not=\emptyset.$ 
If $A$ has  $T$-tracial approximate oscillation zero, then, for 
any $x\in (l^\infty(A)/I_{_{\Qw}})_+,$ there is a projection 
$p\in {{l^\infty(A)/I_{_{\Qw}}}}$ such that $px=x=xp.$
\end{lem}

\begin{proof}
Let  
$B=l^\infty(A)/I_{_{\Qw}}$ and $\Pi: l^\infty(A)\to B$ be the quotient map.
Let $x\in B_+.$  \Wlog, we may assume that $0\le x\le 1.$ 

Let $y=\{y_n\}\in l^\infty(A)$ with $0\le y\le 1$ such that $\Pi(y)=x.$ 
Since $A$ has T-tracial approximate oscillation zero,
there are $d_n\in  \Her(y_n)_+^{\bf 1}$ and $\dt_n
{{\in(0,1/4n)}}
$  such that 
\beq
&& \|y_n-d_n\|_{_{2, \Qw}}<1/4n,
\andeqn\\ 
&&d_\tau(d_n)-\tau(f_{\dt_n}(d_n))<1/4n\rforal \tau\in \Qw \andeqn n\in \N.
\eneq
Define $d=\{d_n\}\in l^\infty(A).$ Then $\Pi(d)
=\Pi({{y}}
)=x.$
{{Put}} $e_n=f_{\dt_n/2}(d_n),\,\, n\in \N$ and $e=\{e_n\}.$   Then $e\in {{l^\infty(A{)_+^{\bf 1}}.}}$
Moreover
\beq\label{LPPP-1}
\lim_{n\to\infty}\|e_nd_n-d_n\|=0.
\eneq
It follows that
\beq
{{\Pi(e)x=}}
\Pi(e)\Pi(d)=\Pi(d)=x{{.}}
\eneq
It remains to show that $p:=\Pi(e)$ is a  projection. 
To do this, we compute that
\beq
\tau(e_n-e_n^2)\le \tau(e_n-f_{\dt_n}(d_n))<d_\tau(d_n)-\tau(f_{\dt_n}(d_n))<1/n\rforal n\in \N.
\eneq
It follows that
\beq
\|e_n-e_n^2\|_{{_{2, \Qw}}}<1/\sqrt{n}\to 0.
\eneq
Thus $p=\Pi(e)=\Pi(e)^2,$ or $p\in B$ is a projection.  
\end{proof}

\begin{thm}\label{PLrr0}
Let $A$ be a    compact 
\CA\ with  
non-empty $QT(A).$  
If $A$ has T-tracial approximate oscillation zero, then
 $l^\infty(A)/I_{\Qw}(A)$   has real rank zero.
\end{thm}

\begin{proof}
Let 
$B=l^\infty(A)/I_{_{\Qw}}(A)$ and $\Pi: l^\infty(A)\to B$ be the quotient map.

We claim that, if $p\in B$ is a non-zero projection, then 
$C:=p{{\wtd B}}p=pBp$ has real rank zero.

Let $\{e_n\}\in  l^\infty(A)_+^{\bf 1}$ such that $\Pi(\{e_n\})=p.$ 
Upon replacing $e_n$ by 
$f_{1/4}(e_n),$   by Lemma \ref{Dpproj}, we may assume 
that $\{e_n\}$ is a permanent projection lifting of $p.$
By (5) of 
{{Proposition}}
\ref{Dpproj}, $C=pBp=\Pi(l^\infty(
{{\{}}{\rm Her}(e_n)
{{\}}})).$

Let $a, b\in C_+$ be such that $ab=0.$ 
We may assume that $\|a\|, \|b\|\le 1.$ 

 Then,  by  Proposition 10.1.10 of \cite{Lor97}, for example,  we may assume that 
$a=\Pi(\{a_n\})$ and $b=\Pi(\{b_n\}),$ where $a_n, b_n\in {\rm Her}(e_n)_+$  and $\{a_n\},\,\{b_n\}\in l^\infty({{\{}}{\rm Her}(e_n){{\}}})$ such that 
$a_nb_n=b_na_n=0$ for all $n\in \N.$
Since $A$ has T-tracial approximate oscillation zero, there are $d_n\in {\rm Her}(a_n)_+^{\bf 1}$  and $\dt_n
{{\in(0,1/4n)}}
$  such that 
\beq
&&\|a_n-d_n\|_{{{_{2, \Qw}}}}<1/4n
\andeqn\\ 
&&|d_\tau(d_n)-\tau(f_{\dt_n}(d_n))|<1/4n\rforal \tau\in \Qw\andeqn n\in \N.
\eneq
Define $d=\{d_n\}\in l^\infty({{\{}}{\rm Her}(a_n){{\}}}).$ Then $a=\Pi(\{a_n\})=\Pi(d).$
{{Put}} 
\beq
{{g_n=f_{\dt_n/2}(d_n),\,\, n\in \N\andeqn g':=\{g_n\}\in l^\infty({{\{}}{\rm Her}(a_n){{\}}})\subset l^\infty({{\{}}{\rm Her}(e_n){{\}}})\subset l^\infty(A).}}
\eneq
Put ${{g=\Pi(g')}}.$ 
Recall that $l^\infty({{\{}}{\rm Her}(e_n){{\}}})=pBp.$ Thus $g\in pBp_+.$   Since $d_n\in  {\rm Her}(a_n),$  we have $g_nb_n=b_ng_n=0.$
In other words,
\beq\label{RR0-1}
gb=0.
\eneq
Note that $f_{\dt_n}(d_n)g_n=f_{\dt_n}(d_n)$ for all $n\in \N.$ 
It follows that
\beq\label{PLrr0-10}
g_n^2\ge f_{\dt_n}(d_n)\rforal n\in \N.
\eneq
We compute that
\beq
\tau(g_n-g_n^2)\le \tau(g_n-f_{\dt_n}(d_n))<d_\tau(d_n)-\tau(f_{\dt_n}(d_n))<1/n\rforal n\in \N.
\eneq
It follows that
\beq
\|g_n-g_n^2\|_{{{_{2, \Qw}}}}<1/\sqrt{n}\to 0.
\eneq
Thus ${{g=g^2}},$ or $g\in pBp$ is a projection.
Recall that 
\beq
\lim_{n\to\infty}\|{{g_n}}d_n-d_n\|=\lim_{n\to\infty}\|f_{\dt_n/2}(d_n)d_n-d_n\|=0.
\eneq
In  other words,
\beq
{{g}}a=g\Pi(d)=\Pi(d)=a.
\eneq
This and  \eqref{RR0-1}
imply  that $pBp$ has real rank zero and the claim is proved.

To show that $B$ has real rank zero, 
let $x\in B_{s.a.}$ and let $\ep>0.$ 
Put $z=x^2.$
Then, by Lemma \ref{LPPP}, there is a  projection 
$p\in B$ such that $pz=z=zp.$ 
Hence $x\in pBp.$
By the claim that we  have just shown, $pBp$ has real rank zero.
Then there is $y\in (pBp)_{s.a.}$  with {{finite spectrum}} 
such that $\|x-y\|<\ep.$   By Theorem 2.9 of \cite{BP},
$B$ has real rank zero.
\end{proof}

If $A$ is unital, then $l^\infty(A)$ is the multiplier algebra of $c_0(A).$
Thus $l^\infty(A)/c_0(A)$ is a corona algebra. Therefore $l^\infty(A)/I_{_{\Qw}}$ is 
a SSAW*-algebra (see Proposition 3 of \cite{Psw} and  section 3 of \cite{Pthree}).  The above theorem also implies 
that, for non-unital case, we also have the next corollary.  (One may also compare the  \ref{Csaw}
and \ref{Cpolar}
(at least in untal case) with those of Lemma 1.8,  Theorem 1.9 and Corollary 1.10
of \cite{LinPlondonIII}.)

\begin{cor}\label{Csaw}
Let $A$ be a $\sigma$-unital compact \CA\, with $QT(A)\not=\emptyset.$
Suppose that $A$ {{has}} T-tracial approximate oscillation zero. 
Then $l^\infty(A)/I_{_{\Qw}}$ is a SSAW*-algebra with real rank zero. 
\end{cor}

\begin{proof}
This is contained in the proof of Lemma \ref{PLrr0}.
Since $M_n(A)$ also has T-tracial approximate oscillation zero,
it suffices to show that 
$B:=l^\infty(A)/I_{\Qw}$ is a SAW*-algebra.
Consider elements $a, b\in B_+$ such that $ab=ba=0.$
By  Lemma \ref{PLrr0},  there is a projection $p\in B$ such
that 
$$
{{p(a+b)=(a+b)p=a+b.}}
$$  
Then the first part of the proof of 
Theorem \ref{PLrr0}
provides a projection ${{g}}\in B$ such that ${{g}}a=a$ and ${{g}}b=0.$ Consequently 
$B$ is a SAW*-algebra of real rank zero.  

\end{proof}

\begin{thm}\label{Tasr1}
Let $A$ be an algebraically simple \CA\, 
with $QT(A)\not=\emptyset.$ 
Suppose that $A$ has strict comparison and  T-tracial approximate oscillation zero. Then $l^\infty(A)/I_{\Qw}$ has stable rank one.
\end{thm}

\begin{proof}
Put $B=l^\infty(A)/I_{{{_{\Qw}}}}.$ 
We first show that $pBp$ has stable rank one if $p$ is a non-zero projection. 
By Theorem  \ref{PLrr0}, $pBp$ has real rank zero. 
Let $q, f\in pBp$ be projections such that $q\sim f.$ Then, by Theorem \ref{P1025-0}, 
$p-q\sim p-f.$  By Proposition III 2.4 of \cite{BH} and  Theorem 2.6 of \cite{BP},  $pBp$ has stable rank one.

To show $B$ has stable rank one, let $x\in {{\wtd B}}$ and $\ep>0.$ 
Write $x=\lambda+y,$ where $\lambda\in \C$ and $y\in  B.$ 
By Lemma \ref{LPPP}, there is a projection $p\in B$ such that
$p(y^*y+yy^*)=(y^*y+yy^*).$ It follows that $y\in pBp.$
From what has been shown,  $pBp$ has stable rank one. Choose $z_1\in GL(pBp)$
such that
\beq
\|\lambda p+y-z_1\|<\ep.
\eneq
Define $z_2=z_1+\lambda(1-p),$ if $\lambda\not=0,$ and $z_2=z_1+\ep(1-p),$  if $\lambda=0.$ 
Then $z_2\in GL({{\wtd B}})$ and 
\beq\nonumber
\|x-z_2\|<\ep.
\eneq
\end{proof}

\begin{cor}\label{Cpolar}
Let $A$ be an algebraically simple \CA\, 
with $QT(A)\not=\emptyset.$ 
Suppose that $A$ has strict comparison and  T-tracial approximate oscillation zero. Then $l^\infty(A)/I_{_{\Qw}}$ has 
unitary polar decomposition.
\end{cor}

\begin{proof}
This follows from Corollary \ref{Csaw}  above  and Theorem  3.5 of \cite{Pthree}.
\end{proof}

\section{Range of dimension functions}

In this section we will show that if $A$ {{has}} T-tracial approximate oscillation zero, 
then the image of $\Gamma$ is ``dense", and if, in addition, $A$ has strict comparison, 
then $\Gamma$ is surjective.

\begin{lem}\label{Ldt}
Let $0\le a\le 1,$ $b, c\in A^{\bf 1}$
be such that
\beq
a\le (b+c)^*(b+c) 
\eneq
(or $a\le b+c,$ $b, c\in A_+^{\bf 1}$).
Then, for any $\dt\in (0, 1/2),$ 
\beq
[f_\dt(a)]\le [f_{\dt/4}(b^*b)]+[f_{\dt/4}(c^*c)]
\eneq
(or $[f_\dt(a)]\le [f_{\dt/4}(b)]+[f_{\dt/4}(c)]$).
\end{lem}

\begin{proof}
Note that
\beq
(b+c)^*(b+c)\le 2(b^*b+c^*c).
\eneq
Let $0<\eta<1/4.$ 
Then, by 
{{Lemma 1.7 of \cite{Phi}}},
\beq\label{Ldt-n1}
((b+c)^*(b+c)-\dt/2)_+
\lesssim
(2(b^*b+c^*c)-\dt/2)_+\lesssim 
f_{\dt/2}(b^*b+c^*c)
\eneq
(recall \ref{Nfg}).
Put 
$
z=\begin{pmatrix} b & 0\\
                               c & 0\end{pmatrix}.$
       Then
      $$ {{z^*z=\diag(b^*b+c^*c, 0)\andeqn zz^*\le 2\diag(bb^*,cc^*).}}$$ 
Hence   {{(see Lemma 1.7 of \cite{Phi}, for example, for the first $``\lesssim"$ sign below),}} 
\beq  
f_{\dt/2}(b^*b+c^*c) &\sim& f_{\dt/2}(zz^*)\sim  (zz^*-\dt/4)_+\\
  & \lesssim &\diag(f_{\dt/4}(bb^*), f_{\dt/4}(cc^*))
  \sim \diag(f_{\dt/4}(b^*b), f_{\dt/4}(c^*c)).
\eneq

We then have (see also \eqref{Ldt-n1})
\beq
f_\dt(a)\sim (a-\dt/2)_+\lesssim ((b+c)^*(b+c)-\dt/2)_+
\lesssim 
\diag(f_{\dt/4}(b^*b), f_{\dt/4}(c^*c)).
\eneq

For the case that $a\le b+c,$    as computed above, 
there is $z\in  M_2(A)$ such that
\beq
z^*z=\diag(b+c,0) \andeqn zz^*
\le 2\diag(b, c).
\eneq
One then sees the proof of the second part is exactly the same as that of the first part.
\end{proof}

\begin{lem}\label{Leqv-2}
Let $A$ be a non-elementary simple \CA\, with  
${\rm Ped}(A)=A$ and 
with  $QT(A)\not=\emptyset.$
Let $\{e_n\}, \{b_n\}\in l^\infty({\rm Ped}(A){{_+^{\bf 1}}}).$ 
{{Recall that $\Pi: l^\infty(A)\to l^\infty(A)/I_{_{\Qw}}$ is the quotient map.}}

(1) Suppose 
that $\Pi(\{e_n\})\le \Pi(\{b_n\})$ (or $\Pi_{cu}(\{e_n\})\le \Pi_{cu}(\{b_n\})$). Then, 
any integer $m\in \N$ and $\ep>0,$ ({{any  $d\in 
{\rm Ped}(A)_+^{\bf 1}\setminus \{0\}$),}}
there exists {{$k_0\in\N$}} such that, for all $k\ge k_0,$
\beq
[f_\ep(e_k)]\le  [b_k]+[d_k],
\eneq
where 
$d_k\in   A_+$ and
$\sup\{d_\tau(d_k):\tau\in \Qw\}<1/m$ (or $d_k\lesssim d$ and $\sup\{d_\tau(d_k):\tau\in \Qw\}<1/m$).

(2) Suppose that  $p=\Pi(\{e_n\})$ is a  projection  (or $p=\Pi_{cu}(\{e_n\})$ is a projection
and $\omega^c(e_n)\to 0$), 
$\{e_n\}$ is a permanent projection lifting  of $p$
and 
\beq
p\le \Pi(\{b_n\})\,\,\,\, {\rm{(or}}\,\, p\le \Pi_{cu}(\{b_n\}){\rm )}.
\eneq
Then,
any integer $m\in \N,$  and $\ep\in (0,1)$ (and any nonzero $d\in {\rm Ped}(A)_+$),
there exists $k_0\ge \N$ such that, for all $k\ge k_0,$
\beq
[e_k]\le  [b_k]+[d_k],
\eneq
where $0\le d_k\le 1,$  $d_k\in (A\otimes {\cal K})_+$ and 
\beq
d_\tau(e_k)<d_\tau(b_k)+d_\tau(d_k)+\ep\tforal \tau\in \Qw,
\eneq
where $\sup\{d_\tau(d_k):\tau\in \Qw\}<1/m,$
(or, 
$d_k\lesssim d$ and $\sup\{d_\tau(d_k):\tau\in \Qw\}<1/m$). 
\end{lem}

\begin{proof}
We will use the following easy claim:  If $B$ is a \CA\, and $I\subset B$ is a (closed two-sided) ideal, 
and, if $x, y\in B_+$ such that $\pi(x)\le \pi(y),$ then there exists
 $j\in I_+$ such that $x\le y+j.$  
 In fact, there is $z\in A_+$  such that $(y-x)-z\in I.$ Put $c=-y+x+z\in I.$ 
  Then $${{x=y-z+c\le y+c\le y+c_+.}}$$ Choose $j=c_+\in I_+.$  This proves the claim.

{{Since $A$ is a non-elementary simple \CA, one may}} choose $d_0\in {\rm Her}(d)_+^{\bf 1}
{{\setminus \{0\}}}$ such that
$2(m+1)[d_0]\le [d].$
By  the easy claim {{above,}} 
there  is
$\{h_n\}\in ( I_{_{\Qw}})_+^{\bf 1}$ (or $\{h_n\}\in N_{cu}(A)_+^{\bf 1}$) such that,
in all {{cases,}}
\beq
b_n+h_n\ge e_n\rforal n\in \N.
\eneq
To show (1), we apply
Lemma \ref{Ldt} to obtain
\beq
f_\ep(e_n)\lesssim \diag(f_{\ep/8}(b_n), f_{\ep/8}(h_n)).
\eneq
Since $h_n$ is in $(I_{_{\Qw}})_+$ 
(or in $(N_{cu})_+$), 
there exists $k_0\in \N$ such 
that, for all $k\ge k_0,$
\beq
d_\tau(f_{\ep/8}(h_k))<1/m\rforal \tau\in \Qw
\eneq
(or $
f_{\ep/8}((h_k)_+)\lesssim d_0
$).
Therefore, with $d_k=f_{\ep/8}(h_k),$  for all $k\ge k_0,$ 
\beq
[f_\ep(e_k)]\le [b_k]+[d_k]\andeqn \sup\{d_\tau(d_k): \tau\in \Qw\}<1/m.
\eneq
Part (1) follows.

For (2), {{we keep the same $d_k$ and $h_k$ as described above.}}
Since now $\{e_n\}$ is a permanent projection lifting, 
we may assume that  $\lim_{n\to\infty}\omega(e_n)=0,$ 
by (2)   of Proposition 
\ref{Dpproj}.  
Thus we  may assume that  there exists $\eta_0>0$ and $n_1\in \N$  such that
\beq\label{17-1}
\sup\{d_\tau(g_\eta(e_n)):\tau\in \Qw
\}\le \sup\{d_\tau(g_{\eta_0}(e_n)): \tau\in  \Qw
\}<\ep/2m
\eneq
{{(see \ref{Nfg})}} for all 
$n\ge n_0\andeqn 0<\eta\le \eta_0,$
(or  we assume that $\omega^c(e_n)\to 0,$ and then
\beq
g_\eta(e_n)\lesssim g_{\eta_0}(e_n)\lesssim d_0\rforal n\ge n_0\andeqn 0<\eta\le \eta_0.
\eneq
 for all 
$n\ge n_0\andeqn 0<\eta\le \eta_0$).

There exists $n_2\in \N$ such 
that {{(recall that $h_n\in (I_{_{\Qw}})^{\bf 1}_+$) (or $h_n\in (N_{cu})^{\bf 1}_+$)}}
\beq\label{17-2}
d_\tau(f_{\eta_0/8}((h_n)) )<\ep/2m\rforal \tau\in \Qw\andeqn n\ge n_2
\eneq
(or 
$
f_{\eta_0/8}((h_n))\lesssim d_0
$
for $n\ge n_2$).
We have, by Lemma \ref{Ldt},
\beq
f_{\eta_0/2}(e_n)\lesssim \diag(f_{\eta_0/8}(b_n), f_{\eta_0/8}((h_n))).
\eneq
Put $k_0=\max\{n_1, n_2\}.$ 
Then, if $n\ge k_0,$ 
we have
\beq\label{Leqv-2-e-1}
[e_n] &\le& [g_{\eta_0}(e_n)]+[f_{\eta_0/2}(e_n)]\le [g_{\eta_0}(e_n)]+[f_{\eta_0/8}(b_n)]+[f_{\eta_0/8}((h_n))]\\
\label{Leqv-2-e-2}
&\le & [b_n] +[d_n],
\eneq
where $d_n=\diag(g_{\eta_0}(e_n), f_{\eta_0/8}((h_n))),$ $n\in \N.$
Recall (by \eqref{17-1}  and \eqref{17-2}) that $d_\tau(d_n)<1/m$ for all $\tau\in \Qw
$
(or, in the second case,
$[d_k]\le 2[d_0]\le [d]$).
Part (2)  then follows.
\end{proof}

\begin{lem}\label{Losfull}
Let $A$ be a non-elementary   algebraically simple \CA\,
with $QT(A)\not=\emptyset.$

(1) Suppose that $A$ has tracial approximate oscillation zero. 
Then, for any $a\in A_+^{\bf 1}\setminus \{0\},$ 
there exists a sequence $0\le e_{n}\le 1$ in ${\rm Her}(a)$ such 
that  $\Pi_{cu}(\{e_{n}\})$ is full in $l^\infty(A)/N_{cu}(A)$ and
\beq
\lim_{n\to\infty}\omega(e_{n})=0.
\eneq

(2) Suppose that $A$ has $T$-tracial approximate oscillation zero.
Then, for any $a\in A_+^{\bf 1}\setminus \{0\},$ 
there exists a sequence $0\le e_{n}\le 1$ in ${\rm Her}(a)$ such 
that $\Pi_{cu}(\{e_{n}\})$ is full in $l^\infty(A)/N_{cu}(A)$ and
\beq
\lim_{n\to\infty}\omega(e_{n})=0.
\eneq
%
\end{lem}

\begin{proof}
(1) Since $A$ has tracial approximate oscillation zero,  by Proposition \ref{PapproxeT},  there exists a tracial  approximate identity
 $\{e_n\}$ for ${\rm Her}(a)$ (with $\|e_n\|\le 1$) such that
 $\lim_{n\to\infty}\omega(e_n)=0.$
 Note 
 that
 \beq
 \Pi_{cu}(\iota(a))=\Pi_{cu}(\iota(a^{1/2})\{e_n\}\iota(a^{1/2})).
 \eneq
 Since $A={\rm Ped}(A),$ by Proposition 5.6 of \cite{eglnp}, 
 $\iota(a)$ is full in $l^\infty(A).$ Hence  
 $\Pi_{cu}(\iota(a))$ is full in $l^\infty(A)/N_{cu}(A),$ and  so is $\Pi(\{e_n\}).$ 

The proof of (2) is exactly the same. We omit it.
\end{proof}

\begin{lem}\label{LsmallMn}
Let $A$ be a  non-elementary  algebraically simple \CA\, 
with $QT(A)\not=\emptyset.$
Suppose that 
$A$  has  T- tracial  approximate oscillation zero.
Then, for any $n\in \N,$  there is a full projection $p\in l^\infty(A)/I_{_{\Qw}}$ 
with $p=\Pi(\{e_j\})$ for some $e_j\in A_+^{\bf 1}$ {{($j\in \N$)}}
such that 
\beq
\sup\{d_\tau(e_j):\tau\in \Qw\}<1/(n+1)\tforal j\in \N,
\eneq
and there is a sequence of mutually orthogonal full projections   $p_1,p_2,...p_k,...$
in $l^\infty(A)/I_{_{\Qw}}$ such that
$pp_j=0,$ $j\in \N$ and 
\beq
 2^{2k}[p_k]\le [p], \,\,k\in \N.
\eneq
Moreover, for each $k\in \N,$ there are mutually orthogonal and 
mutually equivalent full projections 
$p_{k,1},p_{k,2},...,p_{k,2^{k+1}}$ in $p_k(l^\infty(A)/I_{\Qw})p_k.$

\end{lem}

\begin{proof}
Fix $n\in \N.$ Since $A$ is a non-elementary   simple \CA,  we may choose 
two mutually orthogonal  
elements $a_1, a_2\in {\rm Ped}(A)_+^{\bf 1}\setminus \{0\}$
and $x\in A_+$ 
such that
\beq
x^*x=a_1,\,\, xx^*=a_2\andeqn \sup\{d_\tau(a_1):\tau\in \Qw\}<1/(n+1).
\eneq 
Find four  mutually orthogonal and mutually Cuntz equivalent  elements 
$a_{2,1}, ...,a_{2,4}\in {\rm Her}(a_2)_+^{\bf 1}\setminus \{0\}.$ 
By Lemma \ref{Losfull}, there exists a sequence $\{e_n'\}$ in 
${\rm Her}(a_1)^{\bf 1}$ 
such that $\Pi(\{e_n'\})$ is full and $\lim_{n\to\infty}\omega(e_n')=0.$
There exists, by (3) of Lemma \ref{Dpproj},  a sequence of integers $m(j)$ such 
that $p=\Pi(\{e_j')^{1/m(j)}\})$ is a full projection. 
Note that
\beq
\sup\{\tau(e_j): \tau\in\Qw\}<1/(n+1),
\eneq
where $e_j=(e_j')^{1/m(j)},$ $j\in \N.$

In $B={\rm Her}(a_{2,1}),$ one finds a sequence  {{of}}  mutually orthogonal non-zero positive elements $\{y_n\}$  
such that
\beq
2^{2k}[y_{k+1}]\le[y_k],\,\,\, k\in\N.
\eneq 
To see this, choose mutually orthogonal nonzero elements
$b_1, b_1'\in B_+^{\bf 1}$ and $c_1\in B$ such 
that $c_1^*c_1=b_1$ and $c_1c_1^*=b_1'.$
Choose $y_1=b_1.$ There are mutually orthogonal and mutually Cuntz equivalent  
nonzero elements $b_{2,i}\in \Her(b_1')$  ($1\le i\le 2^2$).
Choose $y_2=b_{2,1}.$ 
We then proceed to divide $b_{2,4}.$ A standard induction argument 
produces the desired sequence $\{y_k\}.$

For each $k,$ there are $2^{k+1}$ mutually orthogonal non-zero elements $y_{k,1},y_{k,2},...,y_{k,2^{k+1}}$
in $\Her(y_k)$ 
and   elements
 $z_{k,1},z_{k,2},...,z_{k,2^{k+1}}$ in $\Her(y_k)$ such 
that
\beq\label{LsmallMn-1}
z_{k,j}^*z_{k,j}=y_{k,1}\andeqn z_{k, j}z_{k,j}^*=y_{k,j},\,\,\, 1\le j\le 2^{k+1}.
\eneq

Applying Lemma \ref{Losfull}, one obtains, for each $k\in \N,$
a sequence $\{e_{k,1, n}\}\subset  {\rm Her}(y_{k,1})_+^{\bf 1}$ 
such that  $\Pi(\{e_{k,1,n}\}_{n\in \N})$ is full in  
$l^\infty(A)/I_{_{\Qw}}$ and $\lim_{n\to\infty}\omega(e_{k,1,n})=0.$ 

Since $\lim_{n\to\infty}\omega(e_{k,1, n})=0,$   by (3) of Lemma \ref{Dpproj},
there is also, for each $k,$ a sequence $m(n,k)\in \N$ such that
\beq
p_{k,1}:=\Pi(\{e_{k,1,n}^{1/m(n,k)}\}_{n\in \N})
\eneq
is a full projection in $l^\infty(A)/I_{S,0}.$ 
Write $z_{k,j}=u_{k,j}|z_{k,j}|$ as polar decomposition of $z_{k,j}$ in $A^{**},$
($1\le j\le 2^{k+1}$). 
Put 
\beq
{{v_{k,j}=\Pi(\{u_{k,j}e_{k,1,n}^{1/m(n,k)}\})\andeqn p_{k,j}=v_{k,j}v_{k,j}^*,\,\,1\le j\le 2^{k+1}.}}
\eneq
Then $v_{k,j}^*v_{k,j}=p_{k,1}$ (see  \eqref{LsmallMn-1}). 
Thus we obtains mutually orthogonal and mutually equivalent 
full projections $p_{k,j},$ $j=1,2,...,2^{k+1}.$ 
By the construction, we also have $p_{k,j}p_{k',j'}=0,$ if $k\not=k',$ as well as 
$pp_{k,j}=0$ for all $k, j\in N.$ 
Put $p_k:=\sum_{j=1}^{2^{k+1}}p_{k,j}.$
Note also  
\beq\nonumber
2^{2k}[p_k]\le [p].
\eneq
\end{proof}

The following two lemmas are variations of S. Zhang's halving projection lemma.
We need some modification as we do not assume the \CA\, is simple.

\begin{lem}[S. Zhang, Theorem I.(i) of \cite{ZhMjot}]
\label{LsZhang}
Let $A$ be a \CA\, of real rank zero and $r$  a  full projection  of $A.$
Suppose that
 $p\in A$ is a non-zero projection such that $[p]\not\le  [r].$
Then, 
there are mutually orthogonal projections $p_1,p_2, p_3\in A$ such that
\beq
p=p_1+p_2+p_3,\,\, p_1\sim p_2,  \,\,
\andeqn p_3\lesssim r.
\eneq
\end{lem}

\begin{proof}
We begin with the following claim  which  is 
extracted from the proof of 
 \cite[Theorem I.(i)]{ZhMjot}. 

{\bf Claim:} Let $C$ be a \CA\ and let $v_1,...,v_{2^m}\in C$ be partial isometries 
such that $v_iv_i^*\perp v_jv_j^*$ ($i\neq j$)
and $v_i^*v_i\ge v_{i+1}^*v_{i+1}$ ($1\le i\le 2^m-1$).
Then there is a partial isometry $v\in C$ such that 
$v^*v\perp vv^*$ and 
$0\le \sum_{i=1}^{2^m} v_iv_i^*-(v^*v+vv^*)\le  v_1v_1^*.$ 

{\bf Proof of the claim:} We use induction on $m.$ 
When $m=1,$ let $v:=v_1v_2^*.$ Then $v^*v\perp vv^*$
because $v_1v_1^*\perp v_2v_2^*,$ and  $v^*v=v_2v_1^*v_1v_2^*=v_2v_2^*$
 because $v_1^*v_1\ge v_2^*v_2.$ 
Note that 
$vv^*=v_1v_2^*v_2v_1^*.$ 
Thus 
\beq
v_1v_1^*+v_2v_2^*-(v^*v+vv^*)=v_1v_1^*-
v_1v_2^*v_2v_1^*=v_1(1-v_2^*v_2)v_1^*
\le v_1v_1^*. 
\eneq
The last equation above also shows that 
$v_1v_1^*+v_2v_2^*-(v^*v+vv^*)$ is positive. Hence the claim holds 
for $m=1.$

Assume that the claim holds for $m\ge 1.$
Let $v_1,...,v_{2^{m+1}}$ be partial isometries 
such that $v_iv_i^*\perp v_jv_j^*$ ($i\neq j$)
and $v_i^*v_i\ge v_{i+1}^*v_{i+1}$ ($1\le i\le 2^{m+1}-1$).
Define $e_i:=v_i^*v_i$  and 
$w_i=v_i(e_i-e_{2^{m+1}-i+1}),$
$i=1,...,2^m.$
For $i=1,...,2^m-1,$  we have
\beq
w_i^*w_i&=&
(e_i-e_{2^{m+1}-i+1})e_i
(e_i-e_{2^{m+1}-i+1})
\\
&=&
e_i-e_{2^{m+1}-i+1}
\ge 
e_{i+1}-e_{2^{m+1}-i}
=w_{i+1}^*w_{i+1}.
\eneq
Note that  the above
also shows that $w_i^*w_i$ are projections for all $i.$ 
Since $w_iw_i^*\le v_iv_i^*,$  we have that $w_iw_i^*$ are mutually orthogonal.

Consider $w_i,$ $1\le  i\le 2^m.$
By induction, there is a partial isometry $w\in C$ such that 
$w^*w\perp ww^*$ and 
\beq
0\le \sum_{i=1}^{2^m} w_iw_i^*-(w^*w+ww^*)
\le  w_1w_1^*\ (\le v_1v_1^*).
\label{0116-5}
\eneq
Hence $ww=0.$ 
Note that
\beq\label{0116-4}
&&w^*w+ww^*
\le 
\sum_{i=1}^{2^m} w_iw_i^*
=\sum_{i=1}^{2^m}  
v_i
(e_i-e_{2^{m+1}-i+1})v_i^*
\\\nonumber
&&\hspace{0.3in}=
\sum_{i=1}^{2^m}  
v_iv_i^*-v_ie_{2^{m+1}-i+1}v_i^*
\in
\left(\sum_{i=1}^{2^m}  
v_ie_{2^{m+1}-i+1}v_i^*\right)^\perp\cap 
\left(\sum_{i=1}^{2^m}v_{2^{m+1}-i+1}v_{2^{m+1}-i+1}^*
\right)^\perp,
\eneq
where $b^\perp=\{a\in A: ab=ba=0\}.$
Recall that $v_{2^{m+1}-i+1}^*v_{2^{m+1}-i+1}=e_{2^{m+1}-i+1}$ ($1\le i\le 2^m$) and 
$v_i^*v_i\ge v_{i+1}^*v_{i+1}$ ($1\le i\le 2^{m+1}-1$).
Hence 
\beq
\sum_{i=1}^{2^m}  
v_ie_{2^{m+1}-i+1}v_i^*=\sum_{i=1}^{2^m}v_iv_{2^{m+1}-i+1}^*( v_iv_{2^{m+1}-i+1}^*)^*\andeqn\\
\sum_{i=1}^{2^m}v_{2^{m+1}-i+1}v_{2^{m+1}-i+1}^*=\sum_{i=1}^{2^m}( v_iv_{2^{m+1}-i+1}^*)^*v_iv_{2^{m+1}-i+1}^*.
\eneq
Thus
\beq\nonumber
&&\hspace{0.2in}ww^*+w^*w\in \left(
\sum_{i=1}^{2^m}v_iv_{2^{m+1}-i+1}^*( v_iv_{2^{m+1}-i+1}^*)^*\right)^\perp
\cap
\left(\sum_{i=1}^{2^m}( v_iv_{2^{m+1}-i+1}^*)^*v_iv_{2^{m+1}-i+1}^*
\right)^\perp.
\eneq
%
%
It follows that, for $1\le i\le 2^m,$
\beq
w\perp v_iv_{2^{m+1}-i+1}^*.
\label{1206-3}
\eneq
Define 
$v:=w+\sum_{i=1}^{2^m}v_iv_{2^{m+1}-i+1}^*.$
%
By \eqref{1206-3} and the fact that 
$w^2=0$ together with 
$v_iv_i^*\perp v_jv_j^*$ ($i\neq j$), we compute that 
\beq
\hspace{-0.4in}v^2&=&(w+\sum_{i=1}^{2^m}v_iv_{2^{m+1}-i+1}^*)(w+\sum_{i=1}^{2^m}v_iv_{2^{m+1}-i+1}^*)=0,\\
\hspace{-0.4in}v^*v&=&w^*w+
\sum_{i=1}^{2^m}v_{2^{m+1}-i+1}v_i^*v_iv_{2^{m+1}-i+1}^*
=w^*w+
\sum_{i=1}^{2^m}v_{2^{m+1}-i+1}v_{2^{m+1}-i+1}^*,\andeqn\\
%
\hspace{-0.4in}vv^*&=&ww^*+
\sum_{i=1}^{2^m}v_iv_{2^{m+1}-i+1}^*v_{2^{m+1}-i+1}v_i^*
=ww^*+
\sum_{i=1}^{2^m}v_ie_{2^{m+1}-i+1}v_i^*.
\eneq
Then 
\beq
\sum_{i=1}^{2^{m+1}}v_iv_i^*-(v^*v+vv^*)
&=&
\sum_{i=1}^{2^{m}}v_iv_i^*-\left(w^*w+ww^*+
\sum_{i=1}^{2^m}v_ie_{2^{m+1}-i+1}v_i^*\right)
\\
&=&
\sum_{i=1}^{2^{m}}
\left(v_iv_i^*-
v_ie_{2^{m+1}-i+1}v_i^*\right)
-\left(w^*w+ww^*\right)
\\
&\overset{\eqref{0116-4}}{=}&
\sum_{i=1}^{2^m} w_iw_i^*-
\left(w^*w+ww^*\right)
\\
&\overset{\eqref{0116-5}}{\le}&
w_1w_1^*\le v_1v_1^*.
\eneq
By induction, the claim holds for all $m\in\N.$

For the proof of the lemma,
applying Lemma 1.1 of \cite{ZhIII}, we obtain partial isometries 
$v_1, v_2,...,v_n\in A$ such that
\beq
&&r\ge v_1^*v_1\ge v_2^*v_2\ge \cdots v_n^*v_n\andeqn
\label{1106-6}
\\
&&p=v_1v_1^*\oplus v_2v_2^* \oplus \cdots   \oplus v_nv_n^*.
\label{1106-7}
\eneq
Since $[p]\not\le  [r],$ $n\ge 2.$ By adding $0$ if necessary, 
we may assume that $n=2^m$ for some $m\in\N.$
Then by \eqref{1106-6}, \eqref{1106-7}, and the claim, 
there is a partial isometry $v\in A$ such that 
$v^*v\perp vv^*,$ $0\le p-(v^*v+vv^*)\le v_1v_1^*\lesssim r.$
Then the lemma holds (by choosing $p_1=vv^*$ and $p_2=v^*v$).
\end{proof}

\begin{lem}[S. Zhang]\label{LsZhang-2}
Let $A$ be a \CA\, of real rank zero and $r$ be a  full projection  of $A$
such that $B=(1-r)A(1-r)$ contains, for each $k\in \N,$  
a sequence of mutually orthogonal 
full projections $\{r_{n,j}': 1\le j\le 2^{n+1}, n\in \N\}$ 
such that $2^{k+n}[r_n']\le [r],$ 
{{where}}
$r_n'=\sum_{j=1}^{2^{n+1}} r_{n,j}',$ 
and $r_{n,1}', r_{n,2}',...,r_{n,2^{n+1}}'$ 
are mutually equivalent ($n\in \N$).
%
Suppose that
 $p\in A$ is a non-zero projection such that $[p]\not\le  [r] .$
Then,  for any $m\in \N,$    there are mutually orthogonal projections $p_{_{1}},p_{_{2}}, ..., p_{_{2^m}}, p_{_{2^m+1}}\in A$ such that
\beq
p=\sum_{j=1}^{2^m} p_j+p_{_{2^m+1}},\,\, p_{_{j}}\sim p_{_{1}}, 1\le j\le 2^m, \andeqn p_{_{2^m+1}}\lesssim r+r',
\eneq
where 
$r'$ is a finite sum of $r_{n,j}' s$ and  
$2[r']<r.$
\end{lem}

\begin{proof}
We use the induction on $m.$
If $m=1,$ the lemma follows from Lemma \ref{LsZhang}.

Suppose that the lemma holds for $m{{\geq 1}}.$ 
Then
 there are mutually orthogonal projections $p_{_{1}},p_{_{2}}, ..., p_{_{2^m}}, p_{_{2^m+1}}\in A$ such that
\beq
p=\sum_{j=1}^{2^m} p_j+p_{_{2^m+1}},\,\, p_j\sim p_1, 1\le j\le 2^m, \andeqn p_{2^m+1}\lesssim r+r'',
\eneq
where 
$r''$ is a finite sum of $r_{n,i}'s$ and  
$2[r'']<[r].$
Choose $r_{K,1}'$ among $\{r_{n,j}'\}$ but not  those which have been used 
for  the sum of $ r''.$ 
We choose $K$ such that $K>2m.$ 
Note that 
$$
{{2{{[r'']}}
+2^{{{K}}+1}[r_{K,1}']<[r].}}
$$
We also note that $[p_1]\not\le  [r_{K,1}'].$ {{(Otherwise, 
$[p]\le 2^m[r_{K,1}']\le 2^K[r_{K,1}]\le [r],$ a contradiction.)}}
Applying Lemma \ref{LsZhang} to $p_1$ (as $p$) and the full projection $r_{K,1}'.$
Then 
we may write 
\beq
p_1=p_{1,1}+p_{1,2}+p_{1,3}, \,\,p_{1,1}\sim p_{1,2}\andeqn p_{1,3}\lesssim r_{K,1}'.
\eneq
Since $p_j\sim p_1,$ we also have  {{mutually orthogonal projections $p_{j,1}, p_{j,2},p_{j,3}$
such that}}
\beq
p_j=p_{j,1}+p_{j,2}+p_{j,3},\,\, p_{j,2}\sim p_{j,1}{{\sim p_{1,1}}}\andeqn p_{j,3}\lesssim r_{K,1}'.
\eneq
Note that 
\beq
p_{j,i}\sim p_{j',i'}, \,\,i,i'=1,2, \,\,j=1,2,....,2^m.
\eneq
Put $s:=\sum_{j=1}^{{{2^m}}}p_{_{j,3}}.$  Then
\beq
&&s\lesssim r+r''+\sum_{j=1}^{2^{{m}}} r_{K,j}', \,\,2[r''+\sum_{j=1}^{2^{{m}}}r'_{K,j}]\le [r]  \andeqn 
p=\sum_{j=1}^{2^m} (p_{j,1}+p_{j,2}) +s.
\eneq
This completes the induction.
\end{proof}

 Recall that, if $C_n\subset A$ ($n\in \N$), then $l^\infty(\{C_n\})=\{\{c_n\}\in l^\infty(A): c_n\in C_n\}.$

\begin{lem}\label{Ldvi-n1}
Let $A$ be a  non-elementary, $\sigma$-unital and  algebraically simple \CA\, 
with $QT(A)\not=\emptyset$ and with $T$-tracial {{approximate}} oscillation zero. 
Let $\{e_k\}$ be a sequence in $A_+^{\bf 1}\setminus \{0\}$   such  that   $\{e_k\}\not\in I_{_{\Qw}},$
and  $\{e_k\}$ is a permanent projection lifting of  the projection $p=\Pi(\{e_k\}),$ and 
 let $n>1$ be an integer and $\ep>0.$ 

Then, there are mutually 
orthogonal and mutually (Cuntz) equivalent elements 
$${{\{f_{k,1}\},\{f_{k,2}\},...,\{f_{k,n}\}\in 
 l^\infty (\{\Her(e_k)\})
\,\,\,{\text{such\,\, that}}}}$$
\beq\label{726}
\lim_{k\to\infty}\omega(f_{k,j})=0,\,\,1\le j\le n,\,\, and,
\eneq
for $k\ge k_0$ (for some $k_0$),
\beq\label{763}
\sup\{|d_\tau(e_k)-nd_\tau(f_{k,1})|:\tau\in \Qw\}<\ep.
\eneq 
\end{lem}

\begin{proof}
Let $m\ge 1$ be an integer {{and let $\ep>0.$}}    Put $B=l^\infty(A)/I_{_{\Qw}}.$ 
Since $A$ has T-tracial oscillation zero, by 
Theorem \ref{PLrr0},  $B$ has real rank zero. 
Since
$\{e_k\}\not\in I_{_{\Qw}},$
\beq
\sigma_0:=\limsup_k(\sup\{\tau(e_k): \tau\in \Qw\})>0.
\eneq
If $r=\Pi(\{r_n\})$  (is full) 
 and $r_n\in A_+^{\bf 1},$ and 
\beq
\sup\{\tau(r_n): \tau\in \Qw\}<\sigma_0/4\rforal n\in \N,
\eneq
then $[\Pi(\{e_k\})]\not\le [r].$ 
It  then follows from Lemma \ref{LsmallMn}  and Lemma \ref{LsZhang-2} that, there 
are $2^{m}+1$ mutually orthogonal projections 
$p_1, p_2,...,p_{2^m}, s\in B$ such that {{$s$  is full, $s=\Pi(\{s_n\}),$
where $s_n\in A_+^{\bf 1},$ }} 
\beq
&&p=\sum_{i=1}^{2^m}p_i+s,\,\, p_1\sim p_j, 1\le j\le 2^m\andeqn\\
&&\sup\{d_\tau(s_n):  \tau\in \Qw\}<\ep/2.
\eneq
Recall that $\{e_{{k}}\}$ is  a permanent projection lifting of $p.$
Then, by (5) of {{Proposition}} 
\ref{Dpproj},  $pBp=
%
 l^\infty(\{\Her(e_{{k}})\})/I_{_{\Qw}} 
.$
Define a \hm\, $\phi: C_0((0,1])\otimes M_{2^m}\to pBp$ such that 
\beq
\phi(\imath\otimes e_{i,i})=p_i,\,\,i=1,2,...,2^m.
\eneq
Since $C_0((0,1])\otimes M_{2^m}$ is semiprojective, 
there is a \hm\, 
\beq
{{\psi: C_0((0,1])\otimes M_{2^m}\to 
 l^\infty(\{\Her(e_n)\}) 
\,\,\,{\rm 
such\,\,\, that}}}\,\,\,
\Pi(\psi(e_{i,i}))=p_i,\,\,i=1,2,...,2^m.
\eneq
Write $\psi=\{\Psi_k\},$ where $\Psi_k: C_0((0,1])\otimes M_{2^m}\to {\rm Her}(e_{{k}})$
is a \hm. Put 
$$
{{g_{k,i}={{f_{1/4}}}(\Psi_k(e_{i,i})){{\in \Her(e_k)}},\,\,\, i=1,2,...,2^m,\,\, k\in \N.}}
$$ 
{{Then $\{g_{k,1}\},\{g_{k,2}\},...,\{g_{k,2^m}\}$  are mutually orthogonal and mutually 
equivalent.}}
{{Note that $\Pi(g_{k,i})\}=p_i$ and $\{g_{k,i}\}$ is a permanent projection lifting
of $p_i,$ and by (2) of Lemma \ref{Dpproj}, $\lim_{k\to \infty}\omega(g_{k,i})=0.$}}

Then, by Lemma \ref{Leqv-2},  {{there is $k_1\in \N,$ such that,}} for all $k\ge k_1,$ 
\beq\label{77-n1}
d_\tau(e_k)&<&
d_\tau(\sum_{j=1}^{2^m}g_{k,j})+d_\tau({{s_k}})+\ep/2\le 2^md_\tau(g_{k,1})+\ep
\rforal \tau\in \Qw.
\eneq
Also, 
 by Lemma \ref{Leqv-2}, 
we may  assume that, for all $k\ge k_1,$ 
\beq\label{77-n2}
2^md_\tau(g_{k,1})
\le d_\tau(\sum_{j=1}^{2^m}g_{k,j}+{{s_k}})\le d_\tau(e_k)+\ep\rforal \tau\in \Qw.
\eneq
{{We choose a large $m$ such that
$2^m=ln+m_0,$ where $l\in \N$ and $m_0\in {{\N\cup\{0\}}}$ 
such that $m_0/2^m<\ep/4.$ 
Note that, since $\{g_{k,1}\}, \{g_{k,2}\},...,\{g_{k,2^m}\}$ are mutually orthogonal, 
for any  sum  $f_{k,j}$ of  some $l$ many $\{g_{k, i}\}'s,$ 
$\lim_{k\to\infty} \omega(f_{k,j})=0.$  
For each $1\le j\le n,$ {{by \eqref{77-n1} and \eqref{77-n2}, and}}
by choosing  $f_{k,j}$ to be a sum of $l$ many (different) $g_{k,i}'s,$ we see 
\eqref{726} and \eqref{763} hold.}}
%

\end{proof}

\begin{cor}\label{Ldvi-n5}
Let $A$ be a non-elementary, $\sigma$-unital and  algebraically  simple \CA\, with
$QT(A)\not=\emptyset$ and T-tracial approximate oscillation zero. 
Let $e_k\in A_+^{\bf 1}\setminus \{0\}$ ($k\in  \N$)  be such  that   $\{e_k\}\not\in I_{\Qw}$ 
and $\{e_k\}$ is a permanent projection lifting of $p=\Pi(\{e_k
\}),$
and let $n>1$ be an integer and $\ep>0.$ 

 Then,  there is $k_0\in \N$ such that, for any $k\ge k_0$ and for any $1\le i\le n,$ there exists ${{h}}_{k,i}\in {{\Her(e_k)_+}}$
 such that
 \beq
 \sup\{| {i\over{n}}d_\tau(e_k)-d_\tau({{h_{k,i}}})|:\tau\in \Qw\}<\ep.
 \eneq

\end{cor}

\begin{proof}
By the proof of Lemma \ref{Ldvi-n1}, for $k\ge k_0,$
 \beq
 \sup\{|{1\over{n}}d_\tau(e_k)-d_\tau(f_{k,1})|: \tau\in \Qw\}<\ep/n.
 \eneq
So, for any $1\le i\le n,$  choose ${{h_{k,i}}}=\sum_{j=1}^i f_{k,j}.$ 
Then 
\beq\nonumber
\sup\{| {i\over{n}}d_\tau(e_k)-d_\tau({{h_{k,i}}})|:\tau\in \Qw\}<\ep.
\eneq
\end{proof}

\begin{df}\label{RRtau}
Let $A$ be a \CA\, with ${\wtd{QT}}(A)\setminus \{0\}\not=\emptyset.$
Define 
\beq
R_{\tau, f}(A)=\{\wh{a}: a\in {\rm Ped}(A\otimes {\cal K})_+\}\subset \Aff_+({\wtd{QT}}(A)).
\eneq
\end{df}

\begin{lem}\label{Tdense}
Let $A$ be a non-elementary,  $\sigma$-unital  and  simple \CA\, with ${\wtd{QT}}(A)\setminus \{0\}\not=\emptyset.$
Suppose that $A$ has $T$-tracial {{approximate}} oscillation zero.
Then the image of the canonical map $\Gamma$ (see \ref{DGamma}) is dense in 
$R_{\tau, f}(A).$ 

\end{lem}

\begin{proof}
Fix {{a nonzero}}  element $0\le e\le 1$ in ${\rm Ped}(A)_+.$ 
Let $A_1=\Her(e).$ Then $A_1={\rm Ped}(A_1).$ 
By Brown's stable isomorphism theorem  (\cite{Br}), $A_1\otimes {\cal K}\cong A\otimes {\cal K}.$ 
It suffices to show that the image of the map $\Gamma_1: \Cu(A_1)=\Cu(A)\to {\rm LAff}_+(\Qwy)$
is dense in 
$${{R_{\tau, f}(A_1)=\{\hat{a}: a\in {\rm Ped}
(A_1\otimes {\cal K})_+\}\subset \Aff_+(\Qwy)}}\hspace{0.4in}{\rm (see\,\, \ref{DGamma}).}$$
%


Fix $a\in {\rm Ped}(A_1\otimes {\cal K})_+.$ 
Let $\ep>0.$ It suffices to show that there is $f\in \Cu(A_1)$ such that
\beq\label{Tdense-0}
\sup\{|\tau(a)-\widehat{f}(\tau)|:\tau\in \Qwy\}<\ep.
\eneq
\Wlog, we may assume that $0\le a\le 1.$
Since $a\in {\rm Ped}(A_1\otimes {\cal K})_+,$  
there exists $r\ge 1$  such that $r[f_{\dt}(e)]\ge [a]$
for some $\dt\in (0,1/4).$ Therefore we may assume 
that $a\in M_r(A_1)$ for some integer $r\ge 1.$

Put $B:=l^\infty(A_1)/I_{_{\Qwy}}.$ 
Then, by Theorem \ref{PLrr0}, 
$B$ has real rank zero. 

Therefore, for any $\ep>0,$ 
there are mutually orthogonal projections $p_1, p_2,...,p_m\in M_r(B)$ and 
$\lambda_1, \lambda_2,...,\lambda_m\in (0, 1)$ such that
\beq\label{Tdense-1}
\|\Pi(\iota(a))-\sum_{i=1}^m\lambda_i p_i\|<\ep/16.
\eneq
We may assume that $\lambda_i\in (0,1)\cap \Q.$
There are  mutually orthogonal elements $\{e_{n,i}\}\in l^\infty( M_r(A_1) )$ ($i=1,2,...,m$)
such that $\Pi(\{e_{n,i}\})=p_i,$ $i=1,2,...,m.$  
By Proposition 
\ref{Dpproj}, we may assume that 
$\{e_{n,i}\}$ is a permanent projection lifting.
By (2) of Proposition 
\ref{Dpproj},  $\lim_{n\to\infty}\omega(e_{n,i})=0.$ 
\Wlog, we may assume that, for all $n\in \N,$
\beq\label{Tdense-2}
d_\tau(e_{n,i})-\tau(e_{n,i})<\ep/16(m+1)^{{2}}\rforal \tau\in {{\Qwy,}}\,\,i=1,2,...,m.
\eneq

Applying Corollary 
\ref{Ldvi-n5}, \wilog, we may also assume that, 
there are permanent projection liftings $\{f_{n,i}\}$ such that
\beq\label{Tdense-4}
\sup\{|\lambda_id_\tau(e_{n,i})-d_\tau(f_{n,i})|:\tau\in {{\Qwy}}\}<\ep/16(m+1)^{{2}},\,\, i=1,2,...,m.
\eneq

By \eqref{Tdense-1},  there exists $\{c_n\}\in I_{_{\Qwy}}$  and $n_1\in \N$ such that, for all $n\ge n_1,$ 
\beq\label{Tdense-5}
\|a-\sum_{i=1}^m \lambda_ie_{i,n}+c_n\|<\ep/8.
\eneq
%
Then, for $n\ge n_1,$
$${{\sup\{|\tau(a-\sum_{i=1}^m \lambda_ie_{n,i}+c_n)|:\tau\in \Qwy\}<\ep/8.}}$$
Since $\{c_n\} \in I_{_{\Qwy}},$ we have $\{|c_n|^{1/2}\}\in I_{{_{\Qwy}}}.$ 
It follows that
$$
\lim_{n\to \infty} 
\sup\{|\tau(c_n)|:\tau\in {{\Qwy}}\}=0.$$
Therefore there exists $n_2\ge n_1$ such that
\beq\label{Tdense-8}
|\tau(a-\sum_{i=1}^m \lambda_ie_{n,i})|<|\tau(c_n)|+\ep/8<\ep/4\rforal \tau\in \Qwy\,\,{\rm and} \rforal n\ge n_2.
\eneq
Thus, by \eqref{Tdense-2} and \eqref{Tdense-4},
for $n\ge n_2,$ 
\beq
|\tau(a)-\sum_{i=1}^m d_\tau(f_{n,i})|<\ep/2\rforal \tau\in \Qwy.
\eneq
Put $e=\diag(f_{n,1}, f_{n,2},...,f_{n,m}).$  Then 
\beq 
|\tau(a)-d_\tau(e)|<\ep\rforal \tau\in \Qwy.
\eneq
This completes the proof.
\end{proof}

\begin{thm}\label{TOstosurj-1}
Let $A$ be a non-elementary and  $\sigma$-unital simple \CA\, with $\wtd{QT}(A)\setminus \{0\}\not=\emptyset$
and  strict comparison. 
Suppose that $A$  has T-tracial approximate oscillation zero.
Then $\Gamma$ is surjective (see \ref{DGamma}).
\end{thm}

\begin{proof}

We keep the same setting as  in the proof of Theorem \ref{Tdense}.

Let $b\in {\rm Ped}(A_1\otimes {\cal K})_+$ with $0\le b\le 1.$ 
Choose 
 $b_n=(1-1/(n+1))b.$ 
Then  $h_n:=b_{n+1}-b_n\in {\rm Ped}(A_1\otimes {\cal K})_+\setminus \{0\}.$
Put
\beq
\sigma_n:=\inf\{\tau(h_n): \tau\in \Qwy\}>0.
\eneq
Applying Theorem \ref{Tdense}, for each $n\in \N,$ we obtain  $f_n\in {\rm Ped}(A_1\otimes {\cal K})_+$ such that
\beq
\eta_n:=\sup\{|\tau(b_n)-d_\tau(f_n)|: \tau\in \Qwy\}<{\min\{\sigma_j: 1\le j\le n+1\}\over{2^{n+2}}}.
\eneq
In particular,   for all $n\in \N,$ 
\beq
\label{TOstosurj-1-e1}
d_\tau(f_n)<\eta_n+{{\tau}}(b_n)<\tau(b) \rforal \tau\in {{\Qwy.}}
\eneq

Then, for all $\tau\in \Qwy$   (note that $b_nb_{n+1}=b_{n+1}b_n$), 
\beq
d_\tau(f_{n+1})-d_\tau(f_n)&>&(\tau(b_{n+1}) -\eta_{n+1})-(\tau(b_n)+\eta_n)\\
&=& {{\tau}}(h_{n})-\eta_{n+1}-\eta_n>
\tau(h_{n})-\sigma_n
/2>0.
\eneq
 It follows from the strict comparison  that $[f_n]$ is an increasing sequence in $\Cu(A).$
 Let $f$ be the supremum of $\{[f_n]\}$ in $\Cu(A).$
 {{We also have,}} for all $\tau\in \Qwy,$
 \beq
 d_\tau(f_{n+1})-{{\tau}}(b_n) &>&{{\tau}}(b_{n+1})-\eta_{n+1}-{{\tau}}(b_n)\\
&\ge
 &{{\tau}}(h_n)-\sigma_{n}/2^{n+1}>0.
 \eneq
It follows that $\widehat{f}(\tau)\ge \tau(b_n)$ for all $\tau\in \Qwy$ {{and}} for each $n\in \N.$ 
Hence 
$$
{{\widehat{f}(\tau)\ge \tau(b)\rforal \tau\in \Qwy.}}
$$
Let $\ep>0.$ By Theorem \ref{Tdense}, 
there is $c_\ep\in \Cu( A_1 )$ such that
\beq
\sup\{|\tau(b+(\ep/2)b )-
 \Gamma_1 (c_\ep)(\tau)|:\tau\in \Qwy\}<
 \ep\sigma_1/8. 
\eneq
Then, for all $n\in \N$ (see also \eqref{TOstosurj-1-e1}),
\beq
 \Gamma_1 (c_\ep)(\tau)>\tau(b+ 
 \ep/2  b)-\ep\sigma_1/8
>\tau(b)>d_\tau(f_n)\rforal \tau\in \Qwy.
\eneq
Hence $[f]\le [c_\ep].$
It follows that
\beq
\widehat{f}(\tau)\le \widehat{c_\ep}(\tau)<\tau(b)+\ep\rforal \tau\in {{\Qwy.}}
\eneq
Let $\ep\to 0.$ We conclude that $\widehat{f}(\tau)=\tau(b)$ for all $\tau\in \Qwy.$

So far we {{have shown}} that, for any $b\in {\rm Ped}(A\otimes {\cal K})_+,$ 
there is $f\in \Cu(A)$ such that $\widehat{f}(\tau)=\tau(b)$ for all $\tau\in \Qw.$ 
Note that, for any $a\in  (A\otimes {\cal K})_+,$  $(a-\|a\|/n)_+\in {\rm Ped}(A\otimes {\cal K}).$
Thus, there are $f_n\in \Cu(A)$ such that $\widehat{f_n}(\tau)=\tau((a-\|a\|/n)_+)$ for all $\tau\in \Qw.$
Since $(a-\|a\|/n)_+\nearrow a,$ we conclude, using the similar argument used above, 
that there is $f\in \Cu(A)$ such that $\widehat{f}(\tau)=\tau(a)$ for all $\tau\in \Qw.$ 
Applying Theorem 5.7 of \cite{ERS} and repeating the argument above, we conclude that
$\Gamma$ is surjective.
\end{proof}

\section{Tracially matricial  property}

\begin{df}\label{DTM}

Let $A$ be a
\CA\,   and $S\subset \wtd{QT}(A)\setminus 
 \{0\}$ be a non-empty compact subset.
\CA\, $A$ is said to have property  (TM) relative to $S,$ 
if for any $a\in  {\rm Ped}(A\otimes {\cal K})_+,$  any $\ep>0,$
any $n\in\N,$
there is a c.p.c.~order zero map 
$\phi: M_n\to \Her(a)$
such that $\|a-\phi(1_n)a\|_{_{2,S}}<\ep.$

A $\sigma$-unital simple \CA\, $A$  with $\wtd{QT}(A)\setminus \{0\}\not=\emptyset$
is said to have  property 
(TM), if for some $e\in {\rm Ped}(A)_+^{\bf 1}\setminus \{0\},$ $\Her(e)$ has property (TM) relative to $\Qw.$

From the definition, it is clear that, if $A$ is a $\sigma$-unital simple \CA\, which has property (TM),
then $A\otimes {\cal K}$ has property (TM), and 
by Brown's stable isomorphism theorem \cite{Br}, 
every $\sigma$-unital hereditary  \SCA\, has property (TM).  Since we only 
need to consider $\Her(a)$ for each $a\in {\rm Ped}(A\otimes {\cal K})_+^{\bf 1},$ it  follows 
that every  hereditary \SCA\, has property (TM).

In the absence of strict comparison, one may also define the following:

A \CA\ $A$ is said to have property (CM), if, for any $n\in \N$  and any $a\in {\rm Ped}(A\otimes {\cal K})_+^{\bf 1},$  there is  
 a c.p.c.~order zero map $\phi: M_n\to l^\infty(C)/N_{cu}(C)$ such that
 \beq
 \iota(b)\phi(1_n)=\iota(b)\rforal b\in C,
 \eneq
where $C=\overline{a(A\otimes {\cal K})a}.$

\end{df}
\begin{rem}\label{R82}
In {{definition}} \ref{DTM}, let $\psi: C_0((0,1])\otimes M_n\to \Her(a)$ 
be the \hm\, {{induced}} by $\phi,$ i.e.,  $\psi(\iota\otimes e_{i,j})=\phi(e_{i,j})$ ($1\le i, j\le n$). 
Then $\Her(a)$ has some ``tracially  large"  matricial structure (see Proposition \ref{p1028-1}).

{{Let $x\in {\rm Ped}(A\otimes {\cal K}).$  Then $a_0=x^*x+xx^*\in {\rm Ped}(A\otimes {\cal K})_+.$
Note $x\in \Her(a_0).$ 
{{Let ${\cal F\subset} \Her(a_0)^{\bf 1}$ 
be a finite set.}}
For any $1>\ep>0,$ 
there  is $a\in \Her(a_0)_+^{\bf 1}$ 
such that
\beq
\|ay-y\|< (\ep/4)^3 \andeqn \|ya-y\|<(\ep/4)^3\rforal y\in {\cal F}.
\eneq
If $A$ has property (TM)  relative to $S$ such that $\|\tau\|\le 1$ for all $\tau\in S,$ 
  then,  there is a c.p.c.~order zero map
$\phi: M_n\to \Her(a_0)$ such that
$
\|a-\phi(1_n)a\|_{_{2,  S} }<(\ep/2)^3.
$
Then, for all $y\in {\cal F},$
\beq
\|y-\phi(1_n)y\|_{_{2, S}}^{2/3}&\le& \|y-ya\|^{2/3}_{{_{2, S}}}+\|ya-ya\phi(1_n)\|_{_{2, S}}^{2/3}\\
&<&(\ep/2)^2+\|y\|\|a-a\phi(1_n)\|_{_{2, S}}^{2/3}<\ep.
\eneq
Similarly
\beq
\|y-\phi(1_n)y\|_{_{2, S}}^{2/3}<\ep.
\eneq}}

\end{rem}

The following fact is well known. 
For completeness, we include a proof here. 
\begin{prop}
\label{p1028-1}
Let $A$ be a \CA, $n\in\N,$
and
$\phi:M_n\to A$ be a c.p.c.~order zero map. 
Then $\Her(\phi(1_n))\cong \Her(\phi(e_{1,1}))\otimes M_n.$
\end{prop}
\begin{proof}

Let $\psi: C_0((0,1])\otimes M_n\to A$ be the \hm\, defined 
by $\psi(\imath\otimes e_{ij})=\phi(e_{i,j}),$ where $\imath$ is the identity function on $(0,1],$ $i,j\le n.$
In particular, $\psi(\imath\otimes 1_n)=\phi(1_n).$
Write $\psi(\imath\otimes e_{i,j})=u_{i,j}r_{{j}}$ as a polar decomposition of $\phi(\imath\otimes e_{i,j})$
in $A^{**}.$ {{Hence}} $r_{{j}}=|{{\psi}}(\imath\otimes e_{i,j})|$ and $u_{i,j}$ is a partial isometry in 
$A^{**}.$  
Note that $au_{i,j}b\in A$ for all $a\in p_iA^{**}p_i\cap A$ and $b\in p_jA^{**}p_j\cap A,$
where $p_i$ is the open projection of $r_i,$ $i,j=1,2,...,n.$ 
Since $\psi$ is a \hm, we compute that $p_i=u_{i,i},$ $i=1,2,...,n,$ and 
$\{u_{i,j}\}_{1\le i,j\le n}$ forms  a system of matrix units for $M_n.$ 

 Define $\Phi: M_n({\rm Her}({{\phi}}(e_{1,1})))\cong {\rm Her}(\phi(e_{1,1}))\otimes M_n\to {\rm Her}(\phi(1_n))$ by
 defining $\Phi(a\otimes e_{i,j})=u_{i,1}au_{1,j}$ for all $a\in {\rm Her}({{\phi}}(e_{1,1})),$ $i,j=1,2,...,n.$ 
 Then $\Phi$ is a \hm. Since $\Phi$ is the identity map on ${\rm Her}{{(\phi(e_{1,1}))}}$ and $M_n$  is simple,
 {{the map}} $\Phi$ is injective.   Put $B=\Phi(M_n({\rm Her}({{\phi}}(e_{1,1})))).$ To see $\Phi$ is surjective, let $x\in A,$ 
 then, for any $i,j=1,2,...,n,$
$b_{i,j}=u_{1,i}r_ixr_ju_{j,1}\in {\rm Her}(\phi(e_{1,1}){{)}}.$ Therefore
 \beq
 \phi(1_n)x\phi(1_n)&=&\sum_{{1\le i, j\le n}} u_{i,i}r_ixr_ju_{j,j}\\
 &=&\sum_{1\le i,j\le n} u_{i,1}(u_{1,i}xr_ju_{j,1})u_{1,j}=\sum_{1\le i,j\le n}u_{i,1}b_{i{{,}}j}u_{1,j}\in B.
 \eneq
It follows that  ${\rm Her}(\phi(1_n))=B.$  The lemma follows.
\end{proof}

\begin{lem}\label{LTVD-2}
Let $A$ be a simple \CA\, with ${\wtd{QT}}(A)\setminus \{0\}\not=\emptyset.$
Suppose that  $A={\rm Ped}(A),$  $A$ has strict comparison and  $\Gamma$ is surjective.
Suppose that $b\in {\rm Ped}(A\otimes {\cal K})_+^{\bf 1}.$
Then, for any $\ep>0$ and any integer $n\ge 1,$ there is a c.p.c. order zero map $\phi: 
M_n\to {\rm Her}(b)$ such that
\beq
\|b-\phi(1_n)b\|_{_{2, \Qw}}<\|b\| \sqrt{\omega(b)+\ep}.
\eneq
Moreover, if $QT(A)=T(A),$   {{then}}
\beq\label{LTVD-2-e}
\|b-\phi(1_n)b\|_{_{2, \Tw}}<{{\min\{\|b\|, \|b\|_{{_{2, \Tw}}}\}}}\sqrt{\omega(b)+\ep}.
\eneq
\end{lem}

\begin{proof}
Fix $\ep\in (0,1/2)$ and $n\in \N.$ Since $\Gamma: {\rm Cu}(A)\to {\rm LAff}_+(\Qw)$ is surjective,
$\Gamma|_{{\rm Cu}(A)_+}$ is also surjective (see \eqref{DGamma}).
Note that $(1/n)\wh{[b]}\in {\rm LAff}_+(\Qw).$
We may choose $b_1\in  (A\otimes {\cal K})_+^{\bf 1}$ 
such that $d_\tau(b_1)=(1/n)d_\tau(b)$ for all $\tau\in \Qw$ and $b_1$ is not Cuntz equivalent 
to a projection.
We compute that $\omega(b_1)=(1/n)\omega(b)$ (see \eqref{Eosc-1226}).
Choose $\dt>0$ such that
\beq
\sup\{d_\tau(b_1)-\tau(f_{2\dt}(b_1)):\tau\in 
{{\Qw}}\}<\omega(b_1)+\ep/4n.
\eneq
Put $b_2=f_\dt(b_1)$ and 
%
$
d_1=\diag(\overbrace{b_1, b_1,...,b_1}^n)\in A\otimes {\cal K}.
$
Then 
\beq
f_\dt(d_1)=\diag(\overbrace{{{b_2, b_2,....,b_2}}}^n)\in A\otimes {\cal K}.
\eneq
Since $b_1$ is not Cuntz equivalent to a  projection,  
{{for any}} $0<\eta<\dt/2,$ 
$d_\tau(f_\eta(d_1))<d_\tau(b)$  for all $\tau\in \Qw.$
Since $A$ has strict comparison, 
by 
\cite[Proposition 2.4 (iv)]{Ror92UHF2}, there is $x\in A\otimes {\cal K}$ such that
\beq
x^*x=f_\dt(d_1)\andeqn xx^*\in {\rm Her}(b).
\eneq
Then one obtains an isomorphism 
$$
{{\psi: {\rm Her}(x^*x)\to {\rm Her}(xx^*)\subset {\rm Her}(b) \,\,\,
{\rm{such\,\, that}} \,\,\, \psi(f(x^*x))=f(xx^*)\rforal f\in C_0((0, \|x\|^2]).}}
$$ 
It induces a \hm\,  $\phi_c:  C_0({\rm sp}(f_\dt({{b_1}})))\otimes M_n\to {\rm Her}(b)$
such that $\phi_c(\imath\otimes 1_n)=xx^*,$   where $\imath\in C_0(f_\dt({{b_1}}))$ is the identity function on ${\rm sp}({{f_{\dt}(b_1)}}).$
Define 
a c.p.c.~order zero map
$\phi: M_n\to \Her(b)$ by $\phi(e_{i,j})=\phi_c(\imath\otimes {{e_{i,j}}})$ ($1\le i,j\le 1$).

Let $p$ be the open projection in $A^{**}$ corresponding to $b$ which  may be identified 
{{with}}  the identity of $\widetilde{\Her(b)}.$ 
We extend each $\tau\in   \Qw$ to a 2-quasitrace on $\widetilde{\Her(b)}$ (see II.2.5 of \cite{BH})
such that $\tau(p)=\|\tau|_{\Her(b)}\|=d_\tau(b).$
Then, we have, for all $\tau\in \Qw,$ 
\beq
\tau((p-\phi(1_n))^2)\le \tau(p-\phi(1_n))=\tau(p)-\tau(\phi(1_n))\\\label{814}
=d_\tau(b)-\tau(f_\dt(d_1))<\omega(b)+\ep.
\eneq
It follows that
\beq
\|b-\phi(1_n)b\|_{_{2, \Qw}}\le \|b\|\|p-\phi(1_n)\|_{_{2, \Qw}}< \|b\|\sqrt{\omega(b)+\ep}.
\eneq
In the case that $QT(A)=T(A),$ one 
can also apply  Cauchy--Bunyakovsky--Schwarz  inequality  {{(and \eqref{814}) to}} 
obtain \eqref{LTVD-2-e}.
\end{proof}

\begin{thm}\label{TVD-1}
Let $A$ be a {{$\sigma$-unital}} simple {{non-elementary}} \CA\, with 
$QT(A)\setminus \{0\}\not=\emptyset.$
Suppose that $A$ has strict comparison
and  T-tracial approximate oscillation zero. 
Then $A$ has property (TM).
\end{thm}

\begin{proof}
{{Choose $e\in {\rm Ped}(A)_+\setminus \{0\}$ and define 
$A_1=\Her(e).$   By Brown's stable isomorphism theorem (\cite{Br}), $A_1\otimes {\cal K}\cong A\otimes {\cal K}.$
To show that $A$ has property (TM), it suffices to show that $A_1$ has property (TM).
To simplify notation, \wilog, we may assume that $A={\rm Ped}(A).$}}

Fix $a\in {\rm Ped}(A\otimes {\cal K})$ such that $0\le a\le 1.$
{{Since  $A$ has T-tracial approximate oscillation zero,}}  $\Omega^T(a)=0.$   {{It follows that}} there is 
a sequence $c_k\in {\rm Her}(a)$ with $0\le c_k\le 1$ such 
that 
\beq
\Pi(\iota(a))=\Pi(c)\andeqn \lim_{k\to\infty}\omega(c_k)=0,
\eneq
where $c=\{c_k\}$ and $\Pi: l^\infty(A)\to l^\infty(A)/I_{_{\Qw}}$ is the quotient map.
{{By Theorem \ref{TOstosurj-1}, $\Gamma$ is surjective.  Then, applying}} 
 Lemma  \ref{LTVD-2},  we obtain a c.p.c.~order  zero
 map  $\phi_k: M_n\to {\rm Her}(
 c_k
 )\subset {\rm Her}(a)$ 
such that
\beq
\|c_k-\phi_k(1_n)c_k\|_{_{2, \Qw}}\le \sqrt{\omega(c_k)+1/k^2},\,\,k=1,2,....
\eneq
Fix $1>\ep>0.$  Choose $k_0\ge 1$ such that 
\beq
\sqrt{\omega(c_k)+1/k^2}<{{(\ep/3)^3}}.
\eneq
Since $\Pi(\iota(a))=\Pi(c),$   there exists $k_1\ge k_0$ such that
\beq
\|a-c_k\|_{{_{2, \Qw}}}<(\ep/3)^3\rforal k\ge k_1.
\eneq
Choose $b=c_{k_1+1}.$     Then, {{for $k\ge k_1,$}}
\beq\nonumber
\|a-\phi_k(1_n)a\|_{_{2, \Qw}}^{2/3} &\le& \|a-b\|_{_{2, \Qw}}^{2/3}
+\|b-\phi_k(1_n)b\|_{_{2, \Qw}}^{2/3}\\\nonumber
&&+\|\phi_k(1_n)(b-a)\|_{_{2, \Qw}}^{2/3}\\\nonumber
&<& \ep/3+\ep/3+\ep/3<\ep.
\eneq
%
%
\end{proof}

\begin{lem}\label{LTM-1}
Let $A$ be a $\sigma$-unital  algebraically  simple \CA\,
with $QT(A)\not=\emptyset.$
Suppose that $A$ has property (TM) and $a\in {\rm Ped}(A\otimes {\cal K})_+^{\bf 1}\setminus \{0\}.$
Then, for any integer $n, r\in \N,$ any $k\in \N$ {{($k\ge 2$)}} and $\ep>0,$  there exist an integer $m_k\geq r\in \N$ and 
mutually orthogonal elements $b_{k,1}, b_{k,2},...,b_{k,n}\in \Her(a)_+^{\bf 1}$ such that
$[b_{k,i}]=[b_{k,1}],$ $i=1,2,...,n,$ 
\beq\label{LTM-1-e1}
d_\tau(f_{1/k}(a^{1/m_k})){{<}} nd_\tau(f_{1/k}(b_{k,1}))+\ep\tforal \tau\in \Qw
\eneq
and, an integer $ l(k) \in \N$ such that
\beq\label{LTM-1e-2}
nd_\tau(f_{1/k}(b_{k,1}))\le d_\tau
(f_{1/l(k)}(a))\tforal \tau\in \Qw.
\eneq
\end{lem}

\begin{proof}
Fix $n\in \N.$  
Since $A$ has property (TM), for each $m\in \N,$ there exists a c.p.c.~order zero map
$ \phi_m : M_n\to {\rm Her}(a)$ such that 
\beq
\|a^{1/m}-a^{1/m}\phi_m(1_n)\|_{_{2, \Qw}}<1/2^{m}.
\eneq

Define  $c:=\Pi(\{a^{1/m}\}{{_{m\in\mathbb{N}}}})$ and $\phi: M_n\to l^\infty(A)/I_{_{\Qw}}$ 
such that $\phi(f)=\Pi(\{\phi_m(f)\})$ for all $f\in M_n.$
Then
\beq
c=c\phi(1_n)=\phi(1_n)c=c.
\eneq
It follows that $c=cf_{1/2}(\phi(1_n))=f_{1/2}(\phi(1_n))c=c.$ 
Thus $c\le f_{1/2}(\phi(1_n)).$  

Let $\ep>0.$ 
By (1) of Lemma \ref{Leqv-2},  for each {{integer $k\ge 2,$}} there exists $m_k{{\geq r}}\in \N$ such that, for all $m\ge m_k,$  
\beq
[f_{1/k}(a^{1/m})]&\le &  [f_{1/2}(\phi_m(1_n))]+[d_m]\\
&\le &  [f_{1/k}(\phi_m(e_{1,1}))+f_{1/k}(\phi_m(e_{2,2}))+\cdots f_{1/k}(\phi_m(e_{n,n}))]+[d_m],
\eneq
where  $\sup\{d_\tau(d_m): \tau\in \Qw\}<\ep/2.$
Put $b_{k,i}=\phi_{m_k}(e_{i,i})$ $1\le i\le n$ and $k\in \N.$
{{Then \eqref{LTM-1-e1} holds.}}
On the other hand, since $\phi_{m_k}(1_n)\in \Her(a),$ 
for each $k,$  there is $l(k)\in \N$ such that
\beq
\|f_{1/l(k)}(a)\phi_{m_k}(1_n)f_{1/l(k)}(a)-{{\phi_{m_k}}}(1_n)\|<1/4k.
\eneq
It follows that {{(see Proposition 2.2 of \cite{Ror92UHF2})}}, for $k\ge 2,$
\beq
 f_{1/k}(\phi_{{m_k}}(1_n))\lesssim f_{{1/l(k)}}(a) \andeqn n[f_{1/k}(b_{k,1})]=[ f_{1/k}(\phi_{{m_k}}(1_n))].
\eneq
{{Then \eqref{LTM-1e-2} holds.}}
\end{proof}

\begin{thm}\label{PTMtosurj}
Let $A$ be a $\sigma$-unital simple \CA\, 
with $QT(A)\not=\{0\}.$
If $A$ has strict comparison and property (TM), 
then $\Gamma$ is surjective (see \ref{DGamma}).
\end{thm}

\begin{proof}
It suffices to prove the proposition for the case that $A={\rm Ped}(A).$

We claim that, for any $a\in {\rm Ped}(A\otimes {\cal K})_+^{\bf 1}\setminus \{0\}$   and  any integer $n\in \N,$ 
there exists $b\in (A\otimes {\cal K})_+$
such that
\beq
n\widehat{[b]}\le \widehat{[a]}\le (n+1)\widehat{[b]}.
\eneq

Case (1): $0$ is an isolated point of ${\rm sp}(a).$ 
In this case we may assume that $a=p$ for some projection $p.$
Choose 
\beq
\eta:=({1\over{(n+1)^2}})\inf\{\tau(p): \tau\in \Qw\}>0.
\eneq
Note that $f_{1/k}(p)=p$ for all $k>1.$
Applying Lemma \ref{LTM-1}, we obtain $b\in \Her(a)_+$ such that
\beq
nd_\tau(b)\le d_\tau(p)<nd_\tau(b)+\eta\rforal \tau\in \Qw.
\eneq
Then we compute that
\beq
n\widehat{[b]}\le \widehat{[a]}\le (n+1)\widehat{[b]}.
\eneq

Case (2): $0$ is not an isolated point of ${\rm sp}(a).$ 
We will use Lemma \ref{LTM-1} for an induction argument.
Put
\beq
\sigma_0:=\inf\{\tau(a): \tau\in \Qw\}>0.
\eneq
Fix $n\in N.$
Since $0$ is a not an isolated point of ${\rm sp}(a),$
for each integer $k,$ there is a smallest integer ${{J}}(k)>k$
such that
\beq
f_{1/J(k)}(a)-f_{1/k}(a)\not=0.
\eneq
Define, for each $k,$ 
\beq\label{PTMtosurj-e1}
\sigma_k&:=&\inf\{d_\tau(f_{1/J(k)}(a))-d_\tau(f_{1/k}(a)): \tau\in \Qw\}>0\andeqn\\
\eta_k &:=&\min\{\sigma_j: {{0}}\le j\le k+1\}/2^{k+1}(n+1).
\eneq

Applying Lemma \ref{LTM-1}, 
there are mutually orthogonal  elements $b_{1,1}, b_{1,2},...,b_{1,n}\in \Her(a)_+^{\bf 1}$ such that
$[b_{1,i}]=[b_{1,1}],$ $i=1,2,...,n,$ and, for some  $m_1\in \N,$ 
\beq
{{d_\tau}}
(f_{1/2}(a^{1/m_1})) {{<}} nd_\tau(f_{1/2}(b_{1,1}))+ \eta_1\tforal \tau\in \Qw,
\eneq
and, an integer $l(1)\in \N$ such that
\beq\label{41}
nd_\tau(f_{1/2}(b_{1,1}))\le 
{{d_\tau}}(f_{1/l(1)}(a))\tforal \tau\in \Qw.
\eneq
Put $c_1:=f_{1/2}(b_{1,1}).$
Then, for all $\tau\in \Qw,$ 
\beq
\widehat{[f_{1/2}(a^{1/m_1})]}(\tau)\le n\widehat{[c_1]}(\tau)+ \eta_1\andeqn n\widehat{[c_1]}(\tau)\le \widehat{[f_{1/l(1)}(a)]}(\tau).
\eneq

Choose $k_2>l(1)$ such that  $k_2\ge J(l(1)).$ Applying Lemma \ref{LTM-1},
we obtain  $m_2\ge m_1$ and mutually orthogonal $b_{2,1},b_{2,2},...,b_{2,n}\in \Her(a)_+^{\bf 1}$ 
with $b_{2,j}\sim b_{2,1}$ ($1\le j\le n$) and $l(k_2)\in \N$
such that
\beq\label{1223-8-1-1}
&&d_\tau(f_{1/k_2}(a^{1/m_2}){{)}}<nd_\tau(f_{1/k_2}(b_{2,1}))+\eta_{l(1)} \andeqn\\\label{1223-8-1}
&&nd_\tau(b_{2,1})\le d_\tau(f_{1/l(2)}(a))\rforal \tau\in \Qw.
\eneq

Put $c_2=f_{1/k_2}(b_{2,1}).$  Then, for all $\tau\in \Qw,$ 
\beq
[f_{1/k_2}(a^{1/m_2}){{\widehat{]}}}(\tau)\le n\widehat{[c_2]}(\tau)+ \eta_{l(1)}\andeqn n\widehat{[c_2]}(\tau)\le 
[f_{1/l(2)}(a){{\widehat{]}}}(\tau).
\eneq
We compute that,  for all $\tau\in \Qw,$  by  \eqref{1223-8-1-1}  and \eqref{41}), 
{{\eqref{PTMtosurj-e1}
(recall that $k_2>J(l(1))$),}}
\beq
nd_\tau(f_{1/k_2}(b_{2,1}))&>&d_\tau(f_{1/k_2}(a^{1/m_2}))-\eta_{l(1)}\\
&=&d_\tau(f_{1/l(1)}(a)) +(d_\tau(f_{1/k_2}(a^{1/m_2}) )-d_\tau(f_{1/l(1)}(a)))-\eta_{l(1)}\\
&\ge & nd_\tau(f_{1/2}(b_{1,1}))+(d_\tau(f_{1/k_2}(a^{{1/m_2}}) )-d_\tau(f_{1/l(1)}(a)))-\eta_{l(1)}
\\
&>&nd_\tau(f_{1/2}(b_{1,1}))+\sigma_{l(1)}-\eta_{l(1)}>nd_\tau(f_{1/2}(b_{1,1})).
\eneq
Since $A$ has strict comparison,  we obtain
\beq
 [c_1]\le [c_2] .
\eneq

Suppose that we have constructed 
integers $k_i, m_i, l(i)\in \N$ and $b_{i,1}\in \Her(a)_+^{\bf 1},$  $1\le i\le I$ 
such that, for all $2\le i\le I$ and  $\tau\in \Qw$ {{(with $l(0)=1$ and $\eta_{l(0)}=\eta_1$),}}
\beq
&&k_i>J(l(i-1))>l(i-1),\\
&&
[f_{1/k_i}(a){{\widehat{]}}}(\tau)<n
[f_{1/k_i}(b_{i,1}){{\widehat{]}}}(\tau)+\eta_{l(i-1)}\andeqn\\
&&n
[f_{1/k_i}(b_{i,1}){{\widehat{]}}}(\tau)\le 
[f_{1/l(i)}(a){\widehat{]}}(\tau),
\eneq
and verified that $
[f_{1/k_i}(b_{i,1}) ] 
\le 
[f_{1/k_{i+1}}(b_{i+1,1})],
$ $1\le i\le I-1.$

Define $c_i=f_{1/k_i}(b_{i,1}),$ $1\le i\le I.$ 
By applying 
Lemma \ref{LTM-1}, there is $k_{I+1}>J(l(i))>l(i),$ $l(I+1)\ge k_{I+1},$ 
 $m_{I+1}\ge m_I,$   and $b_{I+1,1}\in \Her(a)_+^{\bf 1}$  such that,
 for all $\tau\in\Qw,$
 \beq
 &&
 [f_{1/k_{I+1}}(a^{1/m(I+1)}){\widehat{]}}(\tau)<n
 [f_{1/k_{I+1}}(b_{I+1,1}){\widehat{]}}(\tau)+\eta_{l(I)}\andeqn\\
&&n
[f_{1/k_{I+1}}(b_{I+1,1}{\widehat{]}}(\tau)\le 
[ f_{1/l(I+1)}(a){\widehat{]}}(\tau).
 \eneq
Then, for all $\tau\in\Qw,$ 
\beq\nonumber
nd_\tau(f_{1/k_{I+1}}(b_{I+1,1}))&>&d_\tau(f_{1/k_{I+1}}(a^{1/m_{I+1}}))-\eta_{l(I)}\\\nonumber
&=&d_\tau(f_{1/l(I)}(a)) 
+(d_\tau(f_{1/k_{I+1}}(a^{1/m_{I+1}}))-d_\tau(f_{1/l(I)}(a)))-\eta_{l(I)}\\\nonumber
&\ge& nd_\tau(f_{1/k_I}(b_{I,1}))
+(d_\tau(f_{1/k_{I+1}}(a^{1/m_{I+1}}))-d_\tau(f_{1/l(I)}(a)))-\eta_{l(I)}
\\
&>&nd_\tau(f_{1/k_I}(b_{I,1}))+\sigma_{l(I)}-\eta_{l(I)}>nd_\tau(f_{1/k_I}(b_{I,1})).
\eneq
Put $c_{I+1}=f_{1/k_{I+1}}(b_{I+1,1}).$ Then, by  
the strict comparison,  estimates above imply that
\beq
[c_I]\le [c_{I+1}].
\eneq
Thus, by induction, we obtain an increasing sequence $c_i\in \Her(a)_+^{\bf 1}$ 
such that, for all $\tau\in \Qw,$ 
\beq\label{TTM-18}
&&
[f_{1/k_i}(a){\widehat{]}}(\tau)<n\widehat{[c_i]}(\tau)+\eta_{l(i-1)}\andeqn\\\label{TTM-19}
&&n\widehat{[c_i]}(\tau)\le 
[ f_{1/l(i)}(a){\widehat{]}}(\tau), \,\,\, i
\geq 2.
\eneq
Let $c\in (A\otimes {\cal K})_+$ be such that $[c]$ is the supremum of $\{[c_i]\}.$ 
Then, by \eqref{TTM-18}, for all $i,$ 
\beq
d_\tau(f_{1/k_i}(a)) {{<}} n\widehat{[c]}(\tau)+\eta_{l{{(i-1)}}}\le n\widehat{[c]}(\tau) +\sigma_0/2^{i+1}(n+1)
 \rforal \tau\in \Qw.
\eneq
 Thus ({{recall that}} $A$ is simple),  for all sufficiently large $i,$
\beq
d_\tau(f_{1/k_i}(a)) {{<}}  (n+1)\widehat{[c]}(\tau)
 \rforal \tau\in \Qw.
\eneq 
It follows that {{(let $i\to\infty$)}}
\beq\label{TTM-9}
d_\tau(a)\le (n+1)\widehat{[c]}(\tau)\rforal \tau\in \Qw.
\eneq
On the other hand, by \eqref{TTM-19}
\beq
n\widehat{[c_i]}\le \widehat{[a]}\rforal i\in \N.
\eneq

For any $\ep>0,$ choose a non-zero element $e\in A_+$ such 
that $d_\tau(e)<\ep/2$ for all $\tau\in \Qw.$
Then, by  the strict comparison, 
\beq
n[c_i]\le [a]+[e]\rforal i\in \N.
\eneq
If follows that
$n[c]\le [a]+[e].$
Hence 
\beq
n\widehat{[c]}<d_\tau(a)+\ep\rforal \tau\in 
\Qw.
\eneq
Let $\ep\to 0.$ We also obtain
\beq\label{TTM-10}
n\widehat{[c]}\le \widehat{[a]}\rforal \tau\in \Qw.
\eneq
Combining \eqref{TTM-10} and \eqref{TTM-9}, the claim also holds for Case (2).

We now show that the proved claim implies that $\Cu(A)$ has the property 
D stated in the proof of Proposition 6.21 of \cite{Rl}.
Let $x'\ll x,$ where $x=[a]$ for some $a\in ( A\otimes {\cal K} )_+ .$ 
Since 
$$
{{x=\sup\{[(a-\|a\|/k)_+]: k\in \N\},}}
$$  {{then}}
$$
{{ x' \le [(a-\|a\|/k)_+]\,\,\,{\rm for}\,\,\,k\in  \N.}}
$$

Note $(a-\|a\|/k)_+\in {\rm Ped}(A\otimes {\cal K})_+.$ 
Then the claim implies  that there is $y\in \Cu(A)$ such 
that
\beq
\widehat{x'}\le [(a-\|a\|/k)_+{\widehat{]}}\le (n+1) \widehat{y}\andeqn   n\widehat{y}\le [(a-\|a\|/k)_+{\widehat{]}}\le [a]=x.
\eneq
Therefore, as  observed  by L. Robert (see Property D in the proof  of Proposition 6.21 
of \cite{Rl}{{),}}
 following Corollary 5.8  of \cite{ERS}, $\Gamma$ is surjective.
\end{proof}

\begin{lem}\label{Linvertible}
Let $A$ be a \CA\, and $a, b\in A_+^{\bf 1}.$ 
Suppose that there is $x\in A$ such that
\beq
x^*x=a\andeqn xx^*\in {\rm Her}(b).
\eneq
Then,  for any $\ep>0,$ there exists a unitary $U\in M_2
({{\wtd A}})$ such that
\beq
U^*\diag(f_\ep(a), 0)U\in {\rm Her}{{(\diag(b,0)).}}
\eneq
\end{lem}

\begin{proof}
First we claim that, for any $y\in A,$ ${\rm dist}(\diag(y, 0), GL(M_2({{\wtd A}})))=0.$
To see this, let $\ep>0.$ Choose $V=\begin{pmatrix} 0& 1\\
                          1 & 0\end{pmatrix}$ which is a unitary in $M_2(\C).$
                          Then 

$$
Y:=\begin{pmatrix} 0& 1\\
                          1 & 0\end{pmatrix} \begin{pmatrix} y & 0\\
                                                                                 0 & 0\end{pmatrix}=\begin{pmatrix} 0 & 0\\
                                                                                                                                          y & 0\end{pmatrix}
    $$                                                                                                                                      
which is a nilpotent. Therefore $Y+\ep\cdot 1_2\in 
GL(M_2({{\wtd A}})).$ 
Then 
$$
{{\diag(y,0)\approx_\ep V^*(Y+\ep)\in 
GL(M_2({{\wtd A}})).}}
$$ 
This proves the claim.

To prove the lemma, we will combine the claim with an argument of R\o rdam.
By  Proposition 2.4 of \cite{Ror92UHF2}, for any $\ep>0,$  there exist $\dt>0$ and $r\in A$ such that
\beq
f_{\ep/2}(a)=rf_\dt(b)r^*.
\eneq
Put $z=rf_\dt(b)^{1/2}$ and $Z=\diag(z, 0).$
By the claim and  Theorem 5 of \cite{Ppolar},  there is a unitary  $U\in M_2({{\wtd A}})$
such that
\beq
U^*f_{1/2}(ZZ^*)U=U^*f_{1/2}(\diag(zz^*, 0))U=\diag(f_{1/2}(z^*z),0)=f_{1/2}(Z^*Z).
\eneq
Note that $Z^*Z\in {\rm Her}({\bar b}),$ where ${\bar b}=\diag(b,0).$ 
Moreover (with ${\bar a}=\diag(a, 0)$)
\beq\nonumber
U^*f_\ep({\bar a})U\le U^*f_{1/2}(f_{\ep/2}({\bar a}){{)U}}=U^*f_{1/2}(ZZ^*)U=f_{1/2}(Z^*Z).
\eneq
\end{proof}

\begin{lem}\label{LL-N1}
Let $A$ be an  algebraically simple \CA\, 
which has strict comparison.
Suppose that  {{$QT(A)\not=\emptyset$ and}} the canonical map $\Gamma$ is surjective.

Suppose that $a, a'\in {\rm Ped}(A)^{\bf 1}\setminus \{0\}$
with $a\in \Her(a').$  Then,   there exists $1/2>\ep_0$ satisfying the following:
for any $0<\eta<\ep<\ep_0,$ any $\sigma>0,$  there  exist  $c\in {\rm Her}(f_{\eta}(a))_+^{\bf 1}$ 
with $\|c\|\le \|a\|$ and 
unitary $U\in M_2({\rm Her}(a')^\sim)$ 
such that  (with $b=U^*\diag(c, 0)U$)
\beq
\hspace{-1.3in}&&(1)\hspace{0.2in}  \diag(f_{\ep}(a), 0)\le b,\\
\hspace{-1.3in}&&(2) \hspace{0.2in}d_\tau(f_{\ep}(a))\le d_\tau(b)\le d_\tau(f_{\eta}(a))\tforal \tau\in \Qw, 
\eneq
and, for some $1>\dt>0,$ 
\beq
\hspace{-1.4in}(3)&&\,\,\, |d_\tau(b)-\tau(f_\dt(b))|<\sigma \tforal \tau\in  \Qw. 
\eneq
Moreover, 
\beq
\hspace{-2.4in} (4) &&\hspace{0.2in}U^*\diag(g_{\eta/2}(a), a')U\in B,
\eneq
where $B:={\rm Her}(b)^\perp\cap {\rm Her}(\diag(a,a'))$ and 
$g_\eta(t)\in C_0((0, 1])$   is defined  as in \ref{Nfg}.

Consequently, if $e$ 
is a strictly positive element  in $({\rm Her}(b)^\perp\cap {\rm Her}(\diag(a,a'))),$ 
then
\beq 
d_\tau(e)> d_\tau(a') +d_\tau(g_\eta(a)) \tforal \tau\in \Qw.
\eneq
\end{lem}

\begin{proof}
\Wlog, we may assume that $\|a\|\le 1.$ 
Let us  first assume that $[0,\ep_0)\subset {\rm sp}(a)$ for some $\ep_0>0.$

 Fix $0<\ep<\ep_0.$ 
 Note that, \wilog, we may assume that 
\beq\label{1223-8-0}
d_\tau(f_{\ep}(a))<\tau(f_{\dt_1}(a))<
d_\tau(f_{\eta_1}(a))<\tau(f_{\dt_2}(a))<d_\tau(f_{\eta}(a))\rforal \tau\in \Qw,
\eneq
where $\ep/2>\dt_1, \dt_1/2>\eta_1, \eta_1/2>\dt_2, \dt_2/2>\eta.$
Put $h_i(\tau)=\tau(f_{\dt_i}(a))\rforal \tau\in \Qw,$ $i=1,2.$
Then $h_i\in \Aff_+(\Qw),$ $i=1,2.$

Since $\Gamma$ is surjective, there is $c_0\in  (A\otimes  {\cal K})_+$ such 
that $d_\tau(c_0)=h_2(\tau)$ for all $\tau\in {{\Qw.}}$
%
Choose $\dt_0>0$ such that (as $h_2$ is continuous)
\beq\label{LL-1-e4}
d_\tau(c_0)-\tau(f_{\dt_0}(c_0))<\sigma/2\tforal \tau\in \Qw.
\eneq
Since $h_1<h_2$ are continuous, we may also assume,  {{by choosing smaller $\dt_0$,}} that
\beq\label{LL1-e5}
d_\tau(f_{\dt_0}(c_0))>d_\tau(f_{{\eta_1}}(a)) {{>d_\tau(f_\ep(a))}}\rforal \tau\in  \Qw.
\eneq
Since $A$ has strict comparison,  by \eqref{1223-8-0},
there is $x\in A\otimes {\cal K}$ such 
that
\beq
x^*x=f_{\dt_0/4}(c_0)\andeqn xx^*\in {\rm Her}(f_{\eta}(a)).
\eneq
Choose $c=xx^*.$  Then $0\le c\le 1$ and $d_\tau(c)=d_\tau(f_{\dt_0/4}(c_0))$ for all $\Qw.$
Let $C={\rm Her}(f_{\eta}(a)).$  
By \eqref{LL1-e5}, the strict comparison  and Lemma \ref{Linvertible},  we obtain  
a unitary $U\in M_2(\wtd C)$ such that
\beq
U\diag(f_{\eta_1}(a),0)U^*\in {\rm Her}({\bar c}),
\eneq
where ${\bar c}=\diag(c,0).$
Let $b=U^*{\bar c}U.$    Then 
\beq
{{\diag(f_\ep(a),0)\le}} \diag(f_{\eta_1}(a), 0)\le b.
\eneq
 (So (1) holds). 
{{Moreover,}}
$d_\tau(b)=d_\tau(c)$ for all $\tau\in \overline{T(A)}^w.$
Consequently
\beq
d_\tau(f_{\eta_1}(a))\le d_\tau(b)\le d_\tau(f_{\eta}(a))\rforal \tau\in \Qw
\eneq
 (so (2) holds).  Moreover, 
there is $1>\dt>0$ such that
\beq
d_\tau(f_\dt(b))\ge \tau(f_{\dt_0}(c))\rforal \tau\in \Qw.
\eneq
It then follows from \eqref{LL-1-e4} that 
\beq
|d_\tau(b)-\tau(f_\dt(b))|=|d_\tau(f_{\dt_0/4}(c_0))-\tau(f_\dt(b))|<\sigma\rforal \tau\in \Qw
\eneq
 (so (3) holds).  To show the `` Moreover" part, 
put 
$$
{{B={\rm Her}(b)^\perp\cap {\rm Her}(\diag(a,a'))\andeqn 
e'=U^*\diag(g_{\eta/2}(a), a')U.}}
$$   
Since 
$g_{\eta/2}(a)\perp f_\eta(a),$ we have 
$$
{{\diag(g_{\eta/2}(a), a')\perp \diag(f_\eta(a),0)\andeqn 
e'\perp b.}}
$$
It follows that  
$U^*\diag(g_{\eta/2}(a), a')U\in B.$
If $e$ is a strictly positive element in ${\rm Her}(b)^{\perp}\cap {\rm Her}(\diag(a, a')),$ then
\beq
d_\tau(e)\ge d_\tau(e')=d_\tau(a')+d_\tau(g_{\eta/2}(a))\rforal \tau\in \Qw.
\eneq
This proves the case that $[0, \ep_0]\subset {\rm sp}(a).$

If there exists $r_n\in (0,1]$ with 
$$
{{r_n>r_{n+1}\andeqn \lim_{n\to\infty}r_n=0\,\,\, {\rm  
such\,\, that}\,\,\, r_n\not\in {\rm sp}(a),}}
$$ 
then $b_n=f_{2r_n}(a)$ has the property that $\omega(b_n)=0.$ 
  Then the lemma {{follows}}  by choosing $U=\diag(1,1)$ and 
$b=b_n$  for some sufficiently large $n.$
\end{proof}

\begin{lem}\label{Lorthogso}
Let $A$ be a $\sigma$-unital  algebraically simple \CA\, 
with $QT(A){{\not=}}\emptyset.$
Suppose that $A$ has strict comparison and $\Gamma$ is surjective.
Suppose that $a=\diag(0, a_1, a_2,...,a_n)$  in $M_{n+1}(A){{^{\bf 1}_+}}$ for some integer $n\ge 1.$ 
Then, for any $1/2>\ep>0$ and $1/2>\sigma>0,$ there exists $d\in  M_{n+1}(A){{^{\bf 1}_+}}$ 
such that
\beq
f_{\ep}(a)\le d {{\le 1}} \tand \omega(d)<\sigma.
\eneq
\end{lem}

\begin{proof}

For $n=1,$ this follows immediately from Lemma  
\ref{LL-N1}.

Assume that the lemma holds for $n\ge 1.$

Let $0\le e_A\le 1$ be a strictly positive element of $A.$
 Fix $1/2>\ep>0.$ 
 Choose $\eta=\ep/4(n+2)$ and $\sigma_0:=\sigma/2(n+2).$
We will apply Lemma  \ref{LL-N1} with $a_j\in \Her(e_A)$ ($1\le j\le n+1$),
 $\eta$ as above, and $\sigma_0$ (in place of $\sigma$).

By Lemma \ref{LL-N1}, 
there is $c_1\in  {\rm Her}(f_\eta(a_1))_+^{\bf 1},$ a unitary $U_1\in M_2({{\wtd A}})$  and ${{b_1'}}=U_1^*\diag(0,c_1)U_1$
such that
\beq
&&\diag(0,f_\ep(a_1))\le b_1',\\
&&d_\tau(f_{\ep}(a_1))\le d_\tau(b_1')\le d_\tau(f_{\eta}(a_1))\tforal \tau\in \Qw,\\
&&\omega(b_1)<\sigma_0\andeqn 
U_1^*\diag(e_A, g_{\eta/2}(a_1))U_1\in B_1,
\eneq
where $B_1:=({\rm Her}({{b_1'}})^\perp\cap {\rm Her}(\diag(a_1,e_A)))_+.$ 
Put 
\beq
&&V_1=\diag(U_1, \overbrace{1_{{{\wtd A}}},...,
1_{{{\wtd A}}}}^{n-1}),\,\, \af_2=V_1^*\diag(0, 0, a_2,...,a_{n+1})V_1\andeqn\\
&&{{b_1}}=V_1^*(0, c_1,0,...,0)V_1.
\eneq
Define $C_1={\rm Her}(V_1^*(e_A, 0, e_A,...,e_A)V_1).$
Then 
\beq
b_1\in C_1^{\perp}\andeqn C_1\cong M_{n+1}(A).
\eneq
In $C_1,$ we may write $\af_2=\diag(0, a_2,a_3,...,a_{n+1})$
(the number of possible nonzero elements is now reduced to $n$).

By the induction assumption,  there is $b_2\in C_1$ with $0\le b_2\le 1$ such that
\beq
f_\ep(\af_2)\le b_2\andeqn \omega(b_2)<\sigma_0.
\eneq
Define $d:=b_1+b_2.$  Note that $b_1\perp b_2.$
Then
\beq
f_\ep(a)=\diag(0, f_\ep(a_1), f_\ep(a_2), ...,f_\ep(a_{n+1}))\le b_1+b_2.
\eneq
and {{(by  (2) of  \ref{Pbfos}),}}
\beq
\omega(d)\le \omega(b_1)+\omega(b_2)
<2\sigma_0<\sigma.
\eneq

This completes the induction. The lemma follows.

\end{proof}

\begin{thm}\label{TTM}
Let $A$ be a $\sigma$-unital   simple \CA\, with 
$ \wtd{QT} (A)\setminus \{0\}\not=\emptyset.$ 
Suppose that $A$ has  strict comparison
 and property (TM). Then 
$A$ has $T$-tracial approximate oscillation zero. 

\end{thm}

\begin{proof}
{{Choose $e\in {\rm Ped}(A)_+\setminus \{0\}$ and $A_1=\Her(e).$ 
Then ${\rm Ped}(A_1)=A_1.$ 
To prove the theorem, \wilog, we may assume that 
$A={\rm Ped}(A).$}}
By Theorem \ref{PTMtosurj}, $\Gamma$ is surjective.


Fix $a\in {\rm Ped}(A\otimes {\cal K} )_+^{\bf 1}.$ 
We claim  that ${\rm Her}(a)$ has a T-tracial approximate identity $\{d_n\}$
such that $\lim_{n\to\infty}\omega(d_n)=0.$

Put $B=\Her(a).$ Then ${\rm Ped}(B)=B$ and $B\otimes {\cal K}\cong A\otimes {\cal K}.$
Therefore, to simplify notation, \wilog, we may also assume that $a\in A.$

Let $\ep>0$ and let $n\in \N$ such that $1/n<(\ep/8)^2.$ 
Since $A$ has property (TM),  there is a c.p.c.~order zero map $\phi: M_{n+1}\to {\rm Her}(a)$ such that
\beq\label{TTM-n-4}
\|a-\phi(1_{n+1})a\|_{_{2, \Qw}}< (\ep/8)^3. 
\eneq
By Proposition \ref{p1028-1}, let ${{C}}={\rm Her}(\phi(1_{n+1}))\cong M_{n+1}(({\rm Her}(\phi(e_{1,1}))).$
Write 
\beq
\phi(1_{n+1})=\diag(c,c,...,c)\in M_{n+1}({\rm Her}(\phi(e_{1,1}))).
\eneq
Choose  {{$0<\eta<(\ep/16)^2.$}}
Put $c_n=\diag(0, c,c,...,c).$
It follows from Lemma \ref{Lorthogso}  that  there exists $d\in  {{C}}_+^{\bf 1}$
such that
\beq
f_\eta(c_n)\le d \le 1\andeqn \omega(d)<1/2^n. 
\eneq
Hence  
\beq
0\le ({c_n-2\eta})_+(1-d)({c_n-2\eta})_+\le ({{c_n}}-2\eta)_+(1-f_\eta(c_n))({{c_n}}-2\eta)_+=0.
\eneq
{{Hence $d(c_n-2\eta)_+=(c_n-2\eta)_+=(c_n-2\eta)_+d.$
It follows that (we now working in a commutative \SCA)
\beq\label{TTM-n5-1}
(1-d)^2\le (1-({{c_n-2\eta}})_+)^2.
\eneq}}
Note that
\beq\label{TTM-n5}
\|\phi(1_{n+1})-c_n\|_{_{2, \Qw}}<{1\over{n+1}}.
\eneq
Then {{(see also \eqref{D2norm}), by \eqref{TTM-n5-1}, \eqref{TTM-n5} and by \eqref{TTM-n-4},}} 
\beq\nonumber
\|a-da\|^{2/3}_{_{2, \Qw}}&=&(\sup\,_{_{ \tau\in \Qw}}\{\tau(a(1-d)^2a)\})^{1/3}\\\nonumber
&\le & \sup\,_{_{ \tau\in \Qw}}\{\tau(a(1-(c_n-2\eta)_+)^2a)\})^{1/3}
=\|a-(c_n-2\eta)_+a\|^{2/3}_{_{2, \Qw}}\\\nonumber
&\le& \|a-c_na\|^{2/3}_{_{2, \Qw}} +\|(c_n-({c_n-2\eta})_+)a\|^{2/3}_{_{2, \Qw}}\\
&<& \|a-\phi(1_{n+1})a\|^{2/3}_{_{2,\Qw}}+({1\over{n+1}})^{2/3}+(2\eta)^{2/3}\\
&<& (\ep/8)^2+(\ep/8)^2+(\ep/8)^2=3(\ep/8)^2.
\eneq
{{Since $\omega(d)<1/2^n,$}} 
this proves the claim.
The theorem then follows from the claim and Lemma
\ref{PapproxeT}.
\end{proof}

\section{Stable rank one}

Let $A$ be a \CA\, 
and $n\in\mathbb{N}.$
{{Recall that}} we view {{$M_n(A)$}} 
as a $C^*$-subalgebra of $M_{n+1}(A)$ 
in the canonical way, i.e., 
$M_n(A)$ is the upper left block of $M_{n+1}(A).$

Recall an element $a=(a_{i,j})_{n\times n}$  in $M_n(A)$ is called upper
(resp. lower)  
{{triangular}}, if
$a_{i,j}=0$ whenever $i<j$ {{({{resp.}} $i>j$),
and, $a$ is called}}
strictly upper ({{resp.}} lower)  
triangular, 
if $a_{i,j}=0$ whenever $i\le j$ ({{resp.}} $i\ge j$).

The following proposition is a generalization of {{an}}
elementary fact
in linear algebra. 

\begin{prop}
\label{triangular-prod}
Let $A$ be a \CA\, {{such that $A\subset \overline{GL(
{{\wtd A}})}.$}}
{{Then,}} for any $n\in\N,$ any $a\in M_n(A),$
any $\ep>0,$
there is an upper triangular matrix $x\in M_n(A)$
and  a lower triangular matrix $y\in M_n(A)$
such that $a\approx_\ep xy.$
\end{prop}
\begin{proof}
We prove this by induction on $n.$
For $n=1,$ {{let 
$a\in M_n(A)=A$ and 
 $\ep>0.$}}
By the existence of approximate identity, 
there is $e\in A_+$
such that $a\approx_\ep ae.$ Note  {{that}} $a$ and $e$
are triangular matrices in $M_n(A).$
Thus the proposition holds for $n=1.$

Assume the proposition holds for $n\geq 1.$
Let $a=\sum_{i,j=1}^{n+1}a_{i,j}\otimes e_{i,j}
\in 
A\otimes M_{n+1},$
where $a_{i,j}\in A$ and $\{e_{i,j}\}$ 
is the matrix units of $M_{n+1},$
$i,j=1,...,n+1.$ 
Let $\ep>0.$
Since {{$A\subset \overline{GL({{\wtd A}})},$}}
there is $\tilde a\in GL(\wtd A)$ 
such that 
$${{a_{n+1,n+1}\approx_{\ep/2}\tilde a.}}$$ 
{{In what follows in this proof, $1$ is the identity of 
${{\wtd A}}$ and $1_{n+1}$ is the identity
of $M_{n+1}({{\wtd A}}).$}}

Let  {{$b^{(0)}:=\sum_{i=1}^{n}a_{i,n+1}\tilde a^{-1}\otimes e_{i,n+1}$ and $c^{(0)}:=\sum_{j=1}^{n}\tilde a^{-1}a_{n+1,j}\otimes e_{n+1,j}.$  Then $b^{(0)}$ and $c^{(0)}$ are  nilpotents.}}
Put
\beq
a':=a+(\tilde a-a_{n+1,n+1})\otimes e_{n+1,n+1},\eneq
\beq
b:=1_{n+1}
-b^{(0)}
\quad\text{and}\quad 
c:=1_{n+1}-c^{(0)}.
\eneq
Let  
$s:=ba'c.$
Note that $a',b,c, s$ are in $M_{n+1}(\wtd A).$
Let $a'_{i,j}$ (resp. $b_{i,j},c_{i,j},  s_{i,j}$)
be the $(i,j)$-th entry of 
$a'$ (resp. $b,c,  s$), $1\leq i,j\leq n+1.$
Note that  

\beq
s_{i,j}=\sum_{m=1}^{n+1}\left(\sum_{k=1}^{n+1}
b_{i,k}a'_{k,m}c_{m,j}
\right)
\quad 
(1\leq i,j\leq n+1).
\label{e1112-1}
\eneq
{If $A=\wtd A,$ then $s_{i,j}\in A.$}
{{Otherwise denote}} by $\pi: M_{n+1}({{\wtd A}})\to M_{n+1}$ the quotient
map. Then 
$$
{{\pi(b)=\pi(1_{n+1})=\pi(c)\andeqn \pi({{a'}})=\pi(\td a\otimes e_{n+1, n+1}).}}
$$
It follows that $\pi(s)=\pi(\td a\otimes e_{n+1,n+1}).$
Thus 
\beq
s_{i,j}\in A
\quad
(1\leq i,j\leq n).
\label{e1114-3}
\eneq
%
Note that 
$b_{n+1,k}=0$ for $1\leq k \leq n$
and $c_{m,j}=0$ for $m\notin \{j,n+1\},$ {{if $1\le j\le n.$}}
{{By}}
\eqref{e1112-1}, we have, {{for $1\le j\le n,$}} 
\beq
s_{n+1,j}
&=&
b_{n+1,n+1}a'_{n+1,j}c_{j,j}
+
b_{n+1,n+1}a'_{n+1,n+1}c_{n+1,j}
\\
&=&
a_{n+1,j}
+\tilde a\cdot(-\tilde a^{-1}a_{n+1,j})
=0. 
\label{e1114-1}
\eneq

{{If}}
$1\leq i\leq n,$  {{then}}
$b_{i,k}=0$ for $k\notin \{i,n+1\}$
and 
$c_{m,n+1}=0$ for $1\leq m \leq n.$
{{By}} \eqref{e1112-1}, we {{compute}}
\beq
s_{i,n+1}
&=&
b_{i,i}a'_{i,n+1}c_{n+1,n+1}
+
b_{i,n+1}a'_{n+1,n+1}c_{n+1,n+1}
\\
&=&
a_{i,n+1}
+(-a_{i,n+1}\tilde a^{-1})\cdot\tilde a
=0. 
\label{e1114-2}
\eneq
We also have  
\beq
s_{n+1,n+1}=\sum_{m=1}^{n+1}\left(\sum_{k=1}^{n+1}
b_{n+1,k}a'_{k,m}c_{m,n+1}
\right)
=
b_{n+1,n+1}a'_{n+1,n+1}c_{n+1,n+1}
=\tilde a.
\label{e1114-5}
\eneq

Therefore 
\eqref{e1114-3}, \eqref{e1114-1}, \eqref{e1114-2},
and \eqref{e1114-5} show that 
$$
{{ba'c
=
d+
\tilde a\otimes e_{n+1,n+1},}}
$$
where $d\in M_n(A).$ 
Note that $b$ and $c$ are invertible in $M_{n+1}(\wtd A),$
{{as both $b^{(0)}$ and $c^{(0)}$ are nilpotents.}}
Let $\ep_1=
\frac{\ep}{4(1+\|b^{-1}\|
\cdot\|c^{-1}\|)}.$
By our assumption, 
there is {{an}} upper triangular matrix $x_1$
and a lower triangular matrix $y_1$ in $M_n(A)$
such that 
\beq
d\approx_{\ep_1}
x_1y_1.
\eneq
Let $e\in M_{n+1}(A)_+^{\bf 1}$ 
be a diagonal matrix 
such that 
$a\approx_{\ep/4}eae.$

Note that 
$b^{-1}=1_{n+1}+b^{(0)}$ 
 and 
$x{{:}}=eb^{-1}(x_1+\tilde a\otimes e_{n+1,n+1})$
{{are}} upper triangular matrix in $M_{n+1}(A).$
Similarly, 
$c^{-1}{{=1_{n+1}+c^{(0)}}}$  is a lower 
 triangular matrix in $M_{n+1}(\wtd A),$
 and 
$$
{{y:=(y_1+
{{1}}\otimes 1_{n+1,n+1})c^{-1}e}}
$$
is a lower triangular matrix in $M_{n+1}(A).$
Then 
\beq
a
\approx_{\ep/4}eae
&\approx_{\ep/2}&
ea'e
=
eb^{-1}ba'cc^{-1}e=
eb^{-1}
(d+
\tilde a\otimes e_{n+1,n+1})
c^{-1}e\\
&\approx_{\ep/4} &
eb^{-1}
(x_1y_1+
\tilde a\otimes e_{n+1,n+1})
c^{-1}e\\
&=&
eb^{-1}
(x_1+
\tilde a\otimes e_{n+1,n+1})
\cdot(y_1+1\otimes e_{n+1,n+1})
c^{-1}e=
xy.
\eneq
Thus the proposition holds for $n+1.$
By induction, the proposition holds. 
\end{proof}

\begin{prop}
\label{p09-1}
Let $A$ be a \CA\, {{such that $A\subset 
\overline{GL({{\wtd A}})}$}}
{{and let}} $n\in\mathbb{N}.$
Then for any  $a\in M_n(A)$ {{and}}
any $\ep>0,$ 
there is a strictly upper triangular matrix 
$x\in M_{n+1}(A)$ and a 
strictly lower triangular matrix $y\in M_{n+1}(A)$
such that $a\approx_{\ep}xy.$ 

In particular, 
any element in $M_n(A)$ 
can be approximated {{in norm}} by product of two nilpotent
elements in $M_{n+1}(A).$ 
\end{prop}
\begin{proof}
By Proposition \ref{triangular-prod},
there is an upper triangular matrix $x_1\in M_n(A)$
and a lower triangular matrix $y_1\in M_n(A)$ 
such that $a\approx_\ep x_1y_1.$
Let $v=\sum_{i=1}^n1_{\wtd A} \otimes e_{i,i+1}\in M_{n+1}(\wtd A).$
Then $x=x_1v\in M_{n+1}(A)$ is a strict upper triangular matrix
and $y=v^*y_1\in M_{n+1}(A)$ is a strict lower triangular matrix, 
and $xy=x_1vv^*y_1=x_1y_1\approx_\ep a$ {{(Recall that
we identify $M_n(A)$ with the 
{{upper left $n\times n$ corner}}
of $M_{n+1}(A)$).}}

The last part of the proposition follows 
from the fact that strictly triangular matrices 
are nilpotents. 
\end{proof}

\begin{lem}
\label{lem1028-1}
Let $A$ be   a $\sigma$-unital  algebraically simple  {{non-elementary}} \CA\  
with $QT(A)\not=\emptyset$  
 {{which has 
strict comparison.}}  
Suppose that $A$ also has the 
property (TM).
Let $a\in A.$ If there are $b_1,b_2\in A_+\setminus \{0\}$
such that 
{{$a^*a+aa^*, b_1, b_2$ are mutually orthogonal,}} 
{{then,}} for any $\ep>0,$
there are two nilpotents $x,y\in A$
such that $\|a-xy\|<\ep.$
\end{lem}

\begin{proof}
Let  {{$B=l^\infty(A)/I_{_{\Qw}}.$}}
{{Recall that $\Pi:l^\infty(A)\to B$ is  the quotient map and 
$\iota:A\to l^\infty(A)$ is  the canonical embedding. }}
{{Denote}}  ${\bar \iota}:=\Pi\circ\iota.$
{{Fix}} $a\in A.$ 
\Wlog, we may assume $\|a\|\leq 1.$
{{Put $a_0=a^*a+aa^*.$}}
Assume that there {{are}} $b_1,b_2\in A_+$
such that $0=b_1b_2=ab_1=b_1a=ab_2=b_2a{{}}
.$ 
Let $\ep>0.$  Since $A$ is simple {{and non-elementary,}}  one 
can choose $n\in \N$  such that 
$$
{{1/n<\inf\{d_\tau(b_i):\tau\in \Qw\}}},\,\,\,
i=1,2.
$$ 
Since $A$ has property (TM), {{by Theorem \ref{TTM}, $A$ has T-tracial approximate oscillation zero.
Then, by Theorem \ref{Tasr1},  $B$ has stable rank one.}}
{{Also,}}  by the last part of Remark \ref{R82},
for each $m\in\N,$
there is a c.p.c.~order zero map 
$\phi_m: M_n\to \Her({{a_0}})$ such that 
\beq
\label{e1027-1}
\|a-\phi_m(1_n)a\|_{_{2,\Qw}}<1/m. 
\eneq
Let $\phi:M_n\to 
l^\infty(A)$ 
be the map induced by $\{\phi_m\}_{m\in\N}$ {{and}}
${\bar \phi}:=\Pi\circ \phi.$
Then \eqref{e1027-1}
shows that 
\beq
\label{e1028-2}
{{{\bar \phi}(1_n){\bar \iota}(a)={\bar \iota}(a).}}
\eneq
{{Denote by $\{e_{i,j}: 1\le i,j\le n\}$ a system of matrix units for $M_n$ 
and $\{e_{i,j}: 1\le {{i}},j\le n+1\}$ an expanded system of matrix units for $M_{n+1}.$
In particular we view $M_n$ generated by $
\{e_{i,j}: 1\le i,j\le 1\}$ as a \SCA\, of $M_{n+1}$ generated by $\{e_{i,j}:1\le i,j\le n+1\}.$}}

Since $A$ has strict comparison, and 
for all $m\in\N,$
\beq
\sup\{d_\tau(\phi_m(e_{1,1})):\tau\in \Qw\}
\leq
1/n
<\inf \{d_\tau(b_2):\tau\in \Qw\},
\eneq
we have $\phi_m(e_{1,1})\lesssim b_2$ for all $m\in\N.$
By \cite[Proposition 2.4 (iv)]{Ror92UHF2}, 
there are 
${v_m}\in A$ 
\beq
v_m^*v_m=(\phi_m(e_{1,1})-1/m)_+
\quad
\text{ and }
\quad
v_mv_m^*\in \Her_A(b_2)
\quad 
(m\in\N)
\label{e1027-2}
\eneq
(see \eqref{313}) and \eqref{e1025-2}).
Note that 
$${{
\|v_m\|^2=\|v_m^*v_m\|
=\|(\phi_m(e_{1,1})-1/m)_+\|\leq 1.}}
$$
Let $v=\{v_1,v_2,...\}\in l^\infty(A).$ 
Since 
$\|(\phi_m(e_{1,1})-1/m)_+-\phi_m(e_{1,1})\|\leq 1/m$ 
($m\in\N$), 
we have 
\beq
{{\Pi}}(v^*v)={\bar \phi}(e_{1,1}).
\label{e1027-3}
\eneq
The {{facts}} that $\phi_m(1_n)\in \Her({{a_0}}),$
$\Her({{a_0}})\bot \Her(b_2),$ and  
$v_mv_m^*\in \Her(b_2)$  show that 
$\phi_m(1_n)\bot v_mv_m^*$ for all $m\in\N.$
Hence 
\beq
\Pi(vv^*){\bar \phi}(1_n)=0. 
\label{e1027-4}
\eneq
%
Let $h: C_0((0,1])\otimes M_n\to B$ the \hm\, defined by 
$h(\imath\otimes e_{i,j})={\bar \phi}(e_{i,j})$ ($1\le i,j\le n$). Extend
${\wtd h}: C_0((0,1])\otimes M_{n+1}\to B$ by
$\wtd h(\imath\otimes e_{i,j})=h({{\imath\otimes e_{i,j}}})$ and 
$\wtd h(\imath\otimes e_{1, n+1})=v^*.$ By \eqref{e1027-3}  and \eqref{e1027-4},
$\wtd h$ is indeed a  \hm.   Define  $\wtd \phi(e_{i,j})=\wtd h({{\imath\otimes e_{i,j}}})$ for 
$1\le i,j\le n+1.$

{{
As we  view $M_n$ as  \SCA\, of $M_{n+1},$}}
  $\wtd \phi$ is an extension of ${\bar \phi},$ {{i.e., $\wtd \phi|_{M_n}={\bar \phi}.$}} 
By Proposition \ref{p1028-1},
\beq
\Her_B(\wtd\phi(1_{n+1}))\cong
\Her_B(\wtd \phi(e_{1,1}))\otimes M_{n+1}
= \Her_B({\bar \phi}(e_{1,1}))\otimes M_{n+1},
\eneq
and 
\beq
\Her_B({\bar \phi}(1_n))\cong 
\Her_B({\bar \phi}(e_{1,1}))\otimes M_n.
\eneq
{{Moreover, as $\{e_{i,j}: 1\le i,j\le n\}\subset \{e_{i,j}:1\le i,j\le n+1\},$ 
we also write 
$$\Her_B({\bar\phi}(1_{n}))
\subset 
\Her_B(\wtd\phi(1_{n+1})).
$$}}
Since $B$ has stable rank one,   by \cite[Corollary 3.6]{BP95}, 
$\Her_B(\wtd\phi(1_{n+1}))$ also has stable rank one.
{{Note, by \eqref{e1028-2}, ${\bar \iota}(a)\in 
\Her_B({\bar \phi}(1_n)) \cong 
\Her_B({\bar \phi}(e_{1,1}))\otimes M_{n},$}}
Then by Proposition \ref{p09-1},
there are nilpotents 
$$
{{x_1,y_1\in \Her_B(\wtd\phi(1_{n+1}))
{{\cong \Her_B({\bar \phi}(e_{1,1}))\otimes M_{n+1}}}\,\,\,
{\rm{such\,\,\, that}}\,\,\,\|{\bar \iota}(a)-x_1y_1\|<\ep/8.}}
$$
{{Recall that 
$\Pi(vv^*+\phi(1_n))=\wtd \phi(1_{n+1}).$}} 
Thus 
\beq
\Her_B(\wtd \phi(1_{n+1}))=\Pi(\Her_{l^\infty(A)}(vv^*+\phi(1_n))).
\label{e1028-4}
\eneq
Also note that $(vv^*+\phi(1_n))\bot \iota(b_1).$
Thus 
\beq
\Her_{l^\infty(A)}(vv^*+\phi(1_n))
\subset \{\iota(b_1)\}^\bot.
\label{e1028-5}
\eneq
By {{\eqref{e1028-4} 
and}} the fact that 
nilpotents can be lifted 
(see \cite[Theorem 6.7]{OP89nilpotent}),
there are nilpotents $x_2,y_2\in \Her_{l^\infty(A)}(vv^*+\phi(1_n))$
such that 
\beq
\Pi(x_2)=x_1, 
\quad
\Pi(y_2)=y_1.
\eneq
{{It follows from \eqref{e1028-5} that we also have}}
\beq
x_2\bot\iota(b_1),
\quad
\text{ and }
\quad
y_2\bot\iota(b_1).
\label{e-1028-6}
\eneq
Since  $\|{\bar \iota}(a)-x_1y_1\|<\ep/8, $  
there is $\bar z\in I_{\Qw}$
such that 
$$
{{\iota(a)-x_2y_2\approx_{\ep/8} {\bar z}.}}
$$
Note that $\iota(a)-x_2y_2\in \{\iota(b_1)\}^\bot.$
Hence there is $d\in 
\{\iota(b_1)\}^\bot$ such that 
$$
\iota(a)-x_2y_2
\approx_{\ep/8}
d(\iota(a)-x_2y_2)d\approx_{\ep/8}d\bar zd.
$$
Let 
\beq
z:=d\bar z d\in \{\iota(b_1)\}^\bot\cap {{I_{_{\Qw}}}}.
\label{e1028-8}
\eneq
Then 
$$
{{\|\iota(a)-(x_2y_2+z)\|<\ep/4.}}
$$ 
{{Choose}}  $\dt>0$ such that $\|zf_\dt(|z|)-z\|<\ep/8.$ 
Then 
\beq
\label{e-1028-3}
\|\iota(a)-(x_2y_2+zf_\dt(|z|))\|<\ep/2.
\eneq
Write $z=\{z_1,z_2,...\}$ with   $z_i\perp b_1$ 
$(i\in\N).$
Since $z\in {{I_{_{\Qw}}}},$ there is  $i\in\N$ 
such that 
(note, the first inequality of the following always holds): 
\beq
\sup_{\tau\in {{\Qw}}}\{d_\tau(f_{\dt/2}(|z_i|))\}
\leq \frac{{4}}{\dt}\sup_{\tau\in {{\Qw}}}\tau(|z_i|)
<{{\inf_{\tau\in \Qw}\{d_\tau(b_1)\}.}}
\eneq
Since $A$ has strict comparison, 
$f_{\dt/2}(|z_i|)\lesssim b_1.$ 
By \cite[Proposition 2.4 (iv)]{Ror92UHF2},
there is $r\in A$ such that 
\beq
r^*r=f_{\dt}(|z_i|)
\quad
\text{and} 
\quad
rr^*\in \Her(b_1).
\label{e1028-7}
\eneq

{{Write $x_2=\{x_{2, j}\}_{j\in \N}$ and $y_2=\{y_{2, j}\}_{j\in \N}.$}}
By \eqref{e-1028-6}, $x_{2,i}\bot b_1$
and $y_{2,i}\bot b_1.$ 
Together with \eqref{e1028-7}, 
we have  
\beq
x_{2,i}r
=
r^*x_{2,i}
=
y_{2,i}r
=
r^*y_{2,i}=0.
\label{e1028-9}
\eneq
{{Thus}}
\beq
x_{2,i}y_{2,i}+z_if_\dt(|z_i|)
=x_{2,i}y_{2,i}+z_ir^*r
=(x_{2,i}+z_ir^*)(y_{2,i}+r).
\label{e1028-10}
\eneq

{{Since $x_2$ and $y_2$ are nilpotents, so are 
$x_{2,i},y_{2,i},$ }} 
By \eqref{e1028-8} and 
\eqref{e1028-7}, $r^*z_i=0.$ 
Hence 
$$
{{(z_ir^*)^2=z_i r^*z_ir^*=0.}}
$$ 
By \eqref{e1028-7}, $r^2=0.$
By \eqref{e1028-9}, 
$$
{{(z_ir^*)x_{2,i}=0\andeqn y_{2,i}r=0.}}
$$
Let $\af_1:=x_{2,i},$
$\af_2:=z_ir^*,$
$\bt_1:=y_{2,i},$
and 
$\bt_2:=r.$
Then the last paragraph shows that 
$\af_1,\af_2,\bt_1,\bt_2$ are {{all}} nilpotents, 
and $\af_2\af_1=\bt_1\bt_2=0.$
Then it is standard to conclude 
that $x:=\af_1+\af_2$ and 
${y:=}\bt_1+\bt_2$
are nilpotents 
(see the proof of Claim 1 in the proof of
\cite[Lemma 5.6]{FL21-2}).

By \eqref{e-1028-3}  and \eqref{e1028-10}, 
\beq
a\approx_{\ep/2} x_{2,i}y_{2,i}+z_if_\dt(|z_i|)
=(x_{2,i}+z_ir^*)(y_{2,i}+r)
{{=xy}}.
\eneq
The lemma follows.
\end{proof}

\begin{thm}\label{T92}
Let $A$ be a 
$\sigma$-unital 
simple \CA\ 
with  ${\wtd{QT}}(A)\setminus \{0\}\not=\emptyset.$
If $A$  has strict comparison and has $T$-tracial approximate oscillation zero, 
then $A$ has stable rank one. 
\end{thm}
\begin{proof}
We may assume {{that}} $A$ is  {{non-elementary.}}
There are two cases.

Case 1. $A$ has a nonzero projection $p$. 

{{Set}} $A_1:=pAp.$  {{Then}} 
$A_1$ is unital, simple, 
has nonempty  ${{QT(A_1)}},$ {{and}}
has strict comparison {{as well as T-tracial approximate oscillation zero.}}
Hence 
{{$l^\infty(A_1)/I_{_{\overline{QT(A_1)}}}$}} has stable rank one
(see Theorem \ref{Tasr1}). 
Let $a\in A_1$ be a non-invertible element  {{and let}} $\ep>0.$
Since $A_1$ is simple and finite, 
by  \cite[Proposition 3.2, Lemma 3.5]{Ror91UHF1}, 
there is a unitary $u\in U(A_1),$
an element $\bar a\in A_1,$
and a positive element $b\in (A_1)_+\backslash\{0\}$
such that $\|a-\bar a\|<\ep/4$  
and $b(u\bar a)=(u\bar a)b=0.$
Note that $\Her(b)$ is also infinite dimensional.
{{Hence}} there are two {{nonzero}} orthogonal positive elements 
$b_1,b_2\in \Her(b).$

By Theorem \ref{TVD-1}, $A$ has property (TM).
We  {{then}}   apply Lemma \ref{lem1028-1}
to obtain two nilpotent elements 
$x,y$ such that 
$$
{{u\bar a\approx_{\ep/4}xy.}}
$$
Let $\dt>0$ {{be}}  such that 
$$
{{xy\approx_{\ep/4}(x+\dt)(y+\dt).}}
$$
Note that $x+\dt$ and $y+\dt$ are invertible 
since $x,y$ are nilpotents. 
Then {{that}} 
$$
{{a\approx_{\ep/4} u^*u\bar a\approx_{\ep/2} 
u^* (x+\dt)(y+\dt)}}
$$
shows {{that}} $a$ can be approximated by 
an invertible element 
$u^* (x+\dt)(y+\dt)$ up to {{the}} tolerance $\ep.$
Hence $A_1$ has stable rank one. 
It follows that $A$ also has stable rank one.


Case 2. $A$ has no nonzero projections. 
 {{By Theorem \ref{TOstosurj-1}, the canonical map $\Gamma$ 
 is surjective.  Choose $e\in (A\otimes {\cal K})_+$ with $0\le e\le 1$ 
 such that $\wh{[e]}$ is continuous on $\wtd{QT}(A).$ 
 By  Theorem \ref{Pcontscale-1}, $C=\Her(e)$ has continuous scale. 
 By Brown's stable isomorphism theorem (\cite{Br}), 
 $C\otimes {\cal K}\cong A\otimes {\cal K}.$ Therefore it suffices to show that $C$ has 
 stable rank one (see \cite[Theorem 6.4]{Rff}). Hence, \wilog, we may assume that $A$ has continuous scale (and $QT(A)\not=\emptyset$).}}

Let $A_1$ be a {{$\sigma$-unital}} hereditary $C^*$-subalgebra of $A.$
Let $a\in A_1^{\bf 1}$ and let $\ep>0.$
{{Let $a_1:=a^*a+aa^*.$}}
Then there is $\dt>0$ such that 
$$
{{\|a-f_{\dt}({{a_1}})af_\dt({{a_1}})\|<\ep/2.}}
$$ 
Let $\bar a:=f_{\dt}({{a_1}})af_\dt({{a_1}}).$
Since $A_1$ has no nonzero projections, 
{{we may assume that $[0, \dt]\subset {\rm sp}({{a_1}}).$}}
Let $g\in C_0((0,1])_+$ with ${\rm supp}(g)\subset 
[\dt/4,\dt/2],$
then 
$$
{{b{{:}}=g({{a_1}})\neq 0\andeqn b\bar a=\bar a b=0.}}
$$  

{{Let us}} consider  
$A_2=\Her_{A_1}(f_{\dt/8}({{a_1}})).$ 
Note that  $A_2$ is  simple, {{$A_2={\rm Ped}(A_2)$ and $QT(A_2)\not=\emptyset$}}
and has strict comparison.
{{Moreover, by Proposition \ref{Pher},  $A_2$ has T-tracial approximate oscillation zero.}} 
Hence, {{by Theorem \ref{Tasr1},}} 
$l^\infty(A_2)/I_{_{{\overline{QT(A_2)}^w}}}$ has stable rank one. 
Note that  $\bar a, b\in A_2.$
Note {{also}} that, {{since  $A$ is non-elementary,}}
$\Her_{A_2}(b)$ is infinite dimensional. It follows that  there are $b_1,b_2\in  {{\Her_{A_2}(b)_+\setminus \{0\}}}$
such that $b_1\bot b_2.$
{{Since $\bar a^*\bar a+\bar a\bar a^*, b_1, b_2$ are mutually orthogonal,}}
applying Lemma \ref{lem1028-1},
we get two nilpotents $x,y\in A_2\subset A_1$
such that $\|\bar a-xy\|<\ep/2.$ 
It follows that $\|a-xy\|<\ep.$

Therefore, for any {{$\sigma$-unital}}  hereditary 
$C^*$-subalgebra $A_1\subset A,$
any $a\in A_1,$ 
any $\ep>0,$
there are nilpotents $x,y\in A_1$
such that $\|a-xy\|<\ep.$
Together with the facts that $A$ 
is projectionless and {{assumed to have}}  continuous scale, 
{{applying \cite[Theorem 6.4]{FLL21}, we}} 
conclude that $A$ has stable rank one. 
\end{proof}

{\bf The proof of Theorem \ref{Teqiv}}

\begin{proof}

For (1) $\Rightarrow$ (2),  applying Theorem \ref{TOstosurj-1}, we know that $\Gamma$ is surjective. 
{{Then (2) follows from}} 
Theorem \ref{T92}.

{{Both (2) $\Rightarrow$ (3) and  (2) $\Rightarrow$ (4) are obvious.}}
{{That (3) $\Rightarrow$ (2) follows from \cite{APRT}.}}

For (4) $\Rightarrow$ (1), we apply  Theorem \ref{Talsr1}.
That (1) $\Leftrightarrow$ (5) follows from Theorem  
\ref{TVD-1} and Theorem \ref{TTM}.  This ends the proof of Theorem \ref{Teqiv}.
\end{proof}

\vspace{0.1in}

{{Note that the separability condition is only used in the  implication of (3)  $\Rightarrow$ (2).}}

We learned  that the following is also obtained by {{S. Geffen and W. Winter.}}

\begin{cor}\label{TRR0str1}
Let $A$ be a  {{$\sigma$-unital}}  stably finite simple \CA\, of real rank zero which has strict comparison.
Then $\Gamma$ is surjective and $A$ has stable rank one. 
\end{cor}

\begin{proof}
Let $p\in A$ be a non-zero projection and $B=pAp.$ It suffices to show the statement holds for $B.$
Note that $B$ is unital and stably finite.   By the paragraph right after the proof 
of Theorem 3.3 of \cite{BR}, $B$ has a 2-quasitrace (see also III. {{1.3}} of \cite{BH}).
By Proposition \ref{Prr0}, $A$ has tracial approximate oscillation zero.  
Then the  corollary follows from Theorem {{\ref{T92}.}}
\end{proof}


Let $A$ be a separable simple \CA\, with ${\wtd{QT}}(A)\setminus \{0\}\not=\emptyset.$
Let $e\in {\rm Ped}(A)_+\setminus \{0\}.$
{{Recall that}} $T_e=\{\tau\in {\wtd{QT}}(A): \tau(e)=1\}$
is a  compact convex set 
which is also a basis for the cone $\wtd{QT}(A).$ 

\begin{cor}\label{CCcountable}
Let $A$ be a  {{$\sigma$-unital}} 
simple \CA\, 
with $QT(A)\not=\emptyset$ which has strict comparison.
Suppose that, for some $e\in {\rm Ped}(A)_+^{\bf 1}\setminus \{0\}$ 
and $\partial_e(T_e)$ 
has countably many points. 
Then $\Gamma$ is surjective, $A$ has stable rank one,
property (TM) and  
T-tracial approximate oscillation zero.

\end{cor}

\begin{proof}
It follows from Theorem \ref{Tcounableos} that $A$ has norm approximate oscillation zero.
Thus corollary follows from {{Theorem \ref{T92}}} immediately. 
\end{proof}

The following is perhaps known but we are not able to locate it in the literature.

\begin{prop}\label{Plocn}
Let $A$ be a separable \CA\, which has local finite nuclear dimension.
Then every hereditary \SCA\, $B\subset A$ also has local  finite nuclear dimension.
\end{prop}

\begin{proof}
Let $B$ be a hereditary \SCA\, of $A$. Let $\ep>0$ and ${\cal F}\subset B$ be a 
finite subset.  To simplify notation, \wilog, we may assume that  ${\cal F}\subset B^{\bf 1}$ and 
there is $e_B\in B_+^{\bf 1}$ 
such that $e_Bx=xe_B=x$ for all $x\in {\cal F}.$ 

Choose $\dt>0$ as in Lemma 3.3 of \cite{eglnp} associated with $\ep/4$ (in place 
of $\ep$) and $\sigma=\ep/4.$  We may assume that $\dt<\ep/4.$

Since $A$ has local finite nuclear dimension, there is a \SCA\, $C\subset A$ with 
finite nuclear dimension, say $k$ ($k\in \N{{\cup\{0\}}}$), such that
\beq
x\in_{\dt/2} C\rforal x\in {\cal F}\cup\{e_B\}.
\eneq
Choose $d\in C_+^{\bf 1}$ such that $\|e_B-d\|<\dt.$
Then, by Lemma 3.3 of \cite{eglnp}, there is a partial isometry $w\in A^{**}$ such 
that
\beq
&&ww^*f_{\ep/4}(d)=f_{\ep/4}(d)ww^*=f_{\ep/4}(d),\,\, w^*cw\in \Her(e_B)\subset B\andeqn\\
&&\|w^*cw-c\|<(\ep/4)\|c\|\rforal c\in  \overline{f_{\ep/4}(d)Af_{\ep/4}(d)}.
\eneq
Set $C_1=w^*\overline{f_{\ep/4}(d)Cf_{\ep/4}(d)}w\subset B.$ 
By  Proposition 2.5  of {{\cite{WZ10nuc}}}, $\overline{f_{\ep/4}(d)Cf_{\ep/4}(d)}$ has nuclear dimension $k.$ 
Since $C_1\cong \overline{f_{\ep/4}(d)Cf_{\ep/4}(d)},$ $C_1$ has nuclear dimension $k.$
We then estimate that
\beq
x\in_{\ep} C_1\rforal x\in {\cal F}.
\eneq
Thus $B$ has local finite nuclear dimension.
\end{proof}

As in \cite{Th}, we have the following
(note that, by \cite{Haagtrace}, since $A$ is exact, $\wtd{T}(A)=\wtd{QT}(A)$):

\begin{cor}\label{Clocalnuc}
Let $A$ be a separable exact simple \CA\, with $ {\wtd{T}}(A) {{\setminus \{0\}}}\not=\emptyset.$
Suppose that $A$ has strict comparison, T-tracial approximate oscillation zero
and has local finite nuclear dimension. Then $A\otimes{\cal Z}\cong A.$
\end{cor}

\begin{proof}
Choose $e\in {\rm Ped}(A)_+^{\bf 1}\setminus \{0\}$ and 
$B=\Her(e),$  Then ${\rm Ped}(B)=B.$  It suffices to show ({{see}} Cor. 3.1 of \cite{TWst}) 
that $B$ is ${\cal Z}$-stable. Note that $B$ has strict comparison, and by  Proposition \ref{Plocn}, has local finite nuclear dimension.  Since $B$ also has T-tracial approximate oscillation zero
(see \ref{Pher}), by Theorem \ref{TOstosurj-1}, $\Gamma$ is surjective.
It follows that $B$ has $m$-almost divisibility 
for some $m$ (in fact $m$  can be zero).
By 
{{\cite[Theorem 8.5 (iii)]{T-0-Z}}}, $B$ is ${\cal Z}$-stable.
\end{proof}

\begin{rem}\label{Rweakloc}
At least in the unital case, the condition that $A$ has local finite nuclear dimension
in Corollary \ref{Clocalnuc} can be further weakened 
to that $A$ is amenable  and has weak  tracial finite nuclear dimension (see Definition 8.1 and Theorem 8.3 of \cite{Linmem14}).
\end{rem}

\begin{rem}\label{Rlast}
{{(1) Note that,  in Theorem \ref{Teqiv} and Corollary \ref{CCcountable},
we do not assume that $A$ is amenable or even exact.}}

{{(2) Usually, the condition that $A$ has strict comparison implies that $A$ has at least one 
{{densely defined}} non-zero 
2-quasitrace. However, one may insist that the condition that $A$ has strict comparison 
means that, if $A$ has no  non-zero 2-quasitraces, $A$ is purely infinite.  In that case   the assumption 
in 
Theorem \ref{Teqiv} (part of  the assumption of 
Corollary \ref{Clocalnuc}) may be replaced by that $A$ is finite and has strict comparison.}}

{{(3) On the other hand, if one assumes that $\Cu(A)$ is almost unperforated  and $A$ is not purely infinite, 
then, by \cite{Ror04JS}, $A$ has strict comparison (in the usual sense) (see also Remark 2.5 and Proposition 4.9
of \cite{FL21-2}).  Conversely, if $A$ has strict comparison (in usual sense),
$\Cu(A)$ is almost unperforated. Therefore, if one prefers not to mention 
2-quasitraces in  {{Theorem \ref{Teqiv},}} one could use the condition that $A$ is finite and $\Cu(A)$ is almost unperforated.}}

{{(4) If $A$ is a unital stably finite simple \CA, then,  
 by \cite[Theorem 6.1]{Ror92UHF2} {{(see also}} 
\cite[Corollary 4.7]{Cuntzdim} and 
\cite[Theorem II.2.2]{BH}), 
$A$ has at least one non-trivial 2-quasitrace. So, in the unital case, we may assume that $A$ is stably finite instead 
assume that $A$ has a non-trivial 2-quasitrace. This also works for the case that $A$ is not unital 
but $K_0(A)_+\not=\{0\}.$ However, when $A$ is stably projectionless, the situation is somewhat different.
Nevertheless, we may proceed this as follows:}}

{{We assume that $A$ is a separable simple \CA. 
Recall that an element $a\in {\rm Ped}(A)_+$ is infinite, 
if there are nonzero elements ${{b,\ }}c\in {\rm Ped}(A)_+$ such that $ bc=0, $ $b+c\lesssim c$ and $c\lesssim a.$ 
$A$ is said to be finite, if there are no infinite elements in ${\rm Ped}(A)_+.$  $A$ is said to be 
stably finite, if $M_n(A)$ is finite for each $n$ (see Definition 1.1 of \cite{LZ} and Definition 4.7 of \cite{FL21-2},  for example).}} 

{{Choose $e\in {\rm Ped}(A)_+\setminus\{0\}$ and consider $B=\Her(e).$
\Wlog, we may well assume that $A=B$ for convenience. 
Define $K_0^*(A)$ (using $W(A)$ not $\Cu(A)$) exactly the same way as in section  4 of 
 \cite{Cuntzdim}. Note that Lemma 4.1  of \cite{Cuntzdim} holds automatically  with the definition above. 
 The same definition of order there (before Proposition 4.2 of \cite{Cuntzdim}) also works in this case.
  In other words, 
 so defined $K_0^*(A)$ is  a (directed) ordered group and the stably finiteness 
 ensures that $K_0^*(A)$ is not zero.  Since $A$ is simple, Proposition 4.2 of \cite{Cuntzdim}
 still holds.  We now return to the  paragraph right after the proof Theorem 3.3 of \cite{BR}. 
 Note that $(K_0^*(A), K_0^*(A)_+, [e])$ is a scaled ordered group which has a state, and which 
 gives a dimension function. By II. 2.2 of \cite{BH}, the dimension function just mentioned gives a 2-quasitrace on $A.$
 Therefore, we may replace the condition that ${\wtd{QT}}(A)\setminus \{0\}\not=\emptyset$ by 
 the condition that $A$ is stably finite (recall that we assume that $A$ is simple) in Theorem \ref{Teqiv} (see also \cite{K97}
 for the case that $A$ is exact}}).
 

{{(5) One may notice   that the condition that $A$ has strict comparison and $\Gamma$ is surjective
implies that there is an isomorphism $\Gamma^\sim: \Cu(A)\to V(A)\sqcup ({\rm LAff}_+(\wtd{QT}(A))\setminus \{0\})$ 
(see also \ref{DGamma}).}}

\end{rem}

\bigskip

\textsc{Xuanlong Fu}

Department of Mathematics, 
University of Toronto, Toronto, Ontario, M5S 2E4, Canada

E-mail: xfu@math.toronto.edu


\quad

\textsc{Huaxin Lin}
 
Department of Mathematics, 
East China Normal University, 
Shanghai, China

and (current)

Department of Mathematics, University of Oregon, Eugene, OR 97403, USA

E-mail: hlin@uoregon.edu

\clearpage

\end{document}